\documentclass[twoside,11pt]{article}

\usepackage{blindtext}

\usepackage{jmlr2e}
\usepackage{mathtools}

\usepackage{enumitem}
\newtheorem{assumption}{Assumption}
\usepackage{amsmath}
\usepackage{amssymb}
\usepackage{mathrsfs}
\usepackage{bm}
\usepackage{booktabs}
\usepackage{multirow}
\usepackage{bbding}
\usepackage{verbatim}
\usepackage{float}

\usepackage{microtype}
\usepackage{graphicx}
\usepackage{subfigure}
\usepackage{booktabs} 

\usepackage{color}
\usepackage{longtable}

\hypersetup{ hidelinks }

\usepackage{lastpage}
\jmlrheading{25}{2024}{1-\pageref{LastPage}}{3/23}{6/24}{23-0383}{Haobo Zhang, Yicheng Li and Qian Lin}
\ShortHeadings{On the Optimality of Misspecified Spectral Algorithms}{Zhang, Li and Lin}

\begin{document}
\raggedbottom  
\title{On the Optimality of Misspecified Spectral Algorithms}

\author{\name Haobo Zhang \email zhang-hb21@mails.tsinghua.edu.cn \\ 
\AND
\name Yicheng Li \thanks{Haobo Zhang and Yicheng Li contributed equally to this work.} \email liyc22@mails.tsinghua.edu.cn\\ 
\AND
\name Qian Lin \thanks{Corresponding author} \email qianlin@tsinghua.edu.cn \\ 
      \addr Center for Statistical Science, Department of Industrial Engineering\\
      Tsinghua University\\
      Beijing, China
      }

\editor{Kenji Fukumizu}

\maketitle

\begin{abstract}
In the misspecified spectral algorithms problem, researchers usually assume the underground true function $f_{\rho}^{*} \in [\mathcal{H}]^{s}$, a less-smooth interpolation space of a reproducing kernel Hilbert space (RKHS) $\mathcal{H}$ for some $s\in (0,1)$. The existing minimax optimal results require $\|f_{\rho}^{*}\|_{L^{\infty}}<\infty$ which implicitly requires  $s > \alpha_{0}$ where $\alpha_{0}\in (0,1)$ is the embedding index, a constant depending on $\mathcal{H}$. Whether the spectral algorithms are optimal for all $s\in (0,1)$ is an outstanding problem lasting for years. In this paper, we show that spectral algorithms are minimax optimal for any $\alpha_{0}-\frac{1}{\beta} < s < 1$, where $\beta$ is the eigenvalue decay rate of $\mathcal{H}$. We also give several classes of RKHSs whose embedding index satisfies $ \alpha_0 = \frac{1}{\beta} $. Thus, the spectral algorithms are minimax optimal for all $s\in (0,1)$ on these RKHSs.
\end{abstract}


\section{Introduction}
Suppose that the samples $\{ (x_{i}, y_{i}) \}_{i=1}^{n}$ are i.i.d. sampled from an unknown distribution $\rho$ on $\mathcal{X} \times \mathcal{Y}$, where $\mathcal{X} \subseteq \mathbb{R}^{d}$ and $\mathcal{Y} \subseteq \mathbb{R}$. One of the goals of non-parametric least-squares regression is to find a function $\hat{f}$ based on the $n$ samples such that the risk
\begin{equation}\label{def of risk}
  \mathcal{E}(\hat{f}) = \mathbb{E}_{(x,y) \sim \rho} \left[ \left( \hat{f}(x) - y \right)^{2} \right]
\end{equation}
is relatively small. It is well known that the conditional mean function given by $f_{\rho}^*(x) \coloneqq \mathbb{E}_{\rho}[~y \;|\; x~] = \int_{\mathcal{Y}} y \mathrm{d} \rho(y|x)$ minimizes the risk $\mathcal{E}(f) $. 
Therefore, 
we may focus on establishing the convergence rate (either in expectation or in probability) for the excess risk ($L^{2}$-norm generalization error)
\begin{equation}\label{def of gen}
    \mathbb{E}_{x \sim \mu} \left[ \left( \hat{f}(x) - f_{\rho}^{*}(x) \right)^{2} \right],
\end{equation}
where $ \mu $ is the marginal distribution of $\rho$ on $\mathcal{X}$. 

In the non-parametric regression settings, researchers often assume that $f_{\rho}^*(x)$ falls into a class of functions with a certain structure and develop non-parametric methods to obtain the estimator $\hat{f}$. One of the most popular non-parametric regression methods, the kernel method, aims to estimate $f_{\rho}^*$ using candidate functions from a reproducing kernel Hilbert space (RKHS) $\mathcal{H}$, a separable Hilbert space associated with a kernel function $k$ defined on $\mathcal{X}$, e.g., \citet{Kohler2001NonparametricRE, Cucker2001OnTM, Steinwart2008SupportVM}. This paper focuses on a class of kernel methods called the \textit{spectral algorithms} to construct the estimator of $ f_{\rho}^{*}$.

Since the minimax optimality of spectral algorithms has been proved for the attainable case $\left( f_{\rho}^{*} \in \mathcal{H} \right)$ \citep[etc.]{caponnetto2006optimal,Caponnetto2007OptimalRF}, a large body of literature has studied the convergence rate of the generalization error of misspecified spectral algorithms ($ f_{\rho}^{*} \notin \mathcal{H}$) and whether the rate is optimal in the minimax sense. It turns out that the qualification of the algorithm ($\tau >0$), the eigenvalue decay rate ($\beta > 1$), the source condition ($s>0$) and the embedding index ($\alpha_{0} < 1$) of the RKHS jointly determine the convergence behaviors of the spectral algorithms (see Section \ref{section assumption} for definitions). If we only assume that $f_{\rho}^{*}$ belongs to an interpolation space $[\mathcal{H}]^{s}$ of the RKHS $\mathcal{H}$ for some $s > 0$, the well known information-theoretic lower bound shows that the minimax lower bound (with respect to the $L^{2}$-norm generalization error) is $n^{-\frac{s \beta}{s \beta + 1}} $. The state-of-the-art result shows that when $ \alpha_{0} < s \le 2 \tau$, the upper bound of the convergence rate (with respect to the $L^{2}$-norm generalization error) is $ n^{-\frac{s \beta}{s \beta + 1}}$ and hence is optimal (\citealt{fischer2020_SobolevNorm} for kernel ridge regression and \citealt{PillaudVivien2018StatisticalOO} for gradient methods). However, when $ f_{\rho}^{*} \in [\mathcal{H}]^{s}$ for some $0 < s \le \alpha_{0}$, all the existing works need an additional boundedness assumption of $f_{\rho}^{*}$ to prove the same upper bound $ n^{-\frac{s \beta}{s \beta + 1}}$. The boundedness assumption will result in a smaller function space, i.e., $ [\mathcal{H}]^{s} \cap L^{\infty}(\mathcal{X,\mu}) \subsetneqq [\mathcal{H}]^{s}$ when $ s \le \alpha_{0}$. \citet{fischer2020_SobolevNorm} further reveals that the minimax rate associated with the smaller function space is larger than  $ n^{-\frac{\alpha \beta}{ \alpha \beta + 1}}$ for any $\alpha > \alpha_{0}$. This minimax lower bound is smaller than the upper bound of the convergence rate and hence they can not prove the minimax optimality of spectral algorithms when $s \le \alpha_{0}$. 

It has been an outstanding problem for years whether the spectral algorithms are minimax optimal for all $s\in (0,1)$, either with respect to the $L^{2}$-norm or the $[\mathcal{H}]^{\gamma}$-norm introduced later~\citep{PillaudVivien2018StatisticalOO, fischer2020_SobolevNorm, Liu2022StatisticalOO}.
To this end, this paper has three contributions.
\begin{itemize}[leftmargin = 18pt]
    \item Using the tools from real interpolation theory, we analyze the $L^{q}$-embedding property of $[\mathcal{H}]^{s}$, an interpolation space of the RKHS. Specifically, assume that $\mathcal{H} $ has embedding index $\alpha_{0}$. When $s \le \alpha_{0}, $ Theorem \ref{integrability of Hs} proves that $[\mathcal{H}]^{s}$ is continuously embedded into $ L^{q}(\mathcal{X},\mu)$, for $ q = \frac{2 \alpha}{\alpha - s}, \forall \alpha > \alpha_{0}$. 
    
    \item Based on the $ L^{q}$-embedding property of $[\mathcal{H}]^{s}$, the refined proof in this paper removes the boundedness assumption in previous literature and obtains the same upper bound of the convergence rate as the state-of-the-art upper bound. As a result, we prove the minimax optimality of spectral algorithms for $\alpha_0 - \frac{1}{\beta} < s \le 2 \tau$, which can only be proved for $\alpha_0 < s \le 2 \tau$ before. We also recover the upper bound in previous literature when $0 < s \le \alpha_0 - \frac{1}{\beta} $ (if exists) though the optimality does not hold. Note that in this paper, we present the results in terms of $[\mathcal{H}]^{\gamma}$-norm generalization error, where the $L^{2}$-norm \eqref{def of gen} is a special case when $ \gamma = 0$.
    
    \item We give several examples of RKHSs whose embedding index satisfies $ \alpha_{0} = \frac{1}{\beta}$. Besides RKHS with uniformly bounded eigenfunctions and the Sobolev RKHS \citep{fischer2020_SobolevNorm}, we first show that RKHS with shift-invariant kernels and RKHS with dot-product kernels on the sphere satisfy that $ \alpha_{0} = \frac{1}{\beta}$. Therefore, for these RKHSs, this paper proves the optimality of spectral algorithms for all $ 0 < s \le 2 \tau$.
\end{itemize}

The outline of the rest of the paper is as follows. In Section \ref{section pre}, we introduce basic concepts including priori knowledge of RKHS, integral operators and the definition of the interpolation space. In addition, we formally define the spectral algorithm, which is the main interest of this paper, and provide three examples of common spectral algorithms. In Section \ref{section main}, we present our main results of the convergence rates and discuss the minimax optimality. Theorem \ref{main theorem} and Theorem \ref{prop information lower bound} show the upper bound and the minimax lower bound, respectively. In Section \ref{section examples}, we further show four kinds of commonly used RKHSs with embedding index $\alpha_{0} = \frac{1}{\beta}$. This is the ideal case where the minimax optimality can be proved for all $ 0 < s \le 2 \tau$. We verify our results through experiments in Section \ref{section experiments}. In Section \ref{section discussion}, we make a comparison with previous literature and discuss other applications of our techniques. All the proofs can be found in Section \ref{section proofs}. In the appendix, we provide supplementary materials including extended proof, details of the experiments and a table of important notations frequently used throughout the main text.


\subsection{Related work}

General spectral algorithms in the setting of kernel methods were first proposed and studied by \citet{ rosasco2005_SpectralMethods,caponnetto2006optimal, bauer2007_RegularizationAlgorithms, gerfo2008_SpectralAlgorithms}. A large class of regularization methods are introduced collectively as spectral algorithms and are characterized through the corresponding filter functions. The qualification $\tau$ of a spectral algorithm and a prior assumption on $f_{\rho}^{*}$ characterizing the relative smoothness (source condition $s$) are also introduced for the problem setting. In this setting, \citet{bauer2007_RegularizationAlgorithms} proves the upper bound of the convergence rate with respect to the $L^{2}$-norm generalization error. \citet{caponnetto2006optimal} proves the `capacity-dependent' upper bound, i.e., considering the eigenvalue decay rate $\beta$ of the RKHS, which has been adopted by most of the later researchers. Note that these works focus on the well specified case ($f_{\rho}^{*} \in \mathcal{H}$) or assume that $\mathcal{H}$ is dense in $L^{2}(\mathcal{X},\mu)$. There are also other related works studying the well specified case, e.g., \citet{blanchard2018_OptimalRates, dicker2017_KernelRidge, rastogi2017_OptimalRates} for general spectral algorithms, \citet{Caponnetto2007OptimalRF, Smale2007LearningTE} for kernel ridge regression and \citet{Yao2007OnES} for gradient methods.

Since the convergence rates and the minimax optimality of spectral algorithms in the well specified case are clear, a large amount of literature studied the misspecified spectral algorithms. Among these work, \cite{steinwart2009_OptimalRates,dicker2017_KernelRidge,PillaudVivien2018StatisticalOO,fischer2020_SobolevNorm, Celisse2020AnalyzingTD,Li2022OptimalRF,  Talwai2022OptimalLR} consider the $L^{\infty}$-embedding property, while \citet{dieuleveut2016nonparametric,lin2018_OptimalRates,lin2020_OptimalConvergence,JMLR:v23:21-0570} do not. Note that considering the $L^{\infty}$-embedding property is equivalent to introducing the embedding index $\alpha_{0}$ in this paper. It has been shown that this will lead to faster convergence rates for certain embedding indexes (see Section \ref{section discussion} for detailed comparison). In addition, as mentioned in \cite{fischer2020_SobolevNorm}, the convergence rates with respect to the $L^{2}$-norm can be easily extended to the more general $ [\mathcal{H}]^{\gamma} $-norm if one uses the integral operator technique. Up to now, we have introduced five indexes $ \tau, s, \beta, \alpha_{0}$ and $ \gamma$ that we know as a priori to study the convergence rates of the spectral algorithms. To our knowledge, the state-of-the-art results on the convergence rates and the minimax optimality are \citet{fischer2020_SobolevNorm} for kernel ridge regression and \citet{PillaudVivien2018StatisticalOO} for gradient methods. 

But the spectral algorithms in the misspecified case have not been totally solved. When $f_{\rho}^{*}$ falls into a less-smooth interpolation space which does not imply the boundedness of functions therein, all existing works (either considering embedding index or not) require an additional boundedness assumption, i.e., $ \| f_{\rho}^{*} \|_{L^{\infty}(\mathcal{X},\mu)} \le B_{\infty} < \infty$ to prove the desired upper bound. As discussed in the introduction, this will lead to the suboptimality in the $s \le \alpha_{0}$ regime. As far as we know, the $L^{q}$-embedding property of $[\mathcal{H}]^{s}$ has not been discussed in related literature. This paper shows that it turns out to be a crucial property to remove the boundedness assumption and extend the minimax optimality to a broader regime.

\section{Preliminaries}\label{section pre}

\subsection{Basic concepts}\label{section basic concept}
Let a compact set $\mathcal{X} \subseteq \mathbb{R}^{d}$ be the input space and $ \mathcal{Y} \subseteq \mathbb{R}$ be the output space. Let $ \rho $ be an unknown probability distribution on $\mathcal{X} \times \mathcal{Y}$ satisfying $ \int_{\mathcal{X} \times \mathcal{Y}} y^{2} \mathrm{d}\rho(x,y) <\infty$ and denote the corresponding marginal distribution on $ \mathcal{X} $ as $\mu$. We use $L^{p}(\mathcal{X},\mu)$ (in short $L^{p}$) to represent the $L^{p}$-spaces. Denote 
\begin{displaymath}
  f_{\rho}^*(x) \coloneqq \mathbb{E}_{\rho}[~y \;|\; x~] = \int_{\mathcal{Y}} y ~ \mathrm{d} \rho(y|x)
\end{displaymath}
as the conditional mean. Throughout the paper, we denote $\mathcal{H}$ as a separable RKHS on $\mathcal{X}$ with respect to a continuous kernel function $k$ and satisfying
\begin{displaymath}
  \sup\limits_{x \in \mathcal{X}} k(x,x) \le \kappa^{2}.
\end{displaymath}

Denote the natural embedding inclusion operator as $S_{k}: \mathcal{H} \to L^{2}(\mathcal{X},\mu)$. Moreover, the adjoint operator $S_{k}^{*}: L^{2}(\mathcal{X},\mu) \to \mathcal{H}  $ is an integral operator, i.e., for $f \in L^{2}(\mathcal{X},\mu)$ and $x \in \mathcal{X}$, we have 
\begin{displaymath}
\left(S_{k}^{*} f\right)(x)=\int_\mathcal{X} k\left(x, x^{\prime}\right) f\left(x^{\prime}\right) \mathrm{d} \mu\left(x^{\prime}\right).
\end{displaymath}
Then $S_{k}$ and $S_{k}^{*}$ are Hilbert-Schmidt operators (thus compact) and the HS norms (denoted as $\left\| \cdot \right\|_{2}$) satisfy that
    \begin{displaymath}
        \left\| S_{k}^{*} \right\|_{2} = \left\| S_{k} \right\|_{2} = \|k\|_{L^{2}(\mathcal{X},\mu)}:=\left(\int_\mathcal{X} k(x, x) \mathrm{d} \mu(x)\right)^{1 / 2} \le \kappa.
    \end{displaymath}
Next, we can define two integral operators: 
\begin{equation}
    L_{k}:=S_{k} S_{k}^{*}: L^{2}(\mathcal{X},\mu) \rightarrow L^{2}(\mathcal{X},\mu), \quad  \quad T:=S_{k}^{*} S_{k}: \mathcal{H} \rightarrow \mathcal{H}.
\end{equation}
$L_{k}$ and $T$ are self-adjoint, positive-definite and trace class (thus Hilbert-Schmidt and compact) and the trace norms (denoted as $\left\| \cdot \right\|_{1}$) satisfy that
    \begin{displaymath}
      \left\|L_{k}\right\|_{1}=\left\|T\right\|_{1}=\left\|S_{k}\right\|_{2}^2=\left\|S_{k}^{*}\right\|_{2}^2.
    \end{displaymath}
The spectral theorem for self-adjoint compact operators yields that there is an at most countable index set $N$, a non-increasing summable sequence $\{ \lambda_{i} \}_{i \in N} \subseteq (0,\infty)$ and a family $\{ e_{i} \}_{i \in N} \subseteq \mathcal{H} $, such that $\{ e_{i} \}_{i \in N}$ is an orthonormal basis (ONB) of $\overline{\operatorname{ran} S_{k}} \subseteq L^{2}(\mathcal{X},\mu)$ and $\{ \lambda_{i}^{1/2} e_{i} \}_{i \in N}$ is an ONB of $\mathcal{H}$. Further, the integral operators can be written as 
    \begin{equation}\label{spectral representation}
       L_{k}=\sum_{i \in N} \lambda_i\left\langle\cdot,e_{i} \right\rangle_{L^{2}}e_{i} \quad \text { and } \quad T=\sum_{i \in N} \lambda_i\left\langle\cdot, \lambda_i^{1/2} e_i\right\rangle_{\mathcal{H}} \lambda_i^{1/2} e_i.
    \end{equation}
    We refer to $\{ e_{i} \}_{i \in N}$ and $\{ \lambda_{i} \}_{i \in N}$ as the eigenfunctions and eigenvalues. We will also call them the eigenfunctions and eigenvalues of the RKHS $\mathcal{H}$ and the kernel function $k$. (Note that $\{ e_{i} \}_{i \in N}$ and $\{ \lambda_{i} \}_{i \in N}$ also depend on the marginal distribution $\mu$ on $\mathcal{X}$.)  The celebrated Mercer's theorem (see, e.g., \citealt[Theorem 4.49]{Steinwart2008SupportVM}) shows that 
    \begin{displaymath}\label{Mercer decomposition}
      k\left(x, x^{\prime}\right)=\sum_{i \in N} \lambda_i e_i(x) e_i\left(x^{\prime}\right), \quad x, x^{\prime} \in \mathcal{X},
    \end{displaymath}
    where the convergence is absolute and uniform.

We also need to introduce the interpolation spaces (power spaces) of RKHS. For any $ s \ge 0$, the fractional power integral operator $L_{k}^{s}: L^{2}(\mathcal{X},\mu) \to L^{2}(\mathcal{X},\mu)$ is defined as 
\begin{displaymath}
  L_{k}^{s}(f)=\sum_{i \in N} \lambda_i^{s} \left\langle f, e_i\right\rangle_{L^2} e_i.
\end{displaymath}
Then the interpolation space (power space) $[\mathcal{H}]^s $ is defined as
\begin{equation}\label{def interpolation space}
  [\mathcal{H}]^s := \operatorname{Ran} L_{k}^{s/2} = \left\{\sum_{i \in N} a_i \lambda_i^{s / 2}e_{i}: \left(a_i\right)_{i \in N} \in \ell_2(N)\right\} \subseteq L^{2}(\mathcal{X},\mu),
\end{equation}
equipped with the inner product
\begin{equation}\label{def of interpolation norm}
    \langle f, g\rangle_{[\mathcal{H}]^s}=\left\langle L_{k}^{-\frac{s}{2}} f, L_{k}^{-\frac{s}{2}} g\right\rangle_{L^2} .
\end{equation}
It is easy to show that $[\mathcal{H}]^s $ is also a separable Hilbert space with orthogonal basis $ \{ \lambda_{i}^{s/2} e_{i}\}_{i \in N}$. Specially, we have $[\mathcal{H}]^0 \subseteq L^{2}(\mathcal{X},\mu) $ and $[\mathcal{H}]^1 = \mathcal{H}$. For $0 < s_{1} < s_{2}$, the embeddings $ [\mathcal{H}]^{s_{2}} \hookrightarrow[\mathcal{H}]^{s_{1}} \hookrightarrow[\mathcal{H}]^0 $ exist and are compact \citep{fischer2020_SobolevNorm}. For the functions in $[\mathcal{H}]^{s}$ with larger $s$, we say they have higher regularity (smoothness) with respect to the RKHS.


It is worth pointing out the relation between the definition \eqref{def interpolation space} and the interpolation space defined through the real method (real interpolation). For details of interpolation of Banach spaces through the real method, we refer to \citet[Chapter 4.2.2]{sawano2018theory}. Specifically, \citet[Theorem 4.6]{steinwart2012_MercerTheorem} reveals that for $0 < s < 1$,
\begin{equation}\label{inter relation}
  [\mathcal{H}]^{s} \cong \left(L^{2}(\mathcal{X},\mu), [\mathcal{H}]^{1} \right)_{s,2},
\end{equation}
where $(\cdot, \cdot)_{s,2}$ denotes the real interpolation of two normed spaces (please refer to Definition \ref{def real interpolation}). As an example, the Sobolev space $H^{m}(\mathcal{X})$ is an RKHS if $m > \frac{d}{2}$ and its interpolation space is still a Sobolev space given by $ [H^{m}(\mathcal{X})]^s \cong H^{m s}(\mathcal{X}), \forall s>0 $. See Section \ref{section examples sobolev} for detailed discussions.

\subsection{Spectral algorithms}\label{section spectral algorithms}

Suppose that we observed the given samples $Z = \{ (x_{i}, y_{i}) \}_{i=1}^{n}$ and denote $X=\left( x_{1},\cdots,x_{n} \right)$. Define the sampling operator $ K_{x}: \mathbb{R} \rightarrow \mathcal{H}, ~ y \mapsto y k(x,\cdot) $ and its adjoint operator $K_{x}^{*}: \mathcal{H} \rightarrow \mathbb{R},~ f \mapsto f(x)$. Then we can define $T_{x} = K_{x} K_{x}^{*}$. Further, we define the sample covariance operator $T_{X}: \mathcal{H} \to \mathcal{H}$ as
\begin{equation}\label{def of TX}
    T_X:=\frac{1}{n} \sum_{i=1}^n K_{x_i} K_{x_i}^*.
\end{equation}
Then we know that $ \| T_{X} \| \le \left\| T_{X} \right\|_{1} \le \kappa^{2}$, where $\| \cdot \| $ denotes the operator norm and $\| \cdot \|_{1} $ denotes the trace norm. Further, define the sample basis function
\begin{displaymath}
  g_Z:=\frac{1}{n} \sum_{i=1}^n K_{x_i} y_i \in \mathcal{H}.
\end{displaymath}
Based on the $n$ samples, the kernel method aims to choose a function $ \hat{f} \in \mathcal{H}$ such that the risk given by \eqref{def of risk} is small. A direct estimator is $ \hat{f} \in \mathcal{H}$ that minimizing the empirical risk
\begin{displaymath}
  \hat{\mathcal{E}}(f) = \frac{1}{n}\sum\limits_{i=1}^{n} \left( f(x_i) - y_{i}  \right)^{2},
\end{displaymath}
which leads to an equation
\begin{displaymath}
  T_{X} \hat{f} = g_{Z}.
\end{displaymath}
However, on the one hand, minimizing the empirical risk may lead to overfitting. On the other hand, the inverse of the sample covariance operator $T_X$ does not exist in general. The spectral algorithms \citep[etc.]{rosasco2005_SpectralMethods,caponnetto2006optimal,bauer2007_RegularizationAlgorithms,  gerfo2008_SpectralAlgorithms} handle these issues by introducing the regularization and generate estimators through the filter functions. Now, we first define the filter function. 
\begin{definition}[Filter function]\label{def filter}
  Let $\left\{\varphi_\nu:\left[0, \kappa^2\right] \rightarrow \mathbb{R}^{+} \mid \nu \in \Gamma \subseteq \mathbb{R}^{+}\right\}$ be a class of functions and $ \psi_\nu(z)=1-z \varphi_\nu(z) $. If $\varphi_{\nu} ~\text{and}~ \psi_{\nu}$ satisfy:
  \begin{itemize}
      \item $\forall \alpha \in [0,1],$ we have
      \begin{equation}\label{prop1}
          \sup _{z \in\left[0, \kappa^2\right]} z^\alpha \varphi_\nu(z) \leq E \nu^{1-\alpha}, \quad \forall \nu \in \Gamma;
      \end{equation}
      \item $\exists \tau \ge 1 $ s.t. $\forall \alpha \in [0,\tau] $, we have
      \begin{equation}\label{prop2}
          \sup _{z \in\left[0, \kappa^2\right]}\left|\psi_\nu(z)\right| z^\alpha \leq F_\tau \nu^{-\alpha}, \quad \forall \nu \in \Gamma ,
      \end{equation}
  \end{itemize}
  where $ E, F_{\tau}$ are absolute constants, then we call $ \varphi_{\nu} $ a filter function. We refer to $\nu$ as the regularization parameter and $\tau$ as the qualification. 
\end{definition}

Given a filter function $ \varphi_{\nu} $, we can define the corresponding spectral algorithm \footnote{Let $L$ be a self-adjoint, compact operator over a separable Hilbert space $H$ with eigenvalues $\{\sigma_{i}\}_{i=1}^{\infty}$ and eigenfunctions (also an orthonormal basis of $H$) $\{\psi_{i}\}_{i=1}^{\infty}$. For a function $f: \mathbb{R} \to \mathbb{R}$, $f(L)$ is an operator defined by spectral calculus: $ f(L) = \sum_{i=1}^{\infty} f(\sigma_{i}) \psi_{i} \otimes \psi_{i}$.}
\begin{definition}[spectral algorithm]
   Let $\varphi_{\nu}$ be a filter function index with $ \nu >0$. Given the samples  $Z$, the spectral algorithm produces an estimator of $f_{\rho}^{*}$  given by  
   \begin{align}\label{SA estimator}
      \hat{f}_{\nu} = \varphi_{\nu}(T_{X}) g_{Z}.
   \end{align}
\end{definition}

Here we list three kinds of spectral algorithms that are commonly used. 
\begin{example}[Kernel ridge regression] Let the filter function $\varphi_{\nu} $ be defined as
\begin{displaymath}
    \varphi_{\nu}^{\mathrm{krr}}(z) = \frac{\nu}{\nu z + 1}.  
\end{displaymath}
Then the corresponding spectral algorithm is kernel ridge regression (Tikhonov regularization). The qualification $\tau = 1$ and $ E = F_{\tau} = 1$.
\end{example}

\begin{example}[Gradient flow] Let the filter function $\varphi_{\nu} $ be defined as
\begin{displaymath}
    \varphi_{\nu}^{\mathrm{gf}}(z) = \frac{1 - e^{-\nu z}}{z}.
\end{displaymath}
Then the corresponding spectral algorithm is gradient flow. The qualification $\tau$ could be any positive number, $ E = 1$ and $ F_{\tau} = \left( \tau / e \right)^{\tau}$.
\end{example}

\begin{example}[Spectral cut-off] Let the filter function $\varphi_{\nu} $ be defined as
\begin{displaymath}
    \varphi_\nu^{\mathrm{cut}}(z) = \begin{cases}z^{-1}, & z^{-1} \leq \nu, \\ 0, & z^{-1}>\nu. \end{cases}
\end{displaymath}
Then the corresponding spectral algorithm is Spectral cut-off (truncated singular value decomposition). The qualification $\tau$ could be any positive number and $ E = F_{\tau} = 1$.
\end{example}
For other examples of spectral algorithms (e.g., iterated Tikhonov, gradient methods, Landweber iteration, etc.), we refer to \cite{gerfo2008_SpectralAlgorithms}.

\section{Main results}\label{section main}

\subsection{Assumptions}\label{section assumption}

This subsection lists the standard assumptions that frequently appear in related literature.
\begin{assumption}[Eigenvalue decay rate (EDR)]\label{ass EDR}
 Suppose that the eigenvalue decay rate (EDR) of $\mathcal{H}$ is $\beta > 1$, i.e, there are positive constants $c$ and $C$ such that 
 \begin{displaymath}
   c i^{- \beta} \le \lambda_{i} \le C i^{-\beta}, \quad  \forall i \in N.
 \end{displaymath}
\end{assumption}
Note that the eigenvalues $\lambda_{i}$ and EDR are only determined by the marginal distribution $\mu$ and the RKHS $\mathcal{H}$. The polynomial eigenvalue decay rate assumption is standard in related literature and is also referred to as the capacity condition or effective dimension
condition \citep[etc.]{caponnetto2006optimal,Caponnetto2007OptimalRF}.
\newline

We say that $\mathcal{H}$ has the embedding property of order $\alpha\in [\frac{1}{\beta},1]$, if there is a constant $0 < A < \infty$ such that
\begin{equation}\label{embedding property 1.12}
    \left\|[\mathcal{H}]^\alpha \hookrightarrow L^{\infty}(\mathcal{X},\mu)\right\| \leq A,
\end{equation}
where $\|\cdot\| $ denotes the operator norm of the embedding.

In fact, for any $\alpha > 0$, we can define $M_{\alpha} $ as the smallest constant $A > 0$ such that 
\begin{equation}
    \label{eq:EMB_Eigenvalues}
  \sum_{i \in N} \lambda_i^\alpha e_i^2(x) \leq A^2, \quad \mu \text {-a.e. } x \in \mathcal{X},
\end{equation}
if there is no such constant, set $ M_{\alpha} = \infty$. Then \citet[Theorem 9]{fischer2020_SobolevNorm} shows that for $ \alpha > 0$,
\begin{displaymath}
  \left\|[\mathcal{H}]^\alpha \hookrightarrow L^{\infty}(\mathcal{X},\mu)\right\|=M_{\alpha}.
\end{displaymath}

Note that since $ \sup_{x \in \mathcal{X}} k(x,x) \le \kappa^{2} $, $M_{\alpha} \le \kappa < \infty$ is always true for $\alpha \ge 1$. In addition, \citet[Lemma 10]{fischer2020_SobolevNorm} also shows that $\alpha$ can not be less than $\frac{1}{\beta}$. By the inclusion relation of interpolation spaces, it is clear that if $\mathcal{H}$ has the embedding property of order $\alpha$, then it has the embedding property of order $\alpha^{\prime}$ for any $\alpha^{\prime} \geq \alpha$. Thus, we may introduce the following assumption: 

\begin{assumption}[Embedding index]\label{assumption embedding}
Suppose that there exists $\alpha_{0} > 0$, such that
\begin{displaymath}
  \alpha_{0} = \inf\left\{ \alpha \in [\frac{1}{\beta},1] :  \left\|[\mathcal{H}]^\alpha \hookrightarrow L^{\infty}(\mathcal{X},\mu)\right\| < \infty  \right\},
\end{displaymath}
and we refer to $\alpha_{0}$ as the \textit{embedding index} of an RKHS $\mathcal{H}$.
\end{assumption}

 Note that $ \mathcal{H}$ has the embedding property of order $\alpha$ for any $\alpha > \alpha_{0}$. This directly implies that all the functions in $[\mathcal{H}]^\alpha$ are $\mu \text {-a.e.}$ bounded, $\alpha > \alpha_{0}$. However, the embedding property may not hold for $\alpha = \alpha_{0}$.
 
\begin{assumption}[Source condition]\label{ass source condition}
  For $s > 0 $, there is a constant $R > 0 $ such that $f_{\rho}^{*} \in [\mathcal{H}]^{s}$ and
  \begin{displaymath}
    \| f_{\rho}^{*} \|_{[\mathcal{H}]^{s}} \le R.
  \end{displaymath}
\end{assumption}
Functions in $[\mathcal{H}]^{s}$ with smaller $s$ are less smooth, which will be harder for an algorithm to estimate.
\begin{assumption}[Moment of error]\label{ass mom of error}
  The noise $ \epsilon \coloneqq y - f_{\rho}^{*}(x)$ satisfies that there are constants $ \sigma, L > 0$ such that for any $ m \ge 2$,
  \begin{displaymath}
      \mathbb{E}\left(|\epsilon|^m \mid x\right) \leq \frac{1}{2} m ! \sigma^2 L^{m-2}, \quad \mu \text {-a.e. } x \in \mathcal{X}.
  \end{displaymath}
\end{assumption}
This is a standard assumption to control the noise such that the tail probability decays fast \citep{lin2020_OptimalConvergence,fischer2020_SobolevNorm}. It is satisfied for, for instance, the Gaussian noise with bounded variance or sub-Gaussian noise. Some literature (e.g., \citealt{steinwart2009_OptimalRates,PillaudVivien2018StatisticalOO,jun2019kernel}, etc) also uses a stronger assumption $y \in [-L_{0},L_{0}]$ which implies both Assumption \ref{ass mom of error} and the boundedness of $f_{\rho}^{*}$.

\subsection{Convergence results}
Now we are ready to state our main results. Though this paper focuses on the misspecified case, i.e., $ 0 < s < 1$, we state the theorems including those $ s \ge 1$ for completeness.
\setcounter{theorem}{0}
\begin{theorem}[Upper bound]\label{main theorem}
  Suppose that Assumption \ref{ass EDR},\ref{assumption embedding}, \ref{ass source condition} and \ref{ass mom of error} hold for $ 0 < s \le 2 \tau$ and $\frac{1}{\beta} \le \alpha_{0} < 1$. Let $\hat{f}_{\nu}$ be the estimator defined by \eqref{SA estimator}. Then for $0 \le \gamma \le 1$ with $\gamma \le s$:
  \begin{itemize}[leftmargin = 18pt]
      \item In the case of $s + \frac{1}{\beta} > \alpha_{0} $, by choosing $\nu \asymp n^{\frac{\beta }{s \beta + 1}}$, for any fixed $\delta \in (0,1)$, there exists a constant $N$ such that when $n \ge N$, with probability at least $1 - \delta$, we have
      \begin{equation}
          \left\|\hat{f}_{\nu}-f_{\rho}^*\right\|_{[\mathcal{H}]^{\gamma}}^2 \leq\left(\ln \frac{6}{\delta}\right)^2 C n^{-\frac{(s-\gamma) \beta}{s \beta+1}},
      \end{equation}
      where $C$ is a constant independent of $n$ and $\delta$. The constant $N$ only depends on the parameters and constants from Assumption \ref{ass EDR}, \ref{assumption embedding}, \ref{ass source condition} and \ref{ass mom of error}, on $\delta$, on the operator norm $\| T \|$ and on the constants in the scaling of $\nu $ with respect to $n$.
      
      \item In the case of $s + \frac{1}{\beta} \le \alpha_{0} $, for any $ \alpha > \alpha_{0}$, by choosing $\nu \asymp (\frac{n}{\ln^{r}(n)})^{\frac{1}{\alpha}}$ for some $r > 1$, for any fixed $\delta \in (0,1)$, there exists a constant $N$ such that when $n \ge N$, with probability at least $1 - \delta$, we have
      \begin{equation}
          \left\|\hat{f}_\nu-f_{\rho}^*\right\|_{[\mathcal{H}]^{\gamma}}^2 \leq\left(\ln \frac{6}{\delta}\right)^2 C \left(\frac{n}{\ln ^r(n)}\right)^{-\frac{s-\gamma}{\alpha}},
      \end{equation}
      where $C$ is a constant independent of $n$ and $\delta$. The constant $N$ only depends on the parameters and constants from Assumption \ref{ass EDR}, \ref{assumption embedding}, \ref{ass source condition} and \ref{ass mom of error}, on $\alpha$, on $\delta$, on the operator norm $\| T \|$ and on the constants in the scaling of $\nu $ with respect to $n$.
  \end{itemize}
\end{theorem}
Compared with the state-of-the-art results (\citealt{fischer2020_SobolevNorm,PillaudVivien2018StatisticalOO}), Theorem \ref{main theorem} removes the boundedness assumption $ \| f_{\rho}^{*} \|_{L^{\infty}} \le B_{\infty} < \infty$ and prove the same upper bound for general spectral algorithms. This improvement is nontrivial for $ s \le \alpha_{0}$, since $ [\mathcal{H}]^{s} \cap L^{\infty}(\mathcal{X,\mu}) \subsetneqq [\mathcal{H}]^{s}$ when $ s \le \alpha_{0}$. As we will see in Section \ref{section proofs}, the proof of Theorem \ref{main theorem} removes the boundedness assumption by analyzing the $L^{q}$-embedding property of $[\mathcal{H}]^{s} $. With the $L^{q}$-integrability of the functions in $ [\mathcal{H}]^{s} $, although the true function $f_{\rho}^{*}$ may not fall into $ L^{\infty}(\mathcal{X,\mu})$, the tail probability can be controlled appropriately. We present the convergence results for $ [\mathcal{H}]^{\gamma}$-norm generalization error, where the $L^{2}$-norm \eqref{def of gen} is a special case when $ \gamma = 0$. 

Now we are going to state the minimax lower bound, which is often referred to as the information-theoretic lower bound. 
\begin{theorem}[Lower bound]\label{prop information lower bound}
    Let $\mu$ be a probability distribution on $\mathcal{X}$ such that Assumption \ref{ass EDR} is satisfied. Let $\mathcal{P}$ consist of all the distributions on $\mathcal{X}\times \mathcal{Y}$ satisfying \ref{ass source condition}, \ref{ass mom of error} for $ s > 0$ and with marginal distribution $\mu$. Then for $0 \le \gamma \le 1$ with $\gamma \le s$, there exists a constant $C$, for all learning methods, for any fixed $\delta \in (0,1)$, when $n$ is sufficiently large, there is a distribution $\rho \in \mathcal{P}$ such that, with probability at least $1 - \delta$, we have 
   \begin{equation}
       \left\|\hat{f}-f_\rho^*\right\|_{[\mathcal{H}]^{\gamma}}^2 \ge C \delta n^{-\frac{(s-\gamma)\beta}{s \beta +1}}.
   \end{equation}
\end{theorem}

The main difference between our Theorem \ref{prop information lower bound} and the lower rate in \citealt{fischer2020_SobolevNorm} is that we consider the whole space $[\mathcal{H}]^{s}$ instead of $[\mathcal{H}]^{s} \cap L^{\infty} $, thus the rate will be different with theirs when $ 0 < s < \alpha_{0}$.

\begin{remark}[Optimality]
A direct result from Theorem \ref{main theorem} and Theorem \ref{prop information lower bound} is that for $s \in \left(\alpha_0 - \frac{1}{\beta}, 2 \tau \right]$, the upper bound matches the minimax lower bound. Therefore, we prove the minimax optimality of spectral algorithms for $ s \in \left( \alpha_0 - \frac{1}{\beta}, 2 \tau \right]$ with respect to $ [\mathcal{H}]^{\gamma}$-norm ( $ 0 \le \gamma \le s$) generalization error.
\end{remark}

\section{Examples: RKHS with embedding index \texorpdfstring{$\alpha_{0} = \frac{1}{\beta}$ }{2} }\label{section examples}
We prove the minimax optimality of spectral algorithms for $\alpha_0 - \frac{1}{\beta} < s \le 2 \tau$ in the last section. Therefore the embedding index $\alpha_{0}$ of an RKHS is crucial when analyzing the optimality of the spectral algorithms. In the best case of $\alpha_{0} = \frac{1}{\beta}$, only the first situation in Theorem \ref{main theorem} exists and we obtain the optimality for all $ 0 < s \le 2 \tau$. In this section, we give several examples of RKHSs with embedding index $\alpha_{0} = \frac{1}{\beta}$.

\subsection{RKHS with uniformly bounded eigenfunctions}\label{section examples ui}
RKHS with uniformly bounded eigenfunctions, i.e., $\sup_{i \in N}  \| e_{i} \|_{L^{\infty}} < \infty$, are frequently considered \citep{mendelson2010_RegularizationKernel, steinwart2009_OptimalRates, PillaudVivien2018StatisticalOO}. \citet[Lemma 10]{fischer2020_SobolevNorm} has proved that this kind of RKHS satisfies $\alpha_0 = \frac{1}{\beta}$.

\subsection{Sobolev RKHS}\label{section examples sobolev}

Let us first introduce some concepts of (fractional) Sobolev space (see, e.g., \citealt{adams2003_SobolevSpaces}). In this section, we assume that $\mathcal{X} \subseteq \mathbb{R}^{d}$ is a bounded domain with smooth boundary and Lebesgue measure $\nu$. Denote $L^{2}(\mathcal{X}) \coloneqq L^{2}(\mathcal{X},\nu)$ as the corresponding $L^{2}$ space. For $m \in \mathbb{N}$, we denote the usual Sobolev space $W^{m,2}(\mathcal{X})$ by $H^{m}(\mathcal{X})$ and $L^{2}(\mathcal{X})$ by $H^{0}(\mathcal{X})$. Then the (fractional) Sobolev space for any real number $r >0 $ can be defined through the \textit{real interpolation}:
\begin{displaymath}
      H^{r}(\mathcal{X}) := \left(L^{2}(\mathcal{X}), H^{m}(\mathcal{X})\right)_{\frac{r}{m},2},
  \end{displaymath}
where $m:=\min \{k \in \mathbb{N}: k > r\}$. (We refer to Appendix A for the definition of real interpolation and \citealt[Chapter 4.2.2]{sawano2018theory} for more details). It is well known that when $r > \frac{d}{2}$, $H^{r}(\mathcal{X})$ is a separable RKHS with respect to a bounded kernel and the corresponding EDR is (see, e.g., \citealt{edmunds_triebel_1996}) 
\begin{displaymath}
  \beta = \frac{2 r}{d}.
\end{displaymath}
Furthermore, for the interpolation space of $H^{r}(\mathcal{X}) $ under Lebesgue measure defined by \eqref{def interpolation space}, \eqref{inter relation} shows that for $ s > 0$,
\begin{displaymath}
  [H^{r}(\mathcal{X})]^{s} = H^{rs}(\mathcal{X}).
\end{displaymath}

The embedding theorem of (fractional) Sobolev space (see, e.g., 7.57 of \citealt{adams1975sobolev}) shows that if $ d < 2(r-j)$ for some nonnegative integer $j$, then 
\begin{displaymath}
    H^{r}(\mathcal{X}) \hookrightarrow C^{j,\theta}(\mathcal{X}), ~~ \theta= r-j-\frac{d}{2},
\end{displaymath}
where $C^{j,\gamma}(\mathcal{X}) $ denotes the Hölder space and $\hookrightarrow $ denotes the continuous embedding. Therefore for a Sobolev RKHS $ \mathcal{H} = H^{r}(\mathcal{X}), r > \frac{d}{2}$ and any $\alpha > \frac{1}{\beta} = \frac{d}{2r}$,
\begin{displaymath}
    [H^{r}(\mathcal{X})]^{\alpha} = H^{r\alpha}(\mathcal{X}) \hookrightarrow C^{0,\theta}(\mathcal{X}) \hookrightarrow L^{\infty}(\mathcal{X}),
\end{displaymath}
where $\theta > 0$. So the embedding index of a Sobolev RKHS is $ \alpha_{0} = \frac{1}{\beta}$.

Furthermore, if we suppose that $\mathcal{H}$ is a Sobolev RKHS, i.e., $\mathcal{H} = H^{r}(\mathcal{X})$ for some $r > d/2 $ and the distribution $\rho$ satisfies that the marginal distribution $\mu$ on $\mathcal{X}$ has Lebesgue density $ 0 < c \le p(x) \le C$ for two constants $c$ and $C$. Then we also know that the embedding index is $\alpha_{0} = \frac{1}{\beta} $. Note that we say that the distribution $\mu$ has Lebesgue density $0 < c \le p(x) \le C $, if $\mu$ is equivalent to the Lebesgue measure $\nu$, i.e., $ \mu \ll \nu, \nu \ll \mu$ and there exist constants $c, C > 0$ such that $c \leq \frac{\mathrm{d} \mu}{\mathrm{d} \nu} \leq C$.

\subsection{RKHS with shift-invariant periodic kernels}\label{section examples invar}
Let us consider a kernel on $\mathcal{X} = [-\pi,\pi)^d$ satisfying 
\begin{align*}
    k(x,y) = g\big( (x - y) \bmod [-\pi,\pi)^d\big),
\end{align*}
where we denote
\begin{align*}
    a \bmod [-\pi,\pi) = \left[(a+\pi)\bmod 2\pi \right] - \pi~ \in [-\pi,\pi),
\end{align*}
and
\begin{align*}
    (a_1,\dots,a_d)\bmod [-\pi,\pi)^d = \big(a_1 \bmod [-\pi,\pi),\dots, a_d \bmod [-\pi,\pi)\big).
\end{align*}
We further assume that $\mu$ is the uniform distribution on $[-\pi,\pi)^d$.
Then, it is shown in \citet{beaglehole2022_KernelRidgeless} that the Fourier basis $\phi_{\bm{k}}(x) = \exp(i \langle \bm{k},x \rangle)$, $\bm{k} \in \mathbb{Z}^d$ are eigenfunctions of the integral operator $T$.
Since $|{\phi_{\bm{k}}(x)}| \leq 1$, that is, the eigenfunctions are uniformly bounded, we conclude that the embedding index $\alpha_0 = \frac{1}{\beta}$. We refer to Section \ref{section proof invariant} for more details.

\subsection{RKHS with dot-product kernels}\label{section examples sphere}
Dot-product kernels, which satisfy $k(x,y) = f(\langle x,y \rangle)$, have also raised researchers' interest in recent years for its nice property~\citep{smola2000_RegularizationDotproduct,cho2009_KernelMethods,bach2017_BreakingCurse,jacot2018_NeuralTangent}.
Let $k$ be a dot-product kernel on $\mathcal{X} = \mathbb{S}^{d}$, 
the unit sphere in $\mathbb{R}^{d+1}$, and $\mu = \sigma$ be the uniform measure on $\mathbb{S}^{d}$.
Then, it is well-known that $k$ can be decomposed as
\begin{displaymath}
  k(x,y) = \sum_{n=0}^{\infty} \mu_n \sum_{l=1}^{a_n} Y_{n,l}(x)Y_{n,l}(y),
\end{displaymath}
where $\{Y_{n,l}\}$ is a set of orthonormal basis of $L^2(\mathbb{S}^{d},\sigma)$ called the spherical harmonics.
If polynomial decay condition $\mu_n \asymp n^{-d \beta}$ is satisfied (which is equivalent to assume the eigenvalue decay rate is $\beta$), 
 Proposition \ref{prop:EMBIdx_Sphere} shows that the embedding index $\alpha_0 = \frac{1}{\beta}$ for the corresponding RKHS. We refer to Section \ref{section proof sphere} for more details.



\section{Experiments}\label{section experiments}
In this section, we aim to verify through experiments that when $ \alpha_{0}-\frac{1}{\beta} < s < \alpha_{0}$, for those functions $f_{\rho}^{*} $ in $ [\mathcal{H}]^{s}$ but not in $L^{\infty}$, the spectral algorithms can still achieve the optimal convergence rate. We show the $L^{2}$-norm convergence results for two kinds of RKHSs and the three kinds of spectral algorithms mentioned in Section \ref{section spectral algorithms}.

Suppose that $\mathcal{X} = [0,1]$ and the marginal distribution $\mu$ is the uniform distribution on $[0,1]$. The first considered RKHS is $\mathcal{H} = H^{1}(\mathcal{X})$, the Sobolev space with smoothness 1. Section \ref{section examples sobolev} shows that the EDR is $\beta = 2$ and embedding index is $\alpha_{0} = \frac{1}{\beta}$. We construct a function in $[\mathcal{H}]^{s} \backslash L^{\infty}$ by 
\begin{equation}\label{series of sobolev}
    f^{*}(x) =  \sum\limits_{k=1}^{\infty} \frac{1}{k^{s+0.5}} \left( \sin\left( 2 k \pi x\right) + \cos\left( 2 k \pi x\right) \right),
\end{equation}
for some $0 < s < \frac{1}{\beta} = 0.5$. We will show in Appendix C that the series in \eqref{series of sobolev} converges on $(0,1)$. In addition, since $ \sin 2 k \pi + \cos 2 k \pi \equiv 1$, we also have $f^{*} \notin L^{\infty}(\mathcal{X})$. The explicit formula of the kernel associated to $H^{1}(\mathcal{X})$ is given by \citet[Corollary 2]{ThomasAgnan1996ComputingAF}, i.e., $ k(x,y) = \frac{1}{\sinh{1}} \cosh{(1- \max(x,y)) \cosh{(1- \min(x,y))}}$. 

For the second kind of RKHS, it is well known that the following RKHS 
\begin{align}
    \mathcal{H} = \mathcal{H}_{\text{min}}(\mathcal{X}):=\Big\{f:[0,1] \rightarrow \mathbb{R} \mid f &\text { is A.C., } f(0)=0, \int_0^1\left(f^{\prime}(x)\right)^2 \mathrm{~d} x<\infty\Big\}. \notag
\end{align}
is associated with the kernel $k(x,y) = \min(x,y)$ \citep{wainwright2019_HighdimensionalStatistics}. Further, its eigenvalues and eigenfunctions can be written as
\begin{displaymath}
    \lambda_n=\left(\frac{2 n-1}{2} \pi\right)^{-2}, \quad n=1,2, \cdots
\end{displaymath}
and
\begin{displaymath}
    e_n(x)=\sqrt{2} \sin \left(\frac{2 n-1}{2} \pi x\right), \quad n=1,2, \cdots
\end{displaymath}
It is easy to see that the EDR is $\beta = 2$ and the eigenfunctions are uniformly bounded. Section \ref{section examples ui} shows that the embedding index is $\alpha_{0} = \frac{1}{\beta}$. We construct a function in $[\mathcal{H}]^{s} \backslash L^{\infty}$ by
\begin{equation}\label{series of min}
    f^{*}(x) = \sum\limits_{k = 1}^{\infty} \frac{1}{k^{s+0.5}} e_{2k-1}(x),
\end{equation}
for some $0< s < \frac{1}{\beta} = 0.5$. We will show in Appendix C that the series in \eqref{series of min} converges on $(0,1)$. Since $e_{2k-1}(1) \equiv 1$, we also have $f^{*} \notin L^{\infty}(\mathcal{X})$. 

We consider the following data generation procedure:
\begin{displaymath}
    y = f^{*}(x) + \epsilon, 
\end{displaymath}
where $f^{*}$ is numerically approximated by the first 3000 terms in \eqref{series of sobolev} or \eqref{series of min} with $s=0.4$, $x \sim \mathcal{U}[0,1]$ and $\epsilon \sim \mathcal{N}(0,1)$. Three kinds of spectral algorithms (kernel ridge regression, gradient flow and spectral cut-off) are used to construct estimators $\hat{f}$ for each RKHS, where we choose the regularization parameter as $\nu = c n^{\frac{\beta}{s \beta + 1}} = c n^{\frac{10}{9}}$ for a fixed $c$. The sample size $n$ is chosen from 1000 to 5000, with intervals of 100. We numerically compute the generalization error $ \|\hat{f} - f^{*}\|_{L^{2}}$ by Simpson's formula with $N \gg n$ testing points. For each $n$, we repeat the experiments 50 times and present the average generalization error as well as the region within one standard deviation. To visualize the convergence rate $r$, we perform logarithmic least-squares $\log \text{err} = r \log n + b$ to fit the generalization error with respect to the sample size and display the value of $r$.

\begin{figure}[ht]
\vskip 0.05in
\centering
\subfigure[]{\includegraphics[width=0.3\columnwidth]{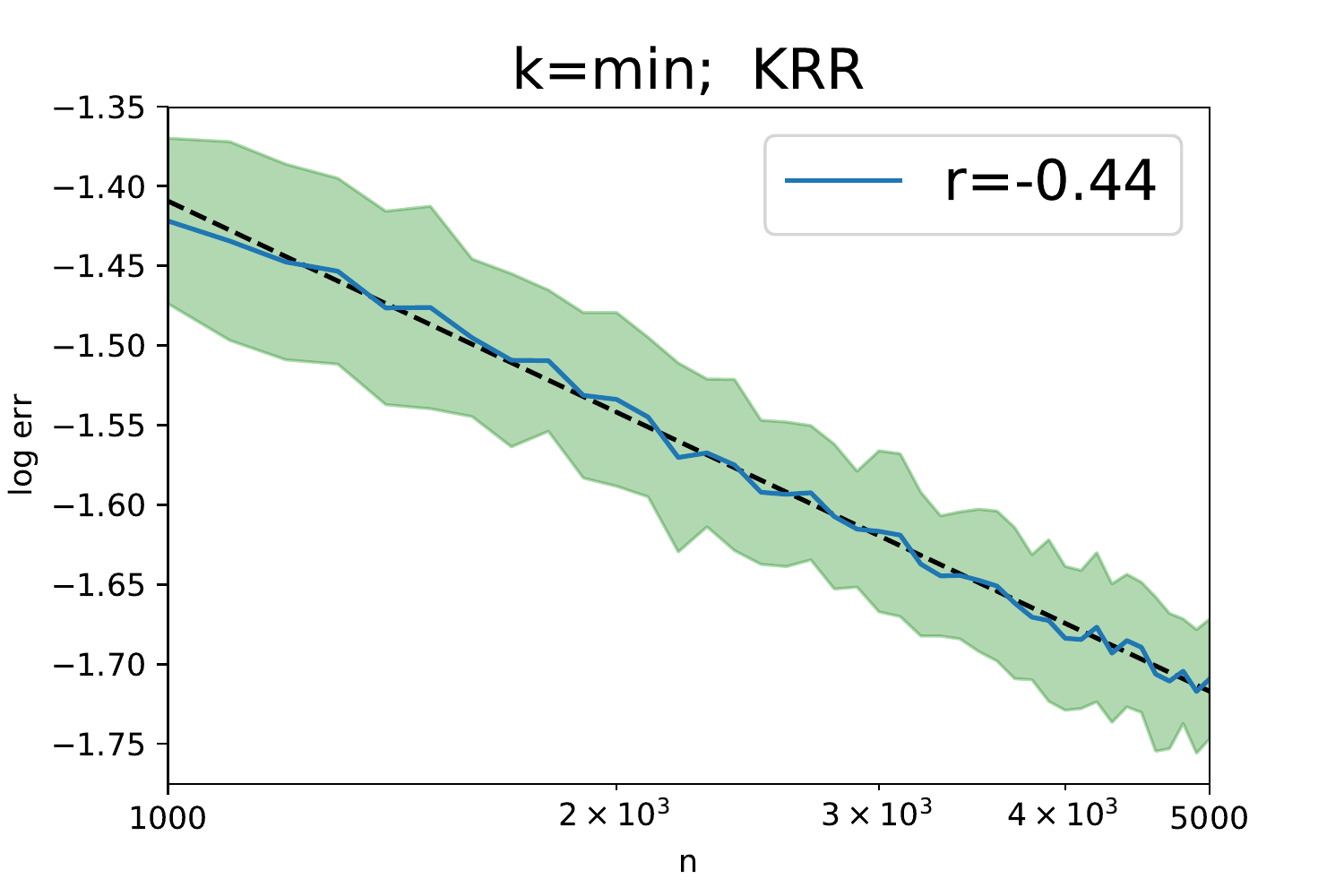}}
\subfigure[]{\includegraphics[width=0.3\columnwidth]{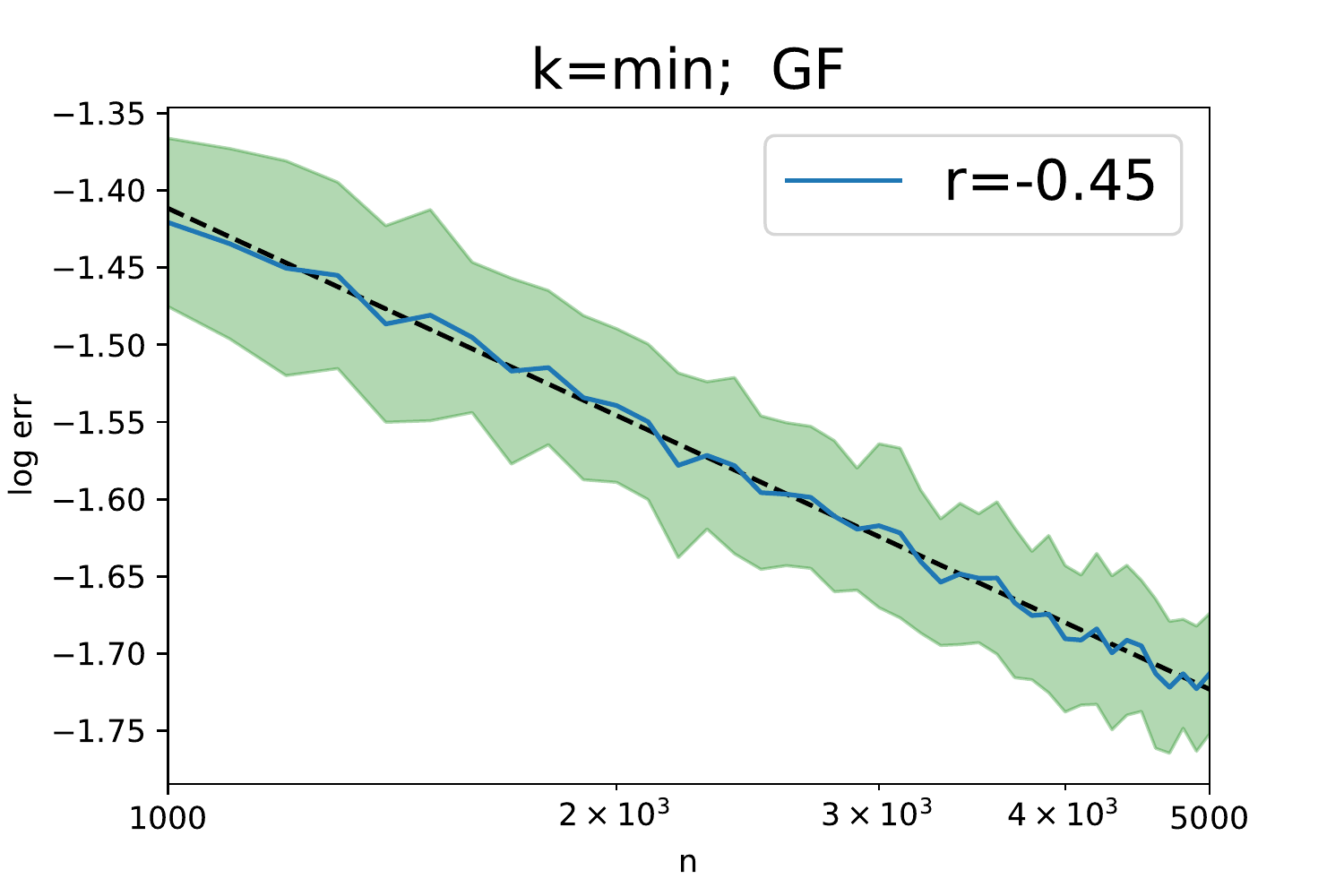}}
\subfigure[]{\includegraphics[width=0.3\columnwidth]{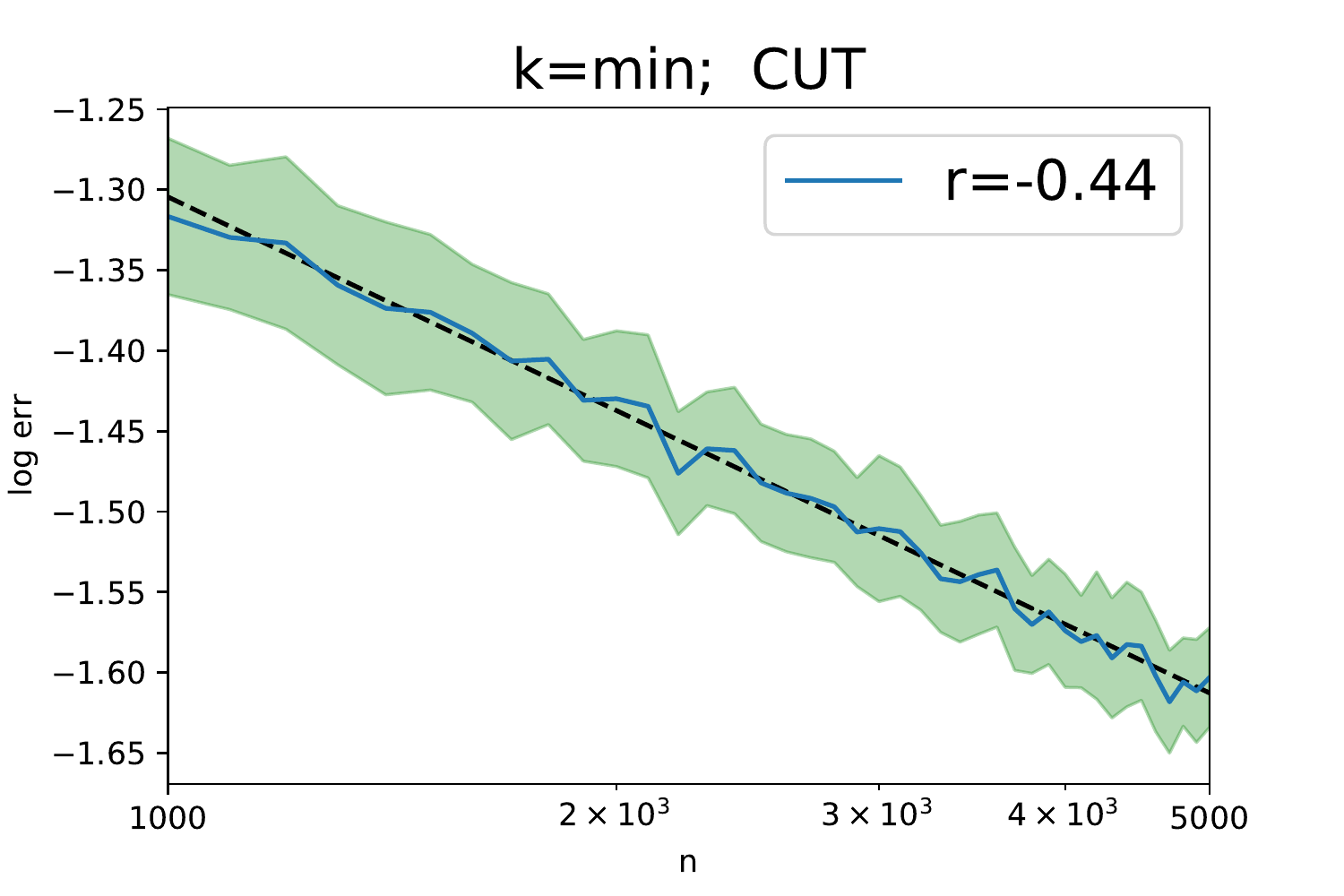}}

\subfigure[]{\includegraphics[width=0.3\columnwidth]{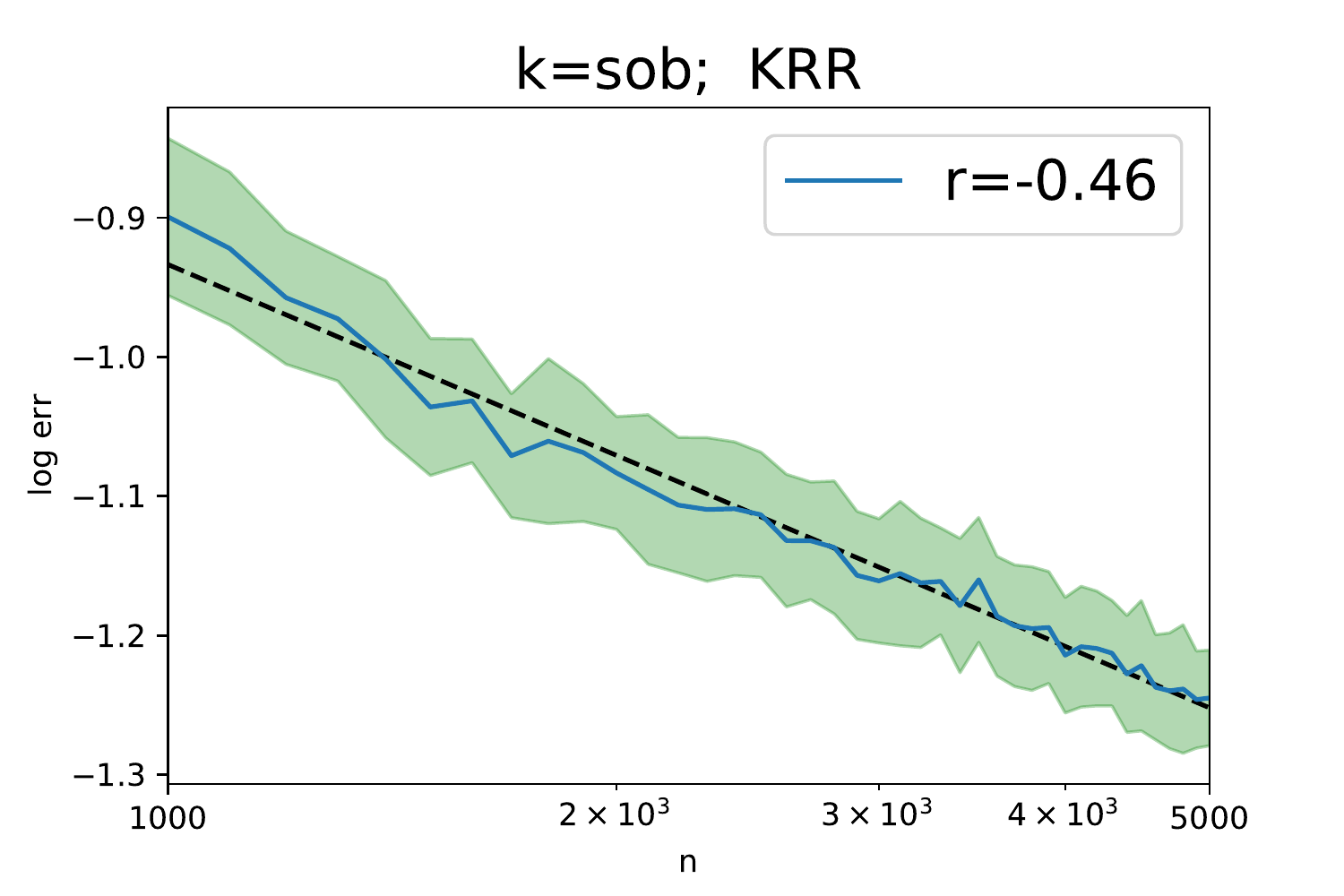}}
\subfigure[]{\includegraphics[width=0.3\columnwidth]{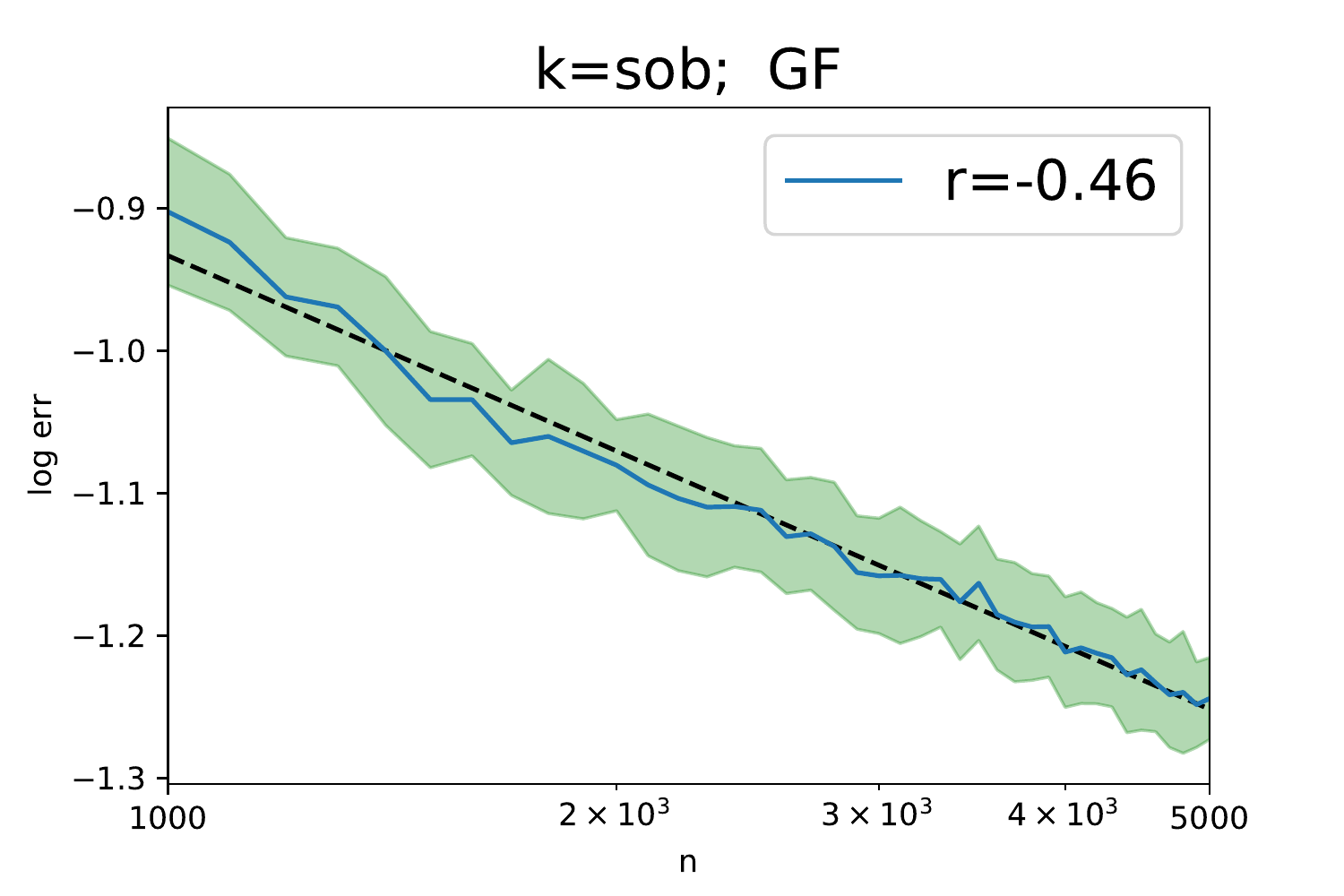}}
\subfigure[]{\includegraphics[width=0.3\columnwidth]{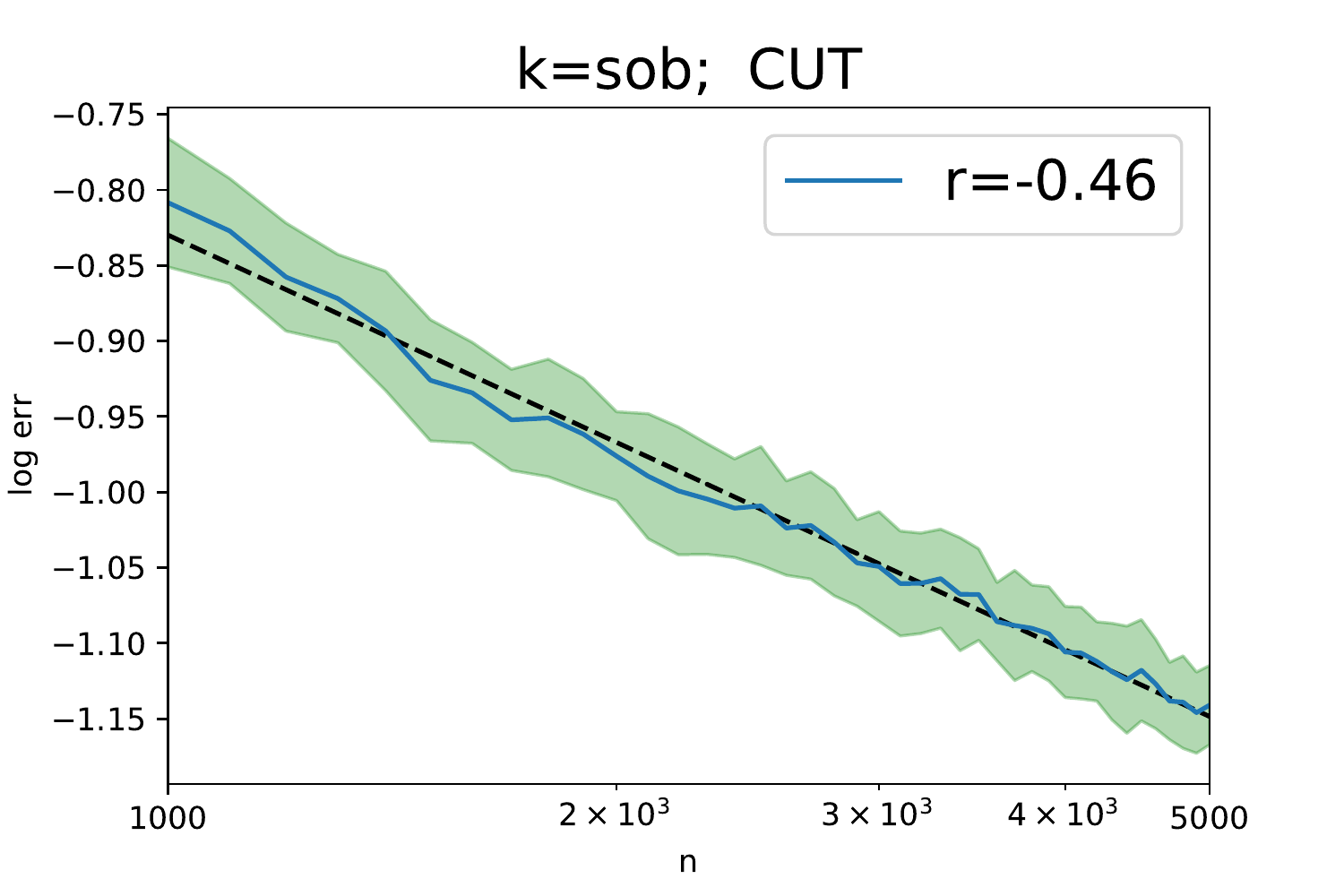}}

\caption{ Error decay curves of two kinds of RKHSs and three kinds of spectral algorithms with the best choice of $c$. Both axes are are scaled logarithmically. The curves show the average generalization errors over 50 trials; the regions within one standard deviation are shown in green. The dashed black lines are computed using logarithmic least-squares and the slopes represent the convergence rates $r$. Figures in the first row correspond to the Sobolev RKHS $\mathcal{H} = H^{1}(\mathcal{X})$ and the second correspond to the $\mathcal{H} = \mathcal{H}_{\text{min}}(\mathcal{X})$. } 
\label{figure bestc}
\vskip 0.05in
\end{figure}

We try different values of $c$, Figure \ref{figure bestc} presents the convergence curves under the best choice of $c$. For each setting, it can be concluded that the convergence rates of the $L^{2}$-norm generalization errors of spectral algorithms are indeed approximately equal to $ n^{-\frac{s \beta}{s \beta + 1}} = n^{-\frac{4}{9}}$, without the boundedness assumption of the true function $f^{*}$. We refer to Appendix C for more details on the experiments.

\section{Discussion}\label{section discussion}
\subsection{Comparison with related results}
In this subsection, we compare this paper's convergence rates and minimax optimality with the results in previous literature. Ignoring the log-term and the constants, Theorem \ref{main theorem} gives the upper bound of the convergence rates of spectral algorithms (with high probability)
\begin{equation}\label{discussion upper bound}
    \left\|\hat{f}_\nu-f_{\rho}^*\right\|_{[\mathcal{H}]^{\gamma}}^2 \le \begin{cases} n^{-\frac{(s-\gamma) \beta} {s \beta + 1}},\quad & \alpha_{0} - \frac{1}{\beta} < s \le 2 \tau , \\ n^{-\frac{s-\gamma } {\alpha_{0} + \epsilon}},\quad & 0 < s  \le \alpha_{0} - \frac{1}{\beta},~~ \forall \epsilon > 0. \end{cases}
\end{equation}
This $ [\mathcal{H}]^{\gamma} $-norm upper bound depends on $ \tau, \beta, s $ and $ \alpha_{0}$, among which $ \beta, \alpha_{0} $ characterize the information of the RKHS; $ s $ characterizes the relative `smoothness' of the true function; and $\tau$ characterizes the spectral algorithm. To our knowledge, this is the most general setting among related literature and will give the most refined analysis. In the well-specified case ($1 \le s \le 2\tau$ or $ f_{\rho}^{*} \in \mathcal{H}$), we recover the well-known minimax optimal rates from a lot of literature \citep[etc.]{Caponnetto2007OptimalRF,caponnetto2006optimal,dicker2017_KernelRidge, blanchard2018_OptimalRates,lin2018_OptimalRates,fischer2020_SobolevNorm} (for either general spectral algorithms or a specific kind).

The improvement in the misspecified case ($ 0 < s < 1$ or $ f_{\rho}^{*} \notin \mathcal{H}$) of this paper is partly due to the advantage of considering the embedding index $ \alpha_{0}$ of the RKHS. The best upper bound without considering the embedding index is (see, e.g., \citealt{dieuleveut2016nonparametric,lin2017optimal,lin2018_OptimalRates,lin2020_OptimalConvergence})
\begin{equation}\label{discussion upper bound no emb}
    \left\|\hat{f}_\nu-f_{\rho}^*\right\|_{[\mathcal{H}]^{\gamma}}^2 \le \begin{cases} n^{-\frac{(s-\gamma) \beta} {s \beta + 1}},\quad & 1 - \frac{1}{\beta} < s \le 2 \tau , \\ n^{-(s-\gamma) },\quad & 0 < s  \le 1 - \frac{1}{\beta}. \end{cases}
\end{equation}
This rate coincides with our upper bound \eqref{discussion upper bound} if the embedding index $\alpha_{0} = 1$. For those RKHSs with $\alpha_{0} < 1 $, \eqref{discussion upper bound} gives refined upper bound for all $ 0 < s \le 2 \tau$. As shown in Section \ref{section examples}, this is the case for many kinds of RKHSs. This is also why we assume $\alpha_{0} \in (0,1)$ throughout our paper.

Compared with the line of work which considers the embedding index \citep[etc.]{Steinwart2008SupportVM, PillaudVivien2018StatisticalOO,fischer2020_SobolevNorm}, this paper removes the boundedness assumption, i.e., $ \| f_{\rho}^{*} \|_{L^{\infty}(\mathcal{X},\mu)} \le B_{\infty} < \infty$. The upper bound in these works is the same as \eqref{discussion upper bound}. But due to the boundedness assumption, \citet{fischer2020_SobolevNorm} reveals that the minimax lower rate associated to the
smaller function space $[\mathcal{H}]^{s} \cap L^{\infty}(\mathcal{X,\mu})$ is larger than
\begin{displaymath}
  n^{-\frac{\max{(s,\alpha)}\beta}{\max{(s,\alpha)}\beta + 1}}, \quad \forall \alpha > \alpha_{0}.
\end{displaymath}
Therefore, they only prove the minimax optimality in the regime
\begin{displaymath}
  \alpha_{0} < s \le 2 \tau.
\end{displaymath}
Combining Theorem \ref{main theorem} and Theorem \ref{prop information lower bound}, this paper extends the minimax optimality of the spectral algorithms to the regime
\begin{displaymath}
  \alpha_{0} - \frac{1}{\beta} < s \le 2 \tau.
\end{displaymath}
This improvement is mainly due to the $ L^{q}$-embedding property of the interpolation space $[\mathcal{H}]^{s}$ proved in Theorem \ref{integrability of Hs} and a truncation method in the proof. Note that only the $ L^{\infty}$-embedding property has been considered before this paper. This new regime of minimax optimality means a lot. Since we have proved that the embedding index $\alpha_{0}$ equals $\frac{1}{\beta}$ for many kinds of RKHSs, the optimality in the misspecified case is well understood for these RKHSs. 

\subsection{Other discussions}
As mentioned in \cite{fischer2020_SobolevNorm}, the empirical process and the integral operator techniques are the two main techniques used to derive the learning rates of kernel methods. \cite{steinwart2009_OptimalRates} firstly introduced the embedding property of RKHS for the empirical process technique. \cite{fischer2020_SobolevNorm} further combined the embedding property of RKHS with the integral operator technique. Our paper further strengthens the integral operator technique, making it the most powerful technique for establishing learning rates of kernel methods (or spectral algorithms). We believe that our technical tools can be used for more related topics. For instance, some literature considers the general source condition, i.e.,
\begin{displaymath}
  f_{\rho}^{*} = \phi(L_{k}) g_{0}, \quad \text{with} ~~ \|g_{0}\|_{L^{2}} \le R,
\end{displaymath}
where $ \phi:\left[0, \kappa^2\right] \rightarrow \mathbb{R}^{+} $ is a non-decreasing index function such that $\phi(0) = 0$ and $\phi(\kappa^{2}) < \infty $ \citep{bauer2007_RegularizationAlgorithms,rastogi2017_OptimalRates,lin2018_OptimalRates,Talwai2022OptimalLR}. The source condition in Assumption \ref{ass source condition} corresponds to a special choice of $ \phi(x) = x^{\frac{s}{2}}$, which is often referred to as the Hölder source condition. Another interesting topic is the distributed version of spectral algorithms \citep{Zhang2013DivideAC,Lin2016DistributedLW,Guo2017LearningTO,Mcke2018ParallelizingSR,lin2020_OptimalConvergence}. It aims to reduce the computation complexity of the original spectral algorithms while maintaining the estimation efficiency. Last but not least, our techniques can also be applied to study other variants of gradient methods as \cite{lin2017optimal,lin2020convergences}. It would be interesting to apply this paper's tools and try refining the convergence rates or optimality in these scenarios.

In addition, we also notice a line of work which studies the learning curves of kernel ridge regression \citep{Spigler2020AsymptoticLC,Bordelon2020SpectrumDL,Cui2021GeneralizationER} and crossovers between different noise magnitudes. At present, their results all rely on a Gaussian design assumption (or some variation), which is a very strong assumption. We believe that studying the misspecified case in our paper is a crucial step to remove the Gaussian design assumption and draw complete conclusions about the learning curves of kernel ridge regression (or further, general spectral algorithms).

The eigenvalue decay rate (also known as the capacity condition or effective dimension condition) and source condition are mentioned in almost all related literature studying the convergence behaviors of kernel methods but are denoted as various kinds of notations. At the end of this section, we list a dictionary of nations in related literature. Recall that in this paper we denote the eigenvalue decay rate as $\beta$ and denote the source condition as $s$. Table \ref{table1} summarizes the notations used in some of the references.

\begin{table}[ht]
\centering
\begin{tabular}{ccc} 
\toprule
$\beta$  &   $s$ & Reference \\
\midrule
 $1/p$       &  $\beta$   & \citet{steinwart2009_OptimalRates,fischer2020_SobolevNorm,Li2022OptimalRF}  \\
 \midrule
 $1/\gamma$  &  $2\zeta$ & \citet{lin2017optimal,lin2018_OptimalRates,lin2020_OptimalConvergence}  \\
 \midrule
 $b$         &  $c$      & \citet{caponnetto2006optimal,Caponnetto2007OptimalRF}  \\
 \midrule
 $-$         &  $2 r$    & \citet{bauer2007_RegularizationAlgorithms,Smale2007LearningTE,gerfo2008_SpectralAlgorithms}  \\
 \midrule
 $2 \nu$     &  $\zeta+1$ & \citet{dicker2017_KernelRidge}  \\
 \midrule
 $b$        &  $2r+1$ & \citet{rastogi2017_OptimalRates,blanchard2018_OptimalRates}  \\
 \midrule
 $1/b$  &  $2\beta$ & \citet{jun2019kernel} \\
 \midrule
  
 \multirow{2}{*}{$\alpha$} & \multirow{2}{*}{$2r$} & \citet{dieuleveut2016nonparametric,PillaudVivien2018StatisticalOO} \\
  & & \citet{Celisse2020AnalyzingTD} \\
\bottomrule
\end{tabular}
\caption{A dictionary of notations in related literature.}
\label{table1}
\end{table}



\section{Proofs}\label{section proofs}
\subsection{\texorpdfstring{$L^{q}-$}~embedding property of the interpolation space}

Before introducing the $L^{q}$-embedding property of the interpolation space $[\mathcal{H}]^{s}$, we first prove the following lemma, which characterizes the real interpolation between two $ L^{p}$ spaces with Lorentz space $L^{p,q}(\mathcal{X}, \mu)$. We refer to Appendix A for details of \textit{real interpolation} and \textit{Lorentz spaces}.
\begin{lemma}\label{Lp interpolation}
  For $ 1 < p_{1} \neq p_{2} < \infty$, $ 1 \le q \le \infty$ and $ 0 < \theta < 1$, we have
  \begin{displaymath}
      \left( L^{p_{1}}(\mathcal{X},\mu), L^{p_{2}}(\mathcal{X},\mu) \right)_{\theta, q} = L^{p_{\theta},q}(\mathcal{X},\mu) ,\quad  \frac{1}{p_{\theta}} = \frac{1-\theta}{p_{1}} + \frac{\theta}{p_{2}},
  \end{displaymath}
  where $L^{p_{\theta},q}(\mathcal{X}, \mu)$ is the Lorentz space.
\end{lemma}
\begin{proof}
  Denote $L^{p}(\mathcal{X},\mu), L^{p,q}(\mathcal{X},\mu)$ as $L^{p}, L^{p,q}$ for brevity. Using Lemma \ref{cong of lorentz}, we know that $ L^{p_{i}} \cong L^{p_{i},p_{i}} = \left(L^{1}, L^{\infty}\right)_{\frac{1}{p_{i}^{\prime}}, p_{i}}$, where $\frac{1}{p_{i}^{\prime}} + \frac{1}{p_{i}} = 1, i=1,2.$ Since $1 < p_{i} <\infty, i=1,2$, Lemma \ref{mono of RI} implies that
  \begin{align}
      \left(L^{1}, L^{\infty}\right)_{\frac{1}{p_{1}^{\prime}}, 1} &\subset L^{p_{1}} \subset \left(L^{1}, L^{\infty}\right)_{\frac{1}{p_{1}^{\prime}}, \infty}; \notag \\
      \left(L^{1}, L^{\infty}\right)_{\frac{1}{p_{2}^{\prime}}, 1} &\subset L^{p_{2}} \subset \left(L^{1}, L^{\infty}\right)_{\frac{1}{p_{2}^{\prime}}, \infty}. \notag
  \end{align}
  Using the Reiteration theorem (Theorem \ref{reiteration theorem}), we have
  \begin{equation}\label{1.10-1}
      \left( L^{p_{1}}, L^{p_{2}} \right)_{\theta, q} = \left(L^{1}, L^{\infty}\right)_{\eta, q},
  \end{equation}
  where $ \eta = \frac{1 - \theta}{p_{1}^{\prime}} + \frac{\theta}{p_{2}^{\prime}}$. Simple calculations show that 
  \begin{align}
      1 - \eta = \frac{1 - \theta}{p_{1}} + \frac{\theta}{p_{2}} := \frac{1}{p_{\theta}}. \notag
  \end{align}
  So by the definition of Lorentz space, we have 
  \begin{displaymath}
      \left(L^{1}, L^{\infty}\right)_{\eta, q} = L^{\frac{1}{1-\eta},q} = L^{p_{\theta},q}.
  \end{displaymath}
  Together with \eqref{1.10-1}, we finish the proof.
\end{proof}

Based on Lemma \ref{Lp interpolation}, the following theorem gives the $L^{q}$-embedding property of the interpolation space of an RKHS $\mathcal{H}$, which is crucial for proving the upper bound. 
\begin{theorem}[$L^{q}$-embedding property]\label{integrability of Hs}
  Suppose that the RKHS $\mathcal{H}$ has embedding index $\alpha_{0}$, then for any $0 < s \le \alpha_{0}$, we have
  \begin{align}
      [\mathcal{H}]^{s} \hookrightarrow L^{q_{s}}(\mathcal{X}, \mu),\quad q_{s} = \frac{2 \alpha}{\alpha - s}, \quad \forall \alpha > \alpha_{0}. \notag
  \end{align}
  where $\hookrightarrow $ denotes the continuous embedding.
\end{theorem}
\begin{proof}

  Since the embedding index is $ \alpha_{0}$, we know that $ [\mathcal{H}]^{\alpha^{\prime}} \hookrightarrow L^{\infty}(\mathcal{X}, \mu), \forall \alpha^{\prime} > \alpha_{0}$. In addition, \eqref{inter relation} shows that
  \begin{align}
      [\mathcal{H}]^{s} = \big[ \left[ \mathcal{H} \right]^{\alpha^{\prime}}\big]^{\frac{s}{\alpha^{\prime}}} \cong \left(L^{2}(\mathcal{X}, \mu),\left[ \mathcal{H} \right]^{\alpha^{\prime}} \right)_{\frac{s}{\alpha^{\prime}},2}. \notag
  \end{align}
  
  So using Lemma \ref{Lp interpolation}, for any $ 0 < M < \infty$, we have
  \begin{align}
       [\mathcal{H}]^{s} \hookrightarrow \left(L^{2}(\mathcal{X}, \mu), L^{M}(\mathcal{X}, \mu) \right)_{\frac{s}{\alpha^{\prime}},2} \cong L^{q_{s}^{\prime},2}(\mathcal{X}, \mu), \notag
  \end{align}
  where $ \frac{1}{q_{s}^{\prime}} = \frac{1- \frac{s}{\alpha^{\prime}}}{2} + \frac{\frac{s}{\alpha^{\prime}}}{M} = \frac{\alpha^{\prime}-s}{2\alpha^{\prime}} + \frac{s}{\alpha^{\prime} M}$. 
  
  For any $ \alpha > \alpha_0$, we can choose the above $ \alpha^{\prime} \in \left( \alpha_0, \alpha \right)$ and $M$ large enough such that $ \frac{\alpha^{\prime}-s}{2\alpha^{\prime}} + \frac{s}{\alpha^{\prime} M} < \frac{\alpha-s}{2 \alpha} $. Letting $ q_s = \frac{2 \alpha}{\alpha - s}$, we have $ q_{s}^{\prime} > q_{s}$. Further, since $0 < s\le \alpha_{0} < \alpha$ thus $ q_{s}^{\prime} > q_{s} > 2$, using Lemma \ref{mono of RI} and Lemma \ref{cong of lorentz}, we have
  \begin{displaymath}
      L^{q_{s}^{\prime},2}(\mathcal{X}, \mu) \hookrightarrow L^{q_{s}^{\prime},q_{s}^{\prime}}(\mathcal{X}, \mu) \cong L^{q_{s}^{\prime}}(\mathcal{X}, \mu) \hookrightarrow L^{q_{s}}(\mathcal{X}, \mu).
  \end{displaymath}
  We finish the proof.

\end{proof}

\subsection{Some bounds}\label{section some bounds}
Throughout the proof, we denote 
\begin{displaymath}
   T_{\nu} = T + \nu^{-1};~~ T_{X \nu} = T_{X} + \nu^{-1},
\end{displaymath}
where $\nu$ is the regularization parameter. We use $\| \cdot \|_{\mathscr{B}(B_1,B_2)}$ to denote the operator norm of a bounded linear operator from a Banach space $B_1$ to $B_2$, i.e., $ \| A \|_{\mathscr{B}(B_1,B_2)} = \sup\limits_{\|f\|_{B_1}=1} \|A f\|_{B_2} $. Without bringing ambiguity, we will briefly denote the operator norm as $\| \cdot \|$. In addition, we use $\text{tr}A$ and $\| A \|_{1}$ to denote the trace and the trace norm of an operator. We use $\| A \|_{2}$ to denote the Hilbert-Schmidt norm. In addition, we denote $ L^{2}(\mathcal{X},\mu)$ as $ L^{2}$, $ L^{\infty}(\mathcal{X},\mu)$ as $ L^{\infty}$ for brevity throughout the proof.  We use $a_n \asymp b_n$ to denote that there exist constants $c$ and $C$ such that $c a_{n} \le b_{n} \le C a_{n}, \forall n=1,2,\cdots$; use $ a_n \lesssim b_n$ to denote that there exists an constant $C$ such that $a_{n} \le C b_{n}, \forall n=1,2,\cdots$ 

In addition, denote the effective dimension as
\begin{align}
    \mathcal{N}(\nu) = \text{tr}\big(T (T + \nu^{-1})^{-1}\big) = \sum\limits_{i \in N} \frac{\lambda_{i}}{\lambda_{i} + \nu^{-1}}. \notag
\end{align}
Since the EDR of $\mathcal{H}$ is $\beta$, Lemma \ref{lemma of effect} shows that $\mathcal{N}(\nu) \asymp \nu^{\frac{1}{\beta}} $. 

Recall that we have define the sample basis function $ g_{Z}$ and the spectral algorithm $ \hat{f}_{\nu}$ in Section \ref{section spectral algorithms}. We also need the following notations:
define the expectation of $g_{Z}$ as
\begin{displaymath}
    g = \mathbb{E} g_{Z} = \int_{\mathcal{X}} k(x,\cdot) f_{\rho}^{*}(x) d\mu(x) = S_{k}^{*} f_{\rho}^* \in \mathcal{H},
\end{displaymath}
and
\begin{displaymath}
    f_{\nu} = \varphi_{\nu}(T) g = \varphi_{\nu}(T) S_{k}^{*} f_{\rho}^{*}.
\end{displaymath}

The following theorem bounds the $[\mathcal{H}]^{\gamma}$-norm of $f_\nu -  f_{\rho}^*$ when $0 \le \gamma \le s $. 
\begin{theorem}\label{theorem of approximation error}
   Suppose that Assumption \ref{ass source condition} holds for $ 0 < s \le 2 \tau$. Then for any $\nu >0$ and $0 \le \gamma \le s $, we have 
   \begin{equation}\label{prop appr 1}
       \left\|f_\nu-f_{\rho}^*\right\|_{[\mathcal{H}]^{\gamma}} \le F_{\tau} R \nu^{-\frac{s-\gamma}{2}}.
   \end{equation}
\end{theorem}
\begin{proof}
  Suppose that $f_{\rho}^{*} = L_{k}^{\frac{s}{2}} g_{0}$ for some $g_{0} \in L^{2}$. Note that 
\begin{align}
    \left\| f_\nu-f_{\rho}^* \right\|_{[\mathcal{H}]^{\gamma}} &= \left\|L_{k}^{-\frac{\gamma}{2}} \left( S_{k} \varphi_{\nu}(T) S_{k}^{*} f_{\rho}^{*} - f_{\rho}^* \right) \right\|_{L^{2}} \notag \\
    &= \left\|L_{k}^{-\frac{\gamma}{2}} \left( \varphi_{\nu}(L_{k}) L_{k} - I \right) L_{k}^{\frac{s}{2}} g_{0} \right\|_{L^{2}} \notag \\
    &\le \left\|L_{k}^{\frac{s - \gamma}{2}} \psi_{\nu}(L_{k})\right\| R \notag \\
    &\le F_{\tau} R \nu^{-\frac{s-\gamma}{2}}, \notag
\end{align}
where we use the property of the filter function (\ref{prop2}) and $\left\| g_{0} \right\|_{L^{2}} = \| f_{\rho}^{*} \|_{[\mathcal{H}]^{s}} \le R$ for the last inequality.
\end{proof}

The following lemma bounds the $L^{\infty}$-norm of $f_{\nu} $ when $s \le \alpha_{0}$.
\begin{lemma}\label{prop infty norm}
   Suppose that Assumption \ref{ass EDR}, \ref{assumption embedding} and \ref{ass source condition} hold for $ 0 < s \le \alpha_{0}$ and $\frac{1}{\beta} \le \alpha_{0} < 1$. Then for any $\nu >0$ and any $\alpha_{0} < \alpha \le 1$, we have 
   \begin{equation}\label{prop appr 2}
       \| f_{\nu} \|_{L^{\infty}} \le M_{\alpha} E R \nu^{\frac{\alpha - s}{2}}.
   \end{equation}
\end{lemma}
\begin{proof}
    Since $s \le \alpha_{0}$ and $\alpha > \alpha_{0} $, we have 
    \begin{align}
      \left\| f_{\nu} \right\|_{[\mathcal{H}]^\alpha} &= \left\| L_{k}^{-\frac{\alpha}{2}} S_{k} \varphi_{\nu}(T) S_{k}^{*} f_{\rho}^{*}  \right\|_{[\mathcal{H}]^\alpha} \notag \\
      &= \left\| L_{k}^{-\frac{\alpha}{2}}  \varphi_{\nu}(L_{k}) L_{k} L_{k}^{\frac{s}{2}} g_{0}  \right\|_{L^{2}} \notag \\
      &= \left\| L_{k}^{1-\frac{\alpha -s}{2}}  \varphi_{\nu}(L_{k})  g_{0}  \right\|_{L^{2}} \notag \\
      &= \left\| L_{k}^{1-\frac{\alpha -s}{2}}  \varphi_{\nu}(L_{k}) \right\| \left\| g_{0}  \right\|_{L^{2}} \notag \\
      &\le E R  \nu^{\frac{\alpha - s}{2}}, \notag
\end{align}
where we use the property of the filter function (\ref{prop1}) for the last inequality. Further, using $\| [\mathcal{H}]^{\alpha} \hookrightarrow L^{\infty}(\mathcal{X}, \mu) \| = M_{\alpha} $ by Assumption \ref{assumption embedding}, we have $ \| f_{\nu} \|_{L^{\infty}} \le M_{\alpha} \| f_{\nu} \|_{[\mathcal{H}]^\alpha} \le M_{\alpha} E R  \nu^{\frac{\alpha - s}{2}}$.
\end{proof}

The following lemma will be frequently used in our proof.
\begin{lemma}\label{due embedding bound}
  Suppose that the RKHS $\mathcal{H}$ has embedding index $\alpha_{0}$. For any $\alpha_{0} < \alpha \le 1$, we have 
   \begin{align}\label{bound of Tk}
      \|T_{\nu}^{-\frac{1}{2}} k(x,\cdot) \|_{\mathcal{H}}^{2} \le M_{\alpha}^{2} \nu^{\alpha}, \quad \mu \text {-a.e. } x \in \mathcal{X}.
  \end{align}
\end{lemma}
\begin{proof}
   Recalling the definition of the embedding index, for any $\alpha_{0} < \alpha \le 1$, 
   \begin{displaymath}
    \sum_{i \in N} \lambda_i^\alpha e_i^2(x) \leq M_{\alpha}, \quad \mu \text {-a.e. } x \in \mathcal{X}.
   \end{displaymath}
   So, we have
   \begin{align}
      \|T_{\nu}^{-\frac{1}{2}} k(x,\cdot) \|_{\mathcal{H}}^{2} &= \Big\| \sum\limits_{i \in N} ( \frac{1}{\lambda_{i} + \nu^{-1}})^{\frac{1}{2}} \lambda_{i} e_{i}(x) e_{i}(\cdot)  \Big\|_{\mathcal{H}}^{2} \notag \\
      &=  \sum\limits_{i \in N}  \frac{\lambda_{i}}{\lambda_{i} + \nu^{-1}} e_{i}^{2}(x) \notag \\
      &= \big[ \sum\limits_{i \in N}  \lambda_{i}^{\alpha} e_{i}^{2}(x) \big] \sup\limits_{i \in N} \frac{\lambda_{i}^{1-\alpha}}{\lambda_{i} + \nu^{-1}} \notag \\
      & \le M_{\alpha}^{2} \nu^{\alpha}, \quad \mu \text {-a.e. } x \in \mathcal{X}. \notag
  \end{align}
  where we use Lemma \ref{basic ineq} for the last inequality and we finish the proof.
\end{proof}

Lemma \ref{due embedding bound} has a direct corollary.
\begin{lemma}\label{emb norm}
   Suppose that the RKHS $\mathcal{H}$ has embedding index $\alpha_{0}$. For any $\alpha_{0} < \alpha \le 1$, we have
\begin{displaymath}
    \| T_{\nu}^{-\frac{1}{2}} T_{x} T_{\nu}^{-\frac{1}{2}}\| \le M_{\alpha}^{2} \nu^{\alpha}, \quad \mu \text {-a.e. } x \in \mathcal{X}.
\end{displaymath}
\end{lemma}
\begin{proof}
  Note that for any $f \in \mathcal{H}$,
  \begin{align}
      T_{\nu}^{-\frac{1}{2}} T_{x} T_{\nu}^{-\frac{1}{2}} f &= T_{\nu}^{-\frac{1}{2}} K_{x} K_{x}^{*}  T_{\nu}^{-\frac{1}{2}} f \notag \\
      &= T_{\nu}^{-\frac{1}{2}} K_{x} \langle k(x,\cdot), T_{\nu}^{-\frac{1}{2}} f \rangle_{\mathcal{H}} \notag \\
      &= T_{\nu}^{-\frac{1}{2}} K_{x} \langle T_{\nu}^{-\frac{1}{2}} k(x,\cdot),  f \rangle_{\mathcal{H}} \notag \\
      &=  \langle T_{\nu}^{-\frac{1}{2}} k(x,\cdot),  f \rangle_{\mathcal{H}} \cdot T_{\nu}^{-\frac{1}{2}} k(x,\cdot). \notag
  \end{align}
  So $\| T_{\nu}^{-\frac{1}{2}} T_{x} T_{\nu}^{-\frac{1}{2}} \| = \sup\limits_{\| f\|_{\mathcal{H}}=1} \| T_{\nu}^{-\frac{1}{2}} T_{x} T_{\nu}^{-\frac{1}{2}} f\|_{\mathcal{H}} = \sup\limits_{\| f\|_{\mathcal{H}}=1} \langle T_{\nu}^{-\frac{1}{2}} k(x,\cdot),  f \rangle_{\mathcal{H}} \cdot \|T_{\nu}^{-\frac{1}{2}} k(x,\cdot) \|_{\mathcal{H}} = \|T_{\nu}^{-\frac{1}{2}} k(x,\cdot) \|_{\mathcal{H}}^{2}$. 
  Using Lemma \ref{due embedding bound}, we finish the proof.
\end{proof}

The following lemma is a corollary of Lemma \ref{bernstein}, which is also used in \citet[Lemma 5.5]{lin2018_OptimalRates} and \citet{Smale2007LearningTE}. 
\begin{lemma}\label{concentra of operator}
  Let $ 0 <\delta < \frac{1}{2}$, it holds with probability at least $1-\delta$
  \begin{displaymath}
    \left\|T_X-T\right\| \le \left\|T_X-T\right\|_2 \le \frac{8 \sqrt{2} \kappa^2}{\sqrt{n}} \ln \frac{2}{\delta},
  \end{displaymath}
  where $ \| \cdot \|$ denotes the operator norm and $ \| \cdot \|_{2}$ denotes the Hilbert-Schmidt norm.
\end{lemma}
\begin{proof}
  Define $\xi(x) = T_{x}$, then we have 
  \begin{displaymath}
     T_{X} - T  = \frac{1}{n} \sum\limits_{i=1}^{n} \xi(x_{i}) - \mathbb{E}\xi(x).
  \end{displaymath}
  Since $\sup\limits_{x \in \mathcal{X}} k(x,x) \le \kappa^{2}$, the Hilbert-Schmidt norm of $ \xi(x)$ satisfies that 
  \begin{displaymath}
      \left\| \xi(x) \right\|_{2} \le \kappa^{2}, \quad \forall x \in \mathcal{X}. 
  \end{displaymath}
  Applying Lemma \ref{bernstein} with $ L = \sigma = \kappa^{2}$, with probability at least $1-\delta$, we have
  \begin{displaymath}
      \left\| T_{X} - T \right\|_{2} \le 4 \sqrt{2} \ln{\frac{2}{\delta}} \left( \frac{\kappa^{2}}{n} + \frac{\kappa^{2}}{\sqrt{n}} \right) \le \frac{8 \sqrt{2} \kappa^2}{\sqrt{n}} \ln \frac{2}{\delta}.
  \end{displaymath}
  The first inequality follows from the fact that $ \left\|T_X-T\right\| \le \left\|T_X-T\right\|_2 $.
\end{proof}

\subsection{Upper bound}
\begin{lemma}\label{lemma4.6}
   Suppose that the RKHS $\mathcal{H}$ has embedding index $\alpha_{0}$. Then for any $\alpha_{0} < \alpha \le 1$ and all $\delta \in (0,1)$, with probability at least $1 - \delta$, we have
   \begin{displaymath}
        \Vert T_\nu^{-\frac{1}{2}} (T - T_X) T_\nu^{-\frac{1}{2}} \Vert
        \le \frac{4 M_{\alpha}^{2} \nu^{\alpha}}{3n} B + \sqrt {\frac{2 M_{\alpha}^{2} \nu^{\alpha}}{n} B},
   \end{displaymath}
   where 
   \begin{displaymath}
       B = \ln{\frac{4 \mathcal{N}(\nu) (\|T\| + \nu^{-1}) }{\delta \|T\|}}.
   \end{displaymath}
\end{lemma}
\begin{proof}
  Denote $A_{i} = T_\nu^{-\frac{1}{2}} (T - T_{x_{i}}) T_\nu^{-\frac{1}{2}} $, using Lemma \ref{emb norm}, we have 
  \begin{displaymath}
      \| A_{i} \| = \| T_\nu^{-\frac{1}{2}} T T_\nu^{-\frac{1}{2}} \| + \| T_\nu^{-\frac{1}{2}} T_{x_{i}} T_\nu^{-\frac{1}{2}} \| \le 2 M_{\alpha}^{2} \nu^{\alpha}, \quad \mu \text {-a.e. } x \in \mathcal{X}.
  \end{displaymath}
  We use $ A \preceq B$ to denote that $ A-B$ is a positive semi-definite operator. Using the fact that $\mathbb{E}(B-\mathbb{E} B)^2 \preceq \mathbb{E} B^2$ for a self-adjoint operator $B$, we have
  \begin{displaymath}
    \mathbb{E} A_{i}^{2} \preceq \mathbb{E}\left[T_\nu^{-\frac{1}{2}} T_{x_{i}} T_\nu^{-\frac{1}{2}}\right]^2.
    \end{displaymath}
  In addition, Lemma \ref{emb norm} shows that $0 \preceq T_\nu^{-\frac{1}{2}} T_{x_{i}} T_\nu^{-\frac{1}{2}} \preceq M_{\alpha}^{2}\nu^{\alpha}, \mu \text {-a.e. } x \in \mathcal{X}$. So we have
  \begin{displaymath}
     \mathbb{E} A_{i}^{2}  \preceq \mathbb{E} \left[T_\nu^{-\frac{1}{2}} T_{x_{i}} T_\nu^{-\frac{1}{2}}\right]^{2} \preceq \mathbb{E}\left[ M_{\alpha}^{2} \nu^{\alpha} \cdot T_\nu^{-\frac{1}{2}} T_{x_{i}} T_\nu^{-\frac{1}{2}}\right] = M_{\alpha}^{2}\nu^{\alpha} T_{\nu}^{-1} T,
  \end{displaymath}
  Define an operator $V := M_{\alpha}^{2} \nu^{\alpha} T_{\nu}^{-1} T$, we have  
  \begin{align}
      \| V \| &= M_{\alpha}^{2} \nu^{\alpha} \frac{\lambda_{1}}{\lambda_{1} + \nu^{-1}} = M_{\alpha}^{2} \nu^{\alpha} \frac{\|T\|}{\|T\| + \nu^{-1}} \le M_{\alpha}^{2} \nu^{\alpha}; \notag \\
      \text{tr}V &= M_{\alpha}^{2} \nu^{\alpha} \mathcal{N}(\nu); \notag \\
      \frac{\text{tr}V}{ \| V \|} &= \frac{\mathcal{N}(\nu) (\|T\| + \nu^{-1})}{\|T\|}. \notag
  \end{align}
  Use Lemma \ref{lemma concentration of operator} to $A_{i}$, $V$ and we finish the proof.
\end{proof}

\begin{lemma}\label{lemma 4.7}
   Suppose that the RKHS $\mathcal{H}$ has embedding index $\alpha_{0}$. For any $\alpha_{0} < \alpha \le 1$, if $\nu $ and $n$ satisfy that
   \begin{equation}\label{require 3.10}
    \frac{M_{\alpha}^{2} \nu^{\alpha}}{n} \ln \frac{4 \kappa^2 \mathcal{N}(\nu)\left(\|T\|+\nu^{-1}\right)}{\delta\|T\|} \leq \frac{1}{8},
    \end{equation}
   then for all $\delta \in (0,1)$, with probability at least $1 - \delta$, we have
   \begin{displaymath}
   \left\|T_\nu^{-\frac{1}{2}} T_{X \nu}^{\frac{1}{2}}\right\|^2 \leq 2, \quad\left\|T_\nu^{\frac{1}{2}} T_{X \nu}^{-\frac{1}{2}}\right\|^2 \leq 3.
   \end{displaymath}
\end{lemma}
\begin{proof}
   Define 
   \begin{displaymath}
      u = \frac{M_{\alpha}^{2} \nu^{\alpha}}{n} \ln \frac{4 \kappa^2 \mathcal{N}(\nu)\left(\|T\|+\nu^{-1}\right)}{\delta\|T\|} \le \frac{1}{8}.
   \end{displaymath}
   Using Lemma \ref{lemma4.6}, with probability at least $1 - \delta$, we have
   \begin{displaymath}
       a := \Vert T_\nu^{-\frac{1}{2}} (T - T_X) T_\nu^{-\frac{1}{2}} \Vert \le \frac{4}{3} u + \sqrt{2 u} \le \frac{2}{3}.
   \end{displaymath}
   So we have
   \begin{align}
   \left\|T_\nu^{-\frac{1}{2}} T_{X \nu}^{\frac{1}{2}}\right\|^2 & =\left\|T_\nu^{-\frac{1}{2}} T_{X \nu} T_\nu^{-\frac{1}{2}}\right\|=\left\|T_\nu^{-\frac{1}{2}}\left(T_X+\nu^{-1}\right) T_\nu^{-\frac{1}{2}}\right\| \notag \\ & =\left\|T_\nu^{-\frac{1}{2}}\left(T_X-T+T+\nu^{-1}\right) T_\nu^{-\frac{1}{2}}\right\| \notag \\ 
   & =\left\|T_\nu^{-\frac{1}{2}}\left(T_X-T\right) T_\nu^{-\frac{1}{2}}+I\right\| \notag\\ 
   & \leq a+1 \leq 2; \notag
   \end{align}
   and
   \begin{align}
    \left\|T_\nu^{\frac{1}{2}} T_{X \nu}^{-\frac{1}{2}}\right\|^2 & =\left\|T_\nu^{\frac{1}{2}} T_{X \nu}^{-1} T_\nu^{\frac{1}{2}}\right\|=\left\|\left(T_\nu^{-\frac{1}{2}} T_{X \nu} T_\nu^{-\frac{1}{2}}\right)^{-1}\right\| \notag \\
    &= \left\|\left(I - T_{\nu}^{-\frac{1}{2}} (T_{X}-T) T_{\nu}^{-\frac{1}{2}} \right)^{-1}\right\| \notag \\
    &\le \sum\limits_{k=0}^{\infty} \left\| T_{\nu}^{-\frac{1}{2}} (T_{X}-T) T_{\nu}^{-\frac{1}{2}} \right\|^{k}  \notag \\
    &\le \sum\limits_{k=0}^{\infty} \left(\frac{2}{3}\right)^{k} \le 3. \notag
    \end{align}
\end{proof}

The following theorem is an application of the classical Bernstein inequality but considering a truncation version of $f_{\rho}^{*}$, which will bring refined analysis when handling those $f_{\rho}^{*} \notin L^{\infty} $.
\begin{theorem}\label{theorem 4.9 boundness}
  Suppose that Assumption \ref{ass EDR}, \ref{assumption embedding}, \ref{ass source condition} and \ref{ass mom of error} hold for $ 0 < s \le 2 \tau$ and $\frac{1}{\beta} \le \alpha_{0} < 1$. Denote $\xi_{i} = \xi(x_{i},y_{i}) =  T_{\nu}^{-\frac{1}{2}}(K_{x_{i}} y_{i} - T_{x_{i}} f_{\nu}) $ and $ \Omega_{0} = \{x \in \mathcal{X}: |f_{\rho}^{*}(x)| \le t \}$. Then for any $\alpha_{0} < \alpha \le 1$ and all $\delta \in (0,1)$, with probability at least $1 - \delta$, we have
  \begin{displaymath}
    \left\| \frac{1}{n} \sum\limits_{i=1}^{n} \xi_{i}I_{x_{i} \in \Omega_{0}} - \mathbb{E}\xi_{x}I_{x \in \Omega_{0}} \right\|_{\mathcal{H}} \leq \ln \frac{2}{\delta}\left(\frac{C_1 \nu^{\frac{\alpha}{2}}}{n} \cdot \tilde{M} + \frac{C_2 \mathcal{N}^{\frac{1}{2}}(\nu)}{\sqrt{n}}+\frac{C_3 \nu^{\frac{\alpha-s}{2}}}{\sqrt{n}}\right),
    \end{displaymath}
    where $\tilde{M} = M_{\alpha} (E + F_{\tau}) R \nu^{\frac{\alpha - s}{2}} + t + L$ and $L$ is the constant in (\ref{ass mom of error}). $C_{1} = 8\sqrt{2} M_{\alpha}, C_{2} = 8\sigma, C_{3} = 8\sqrt{2} M_{\alpha} F_{\tau} R$.
\end{theorem}
\begin{proof}
   Note that $f_{\rho}^{*}$ can represent a $\mu$-equivalence class in $L^{2}(\mathcal{X},\mu)$. When defining the set $ \Omega_{0}$, we actually denote $f_{\rho}^{*}$ as the representative $f_{\rho}^{*}(x) = \int_{\mathcal{Y}} y \mathrm{d} \rho(y|x).$ 
   
   To use Lemma \ref{bernstein}, we need to bound the m-th moment of $ \xi(x,y) I_{x \in \Omega_{0}} $.
   \begin{align}\label{proof of 4.9-1}
       \mathbb{E} \left\| \xi(x,y) I_{x \in \Omega_{0}} \right\|_{\mathcal{H}}^{m} &= \mathbb{E} \left\| T_{\nu}^{-\frac{1}{2}} K_{x}(y - f_{\nu}(x))I_{x \in \Omega_{0}} \right\|_{\mathcal{H}}^{m} \notag \\
       &\le \mathbb{E} \Big( \left\| T_{\nu}^{-\frac{1}{2}} k(x,\cdot)\right\|_{\mathcal{H}}^{m}  \mathbb{E} \big( \left|(y - f_{\nu}(x)) I_{x \in \Omega_{0}} \right|^{m} ~\big|~ x\big) \Big). 
   \end{align}
Using the inequality $(a+b)^m \leq 2^{m-1}\left(a^m+b^m\right)$, we have 
\begin{align}\label{proof-plug}
    \left|y - f_\nu(x)\right|^m & \leq 2^{m-1}\left(\left|f_\nu(x)-f_{\rho}^{*}(x)\right|^m+\left|f_{\rho}^{*}(x)-y\right|^m\right) \notag \\
    & =2^{m-1}\left(\left|f_\nu(x)-f_{\rho}^{*}(x)\right|^m+|\epsilon|^m\right).
\end{align}
Plugging \eqref{proof-plug} into \eqref{proof of 4.9-1}, we have
\begin{align}
  \mathbb{E} \left\| \xi(x,y)I_{x \in \Omega_{0}} \right\|_{\mathcal{H}}^{m} ~\le~ &2^{m-1} \mathbb{E} \Big( \left\| T_{\nu}^{-\frac{1}{2}} k(x,\cdot)\right\|_{\mathcal{H}}^{m} \left|(f_{\nu}(x) - f_{\rho}^{*}(x)) I_{x \in \Omega_{0}} \right|^{m}\Big) \label{proof of 4.9-2} \\
  &+ 2^{m-1} \mathbb{E} \Big( \left\| T_{\nu}^{-\frac{1}{2}} k(x,\cdot)\right\|_{\mathcal{H}}^{m}  \mathbb{E} \big( \left| \epsilon~ I_{x \in \Omega_{0}} \right|^{m} ~\big|~ x \big) \Big) \label{proof of 4.9-3}
\end{align}

Now we begin to bound (\ref{proof of 4.9-3}). Note that we have proved in Lemma \ref{due embedding bound} that for $\mu$-almost $x \in \mathcal{X}$,
\begin{displaymath}
  \left\| T_{\nu}^{-\frac{1}{2}} k(x,\cdot)\right\|_{\mathcal{H}} \le M_{\alpha} \nu^{\frac{\alpha}{2}};  
\end{displaymath}
In addition, we also have
\begin{align}
    \mathbb{E} \left\|T_\nu^{-\frac{1}{2}} k(x,\cdot)\right\|_\mathcal{H}^{2} &= \mathbb{E}
    \Big\| \sum\limits_{i \in N} ( \frac{1}{\lambda_{i} + \nu^{-1}})^{\frac{1}{2}} \lambda_{i} e_{i}(x) e_{i}(\cdot)  \Big\|_{\mathcal{H}}^{2} \notag \\
    &= \mathbb{E} \Big( \sum\limits_{i \in N}  \frac{\lambda_{i}}{\lambda_{i} + \nu^{-1}} e_{i}^{2}(x) \Big) \notag \\
    &= \sum\limits_{i \in N}  \frac{\lambda_{i}}{\lambda_{i} + \nu^{-1}} \notag \\
    &=\mathcal{N}(\nu). \notag
\end{align}
So we have 
\begin{displaymath}
    \mathbb{E}\left\|T_\nu^{-\frac{1}{2}} k(x,\cdot) \right\|_{\mathcal{H}}^m \leq \left(M_{\alpha} \nu^{\frac{\alpha}{2}}\right)^{m-2} \cdot \mathbb{E}\left\|T_\nu^{-\frac{1}{2}} k(x,\cdot)\right\|_{\mathcal{H}}^2 = \big( M_{\alpha} \nu^{\frac{\alpha}{2}} \big)^{m-2} \mathcal{N}(\nu).
\end{displaymath}
Using Assumption \ref{ass mom of error}, we have
\begin{displaymath}
    \mathbb{E}\left(|\epsilon I_{x \in \Omega_{0}}|^m  \mid x\right) \le \mathbb{E}\left(|\epsilon|^m \mid x\right) \leq \frac{1}{2} m ! \sigma^2 L^{m-2}, \quad \mu \text {-a.e. } x \in \mathcal{X},
\end{displaymath}
so we get the upper bound of (\ref{proof of 4.9-3}), i.e., 
\begin{displaymath}
     (\ref{proof of 4.9-3}) \le \frac{1}{2} m !\left(\sqrt{2} \sigma \mathcal{N}^{\frac{1}{2}}(\nu)\right)^2(2 M_{\alpha} \nu^{\frac{\alpha}{2}} L)^{m-2}.
\end{displaymath}

Now we begin to bound (\ref{proof of 4.9-2}).
\begin{itemize}
    \item[(1)] When $s \le \alpha_{0}$, using the definition of $\Omega_{0}$ and Lemma \ref{prop infty norm}, we have 
    \begin{equation}\label{M-1}
    \| (f_{\nu} - f_{\rho}^{*}) I_{x \in \Omega_{0}} \|_{L^{\infty}} \le \| f_{\nu}  \|_{L^{\infty}} + \| f_{\rho}^{*} I_{x \in \Omega_{0}} \|_{L^{\infty}} \le M_{\alpha} E R \nu^{\frac{\alpha - s}{2}} + t.
    \end{equation}
    \item[(2)] When $ s > \alpha_{0} $, without loss of generality, we assume $\alpha_{0} < \alpha \le s$. using Theorem \ref{theorem of approximation error} for $\gamma = \alpha$, we have 
    \begin{equation}\label{M-2}
    \| (f_{\nu} - f_{\rho}^{*}) I_{x \in \Omega_{0}} \|_{L^{\infty}} \le M_{\alpha} \| f_{\nu}  - f_{\rho}^{*} \|_{[\mathcal{H}]^{\alpha}} \le M_{\alpha} F_{\tau} R \nu^{\frac{\alpha - s}{2}}.
    \end{equation}
\end{itemize}
Therefore, \eqref{M-1} and \eqref{M-2} imply that for all $0 < s \le 2$ we have 
\begin{align}\label{def of M}
    \| (f_{\nu} - f_{\rho}^{*}) I_{x \in \Omega_{0}} \|_{L^{\infty}} \le M_{\alpha} (E + F_{\tau}) R \nu^{\frac{\alpha - s}{2}} + t := M.
\end{align}
In addition, using Theorem \ref{theorem of approximation error} for $\gamma=0$, we also have 
\begin{displaymath}
    \mathbb{E} | (f_{\nu}(x) - f_{\rho}^{*}(x)) I_{x\in \Omega_{0}} |^{2} \le \mathbb{E} | f_{\nu}(x) - f_{\rho}^{*}(x)|^{2} \le (F_{\tau} R \nu^{-\frac{s}{2}})^{2}.
\end{displaymath}
So we get the upper bound of (\ref{proof of 4.9-2}), i.e.,
\begin{align}
    (\ref{proof of 4.9-2}) &\le 2^{m-1} (M_{\alpha} \nu^{\frac{\alpha}{2}})^{m} \cdot  \| (f_{\nu} - f_{\rho}^{*}) I_{x \in \Omega_{0}} \|_{L^{\infty}}^{m-2} \cdot \mathbb{E} | (f_{\nu}(x) - f_{\rho}^{*}(x)) I_{x\in \Omega_{0}} |^{2} \notag \\
    &\le 2^{m-1} (M_{\alpha} \nu^{\frac{\alpha}{2}})^{m} \cdot M^{m-2} \cdot (F_{\tau} R \nu^{-\frac{s}{2}})^{2} \notag \\
    &\le \frac{1}{2} m! \big( 2 M_{\alpha} \nu^{\frac{\alpha}{2}} M \big)^{m-2} \big( 2 M_{\alpha} F_{\tau} R \nu^{\frac{\alpha - s}{2}}\big)^{2}. \notag
\end{align}
Denote 
\begin{align}
    \tilde{L} &= 2 M_{\alpha} (M + L) \nu^{\frac{\alpha}{2}} \notag \\
    \tilde{\sigma} &= 2 M_{\alpha} F_{\tau} R \nu^{\frac{\alpha-s}{2}} + \sqrt{2} \sigma \mathcal{N}^{\frac{1}{2}}(\nu), \notag 
\end{align}
then the bounds of \eqref{proof of 4.9-2} and \eqref{proof of 4.9-3} show that $\mathbb{E} \left\| \xi(x,y) I_{x\in \Omega_{0}} \right\|_{\mathcal{H}}^{m} \le \frac{1}{2} m! \tilde{\sigma}^{2} \tilde{L}^{m-2} $. Using Lemma \ref{bernstein}, we finish the proof.
\end{proof}
\begin{remark}
  In fact, when we later applying Theorem \ref{theorem 4.9 boundness} in the proof of Theorem \ref{theorem 4.9}, the truncation in this theorem is necessary only in the $ s \le \alpha_{0}$ case. But for the completeness and consistency of our proof, we also include $s > \alpha_{0}$ in this theorem.
\end{remark}

Based on Theorem \ref{theorem 4.9 boundness}, the following theorem will give an upper bound of
\begin{displaymath}
    \left\|T_\nu^{-\frac{1}{2}}\left[\left(g_Z-T_X f_\nu\right)-\left(g-T f_\nu\right)\right]\right\|_{\mathcal{H}}.
\end{displaymath}

\begin{theorem}\label{theorem 4.9}
  Suppose that Assumption \ref{ass EDR}, \ref{assumption embedding}, \ref{ass source condition} and \ref{ass mom of error} hold for $ 0 < s \le 2 \tau$ and $\frac{1}{\beta} \le \alpha_{0} < 1$.
  \begin{itemize}
      \item In the case of $s + \frac{1}{\beta} > \alpha_{0}$, by choosing $ \nu \asymp n^{\frac{ \beta}{s \beta + 1}}$, for any fixed $\delta \in (0,1)$, when $n$ is sufficiently large, with probability at least $ 1- \delta$ , we have    
      \begin{equation}\label{goal of theorem 4.9}
          \left\|T_\nu^{-\frac{1}{2}}\left[\left(g_Z-T_X f_\nu\right)-\left(g-T f_\nu\right)\right]\right\|_{\mathcal{H}} \le \ln{\frac{2}{\delta}} C \frac{\nu^{\frac{1}{2 \beta}}}{\sqrt{n}} = \ln{\frac{2}{\delta}} C n^{-\frac{1}{2} \frac{s \beta}{s \beta +1}},
      \end{equation}
      where $C$ is a constant independent of $n$ and $\delta$.
      \item In the case of $s + \frac{1}{\beta} \le \alpha_{0} $, for any $ \alpha_{0} < \alpha \le 1$, by choosing $\nu \asymp (\frac{n}{\ln^{r}(n)})^{\frac{1}{\alpha}}$, for some $r > 1$, for any fixed $\delta \in (0,1)$, when $n$ is sufficiently large, with probability at least $1 - \delta$, we have
      \begin{equation}\label{goal of theorem 4.9-2}
          \left\|T_\nu^{-\frac{1}{2}}\left[\left(g_Z-T_X f_\nu\right)-\left(g-T f_\nu\right)\right]\right\|_{\mathcal{H}} \le \ln{\frac{2}{\delta}} C \frac{\nu^{\frac{\alpha-s}{2}}}{\sqrt{n}} \le \ln{\frac{2}{\delta}} C \left(\frac{n}{\ln^{r}(n)}\right)^{-\frac{1}{2} \frac{s}{\alpha}},
      \end{equation}
      where $C$ is a constant independent of $n$ and $\delta$.
  \end{itemize}
\end{theorem}
\begin{proof}
\paragraph{The $ \bm{s + \frac{1}{\beta} > \alpha_{0}} $ case:}

Denote $ \xi_{i} = \xi(x_{i},y_{i}) = T_{\nu}^{- \frac{1}{2}} (K_{x_{i}} y_{i} - T_{x_i}f_{\nu} )$, (\ref{goal of theorem 4.9}) is equivalent to
   \begin{equation}\label{equivalent goal}
       \left\| \frac{1}{n} \sum\limits_{i=1}^{n} \xi_{i} - \mathbb{E}\xi_{x} \right\|_{\mathcal{H}} \le \ln{\frac{2}{\delta}} C \frac{\nu^{\frac{1}{2 \beta}}}{\sqrt{n}} = \ln{\frac{2}{\delta}} C n^{-\frac{1}{2} \frac{s \beta}{s \beta +1}}.
   \end{equation}
Consider the subset $\Omega_{1} = \{x \in \mathcal{X}: |f_{\rho}^{*}(x)| \le t \}$ and $\Omega_{2} = \mathcal{X} \backslash \Omega_{1}$, where $t$ will be chosen appropriately later. Assume that for some $ q \ge 2$, 
\begin{displaymath}
    [\mathcal{H}]^{s} \hookrightarrow L^{q}(\mathcal{X},\mu).
\end{displaymath}
Then Assumption \ref{ass source condition} shows that there exists $0 < C_{q} < \infty$ such that $\| f_{\rho}^{*} \|_{L^{q}(\mathcal{X},\mu)} \le C_{q}$. Using the Markov inequality, we have
\begin{displaymath}
       P(x \in \Omega_{2}) = P\Big(|f_{\rho}^{*}(x)| > t \Big) \le \frac{\mathbb{E} |f_{\rho}^{*}(x)|^{q}}{t^{q}} \le \frac{(C_{q})^{q}}{t^{q}}.
\end{displaymath}
Decompose $\xi_{i}$ as $\xi_{i} I_{x_{i} \in \Omega_{1} } +  \xi_{i} I_{x_{i} \in \Omega_{2} }$ and we have
\begin{align}\label{decomposition}
    \left\|\frac{1}{n} \sum_{i=1}^n \xi_i-\mathbb{E} \xi_x\right\|_\mathcal{H} \le \left\|\frac{1}{n} \sum_{i=1}^n \xi_i I_{x_{i} \in \Omega_{1}}-\mathbb{E} \xi_x I_{x \in \Omega_{1}} \right\|_\mathcal{H} + \| \frac{1}{n} \sum_{i=1}^n \xi_i I_{x_{i} \in \Omega_{2}} \|_{_\mathcal{H}} + \| \mathbb{E} \xi_x I_{x \in \Omega_{2}} \|_{_\mathcal{H}}.
\end{align}
 For the first term in (\ref{decomposition}), denoted as \text{\uppercase\expandafter{\romannumeral1}}, Theorem \ref{theorem 4.9 boundness} shows that there exists $ \alpha_{0} < \alpha^{\prime} < s + \frac{1}{\beta}$ such that with probability at least $1-\delta$, we have
\begin{equation}\label{lead terms 2}
  \text{\uppercase\expandafter{\romannumeral1}} \leq \ln \frac{2}{\delta}\left(\frac{C_1 \nu^{\frac{\alpha^{\prime}}{2}}}{n} \cdot \tilde{M} +\frac{C_2 \nu^{\frac{1}{2 \beta}}}{\sqrt{n}}+\frac{C_3 \nu^{\frac{\alpha^{\prime} - s}{2}}}{\sqrt{n}}\right),
\end{equation}
where $\tilde{M} = M_{\alpha^{\prime}} (E + F_{\tau}) R \nu^{\frac{ \alpha^{\prime} - s}{2}} + t + L$. Recalling that $ \mathcal{N}(\nu) \asymp \nu^{\frac{1}{\beta}}$, simple calculation shows that by choosing $ \nu = n^{\frac{\beta}{s \beta + 1}}$,
\begin{itemize}
    \item the second term in \eqref{lead terms 2}:
    \begin{align}\label{1.2-4}
        \frac{C_2 \mathcal{N}^{\frac{1}{2}}(\nu)}{\sqrt{n}} \asymp \frac{\nu^{\frac{1}{2 \beta}}}{\sqrt{n}} = n^{-\frac{1}{2} \frac{s \beta}{s \beta +1}};  
    \end{align}
    \item the third term in \eqref{lead terms 2}:
    \begin{equation}\label{1.2-5}
        \frac{C_3 \nu^{\frac{\alpha^{\prime} - s}{2}}}{\sqrt{n}} \asymp  n^{\frac{1}{2}(\frac{\alpha^{\prime}}{s + 1 / \beta} - 1)} \cdot n^{-\frac{1}{2} \frac{s \beta}{s \beta +1}} \lesssim  n^{-\frac{1}{2} \frac{s \beta}{s \beta +1}};
    \end{equation}
    \item the first term in \eqref{lead terms 2}:
    \begin{align}\label{1.2-6}
        \frac{C_1 \nu^{\frac{\alpha^{\prime}}{2}}}{n}\cdot \tilde{M} &\asymp \frac{\nu^{\frac{\alpha^{\prime}}{2}}}{n} \nu^{\frac{\alpha^{\prime} - s}{2}} + \frac{\nu^{\frac{\alpha^{\prime}}{2}}}{n} \cdot t + \frac{\nu^{\frac{\alpha^{\prime}}{2}}}{n} \cdot L.
    \end{align}
\end{itemize}
Further calculations show that
\begin{displaymath}
    \frac{\nu^{\frac{\alpha^{\prime}}{2}}}{n} \nu^{\frac{\alpha^{\prime} - s}{2}} = n^{\frac{\alpha^{\prime}}{s + 1 / \beta} - 1} \cdot n^{-\frac{1}{2} \frac{s \beta}{s \beta +1}} \lesssim  n^{-\frac{1}{2} \frac{s \beta}{s \beta +1}},
\end{displaymath}
and 
\begin{displaymath}
    \frac{\nu^{\frac{\alpha^{\prime}}{2}}}{n} = n^{\frac{1}{2} \frac{\alpha^{\prime}\beta - s \beta -2}{s \beta +1}} \cdot n^{-\frac{1}{2}  \frac{s \beta}{s \beta +1}} \lesssim  n^{-\frac{1}{2} \frac{s \beta}{s \beta +1}}.
\end{displaymath}
To make $\eqref{1.2-6} \lesssim n^{-\frac{1}{2} \frac{s \beta}{s \beta +1}}$ when $ \nu = n^{\frac{\beta}{s \beta + 1}}$, letting $ \frac{\nu^{\frac{\alpha^{\prime}}{2}}}{n} \cdot t \le n^{-\frac{1}{2} \frac{s \beta}{s \beta +1}}$, we have the first restriction of $t$:
\begin{equation}\label{restrict1}
    \textbf{(R1)}:\quad t \le n^{\frac{1}{2} (1+\frac{1-\alpha^{\prime} \beta}{s \beta + 1})}.
\end{equation}
That is to say, if we choose $ t \le n^{\frac{1}{2} (1+\frac{1-\alpha^{\prime} \beta}{s \beta + 1})} $, we have 
\begin{displaymath}
    \text{\uppercase\expandafter{\romannumeral1}} \le \ln{\frac{2}{\delta}} C \frac{\nu^{\frac{1}{2 \beta}}}{\sqrt{n}} = \ln{\frac{2}{\delta}} C n^{-\frac{1}{2} \frac{s \beta}{s \beta +1}}.
\end{displaymath}
For the second term in (\ref{decomposition}), denoted as $\text{\uppercase\expandafter{\romannumeral2}}$, we have 
\begin{align}
    \tau_{n} := P(\text{\uppercase\expandafter{\romannumeral2}} > \frac{ \nu^{\frac{1}{2 \beta}}}{\sqrt{n}}) 
    &\le P\Big( ~\exists x_{i} ~\text{s.t.}~ x_{i} \in \Omega_{2}, \Big) = 1 - P\Big(x_{i} \notin \Omega_{2}, \forall x_{i},i=1,2,\cdots,n \Big) \notag \\
    &= 1 - P\Big(x \notin \Omega_{2}\Big)^{n} \notag \\
    &= 1 - P\Big( |f_{\rho}^{*}(x)| \le t\Big)^{n} \notag \\
    & \le 1 - \Big( 1 - \frac{(C_q)^{q}}{t^{q}}\Big)^{n}. \notag
\end{align}
Letting $\tau_{n} := P(\text{\uppercase\expandafter{\romannumeral2}} > \frac{ \nu^{\frac{1}{2 \beta}}}{\sqrt{n}}) \to 0$, we have $ \displaystyle{t^{q}} \gg n$, i.e. $t \gg n^{\frac{1}{q}} $. This gives the second restriction of $t$, i.e., 
\begin{equation}\label{restrict2}
   \textbf{(R2)}:\quad t \gg n^{\frac{1}{q}}, ~\text{or}~ n^{\frac{1}{q}} = o(t).
\end{equation}
For the third term in (\ref{decomposition}), denoted as $\text{\uppercase\expandafter{\romannumeral3}}$. Since Lemma \ref{due embedding bound} implies that $\| T_{\nu}^{-\frac{1}{2}} k(x,\cdot)\|_{\mathcal{H}} \le M_{\alpha^{\prime}} \nu^{\frac{\alpha^{\prime}}{2}}, \mu \text {-a.e. } x \in \mathcal{X},$  so
\begin{align}\label{third term}
    \text{\uppercase\expandafter{\romannumeral3}} &\le \mathbb{E}\| \xi_{x} I_{x\in\Omega_{2}} \|_{\mathcal{H}} \le \mathbb{E}\Big[ \| T_{\nu}^{-\frac{1}{2}} k(x,\cdot) \|_{\mathcal{H}} \cdot \big| \big(y-f_{\nu}(x) \big) I_{x\in\Omega_{2}}\big| \Big] \notag \\
    &\le M_{\alpha^{\prime}} \nu^{\frac{\alpha^{\prime}}{2}} \mathbb{E} \big| \big(y-f_{\nu}(x) \big) I_{x\in\Omega_{2}}\big| \notag \\
    &\le M_{\alpha^{\prime}} \nu^{\frac{\alpha^{\prime}}{2}} \Big( \mathbb{E} \big| \big(f_{\rho}^{*}(x)-f_{\nu}(x) \big) I_{x\in\Omega_{2}}\big| +  \mathbb{E} \big| \big(f_{\rho}^{*}(x)-y \big) I_{x\in\Omega_{2}}\big| \Big) \notag \\
    &\le M_{\alpha^{\prime}} \nu^{\frac{\alpha^{\prime}}{2}} \Big( \mathbb{E} \big| \big(f_{\rho}^{*}(x)-f_{\nu}(x) \big) I_{x\in\Omega_{2}}\big| +  \mathbb{E} \big| \epsilon \cdot I_{x\in\Omega_{2}}\big| \Big) .
\end{align}
Using Cauchy-Schwarz and the bound of approximation error (Theorem \ref{theorem of approximation error}), we have
\begin{align}\label{1.2-1}
    \mathbb{E} \big| \big(f_{\rho}^{*}(x)-f_{\nu}(x) \big) I_{x\in\Omega_{2}}\big| \le \left( \left\|f_{\rho}^*-f_\nu\right\|_{L^{2}}\right)^{\frac{1}{2}} \cdot \left(P(x \in \Omega_{2})\right)^{\frac{1}{2}} \le R \nu^{-\frac{s}{2}} C_{q}^{\frac{q}{2}} t^{-\frac{q}{2}}.
\end{align}
In addition, we have
\begin{align}\label{1.2-2}
    \mathbb{E} \big| \epsilon \cdot I_{x\in\Omega_{2}}\big| = \mathbb{E} \left( \mathbb{E} \big| \epsilon \cdot I_{x\in\Omega_{2}}\big| ~\Big|~ x\right) \le \sigma \mathbb{E} \left| I_{x\in\Omega_{2}}\right| \le \sigma (C_{q})^{q} t^{-q}.
\end{align}
Plugging \eqref{1.2-1} and \eqref{1.2-2} into \eqref{third term}, we have
\begin{equation}\label{1.2-3}
    \text{\uppercase\expandafter{\romannumeral3}} \le M_{\alpha^{\prime}} R C_{q}^{\frac{q}{2}} \nu^{\frac{\alpha^{\prime} - s }{2}} t^{-\frac{q}{2}} + M_{\alpha^{\prime}} \sigma (C_{q})^{q}  \nu^{\frac{\alpha^{\prime}}{2}} t^{-q}.
\end{equation}
Comparing (\ref{1.2-3}) with $C_3 \frac{ \nu^{\frac{\alpha^{\prime} - s}{2}}}{\sqrt{n}}$ and $C_1 \frac{\nu^{\frac{\alpha^{\prime}}{2}}}{n}$ in (\ref{lead terms 2}). We know that if $ t \ge n^{\frac{1}{q}}$, (\ref{third term}) $\le C \frac{\nu^{\frac{1}{2 \beta}}}{\sqrt{n}} = C n^{-\frac{1}{2} \frac{s \beta}{s \beta +1}}$. So the third term will not give further restriction of $t$.

To sum up, if we choose $t$ such that restrictions (\ref{restrict1}) and (\ref{restrict2}) are satisfied, then we can prove that (\ref{equivalent goal}) is satisfied with probability at least $ 1 - \delta - \tau_{n}, (\tau_{n} \to 0)$. Since for a fixed $\delta \in (0,1)$, when $n$ is sufficiently large, $\tau_{n}$ is sufficiently small such that, e.g., $\tau_{n} < \frac{\delta}{10}$. Without loss of generality, we say \eqref{equivalent goal} is satisfied with probability at least $ 1 - \delta$.  

Recalling restrictions (\ref{restrict1}) and (\ref{restrict2}), such $t$ exists if and only if $ [\mathcal{H}]^{s} \hookrightarrow L^{q}(\mathcal{X}, \mu)$ for some $q$ satisfying
\begin{equation}\label{require of q}
    \frac{1}{q} < \frac{1}{2} (1+\frac{1-\alpha^{\prime} \beta}{s \beta + 1}) \Longleftrightarrow q > \frac{2(s \beta + 1)}{2 + (s-\alpha^{\prime}) \beta}.
\end{equation}
If $s > \alpha_{0}$, $[\mathcal{H}]^{s} \hookrightarrow L^{\infty}(\mathcal{X}, \mu) $, hence \eqref{require of q} holds naturally. If $s \le \alpha_{0}$, Theorem \ref{integrability of Hs} shows that there exists $ \alpha_{0} < \alpha^{\prime \prime} < \alpha^{\prime} < s + \frac{1}{\beta}$ such that
\begin{displaymath}
    [\mathcal{H}]^{s} \hookrightarrow L^{q_{s}}(\mathcal{X}, \mu),\quad q_{s} = \frac{2 \alpha^{\prime \prime}}{\alpha^{\prime \prime} - s}.
\end{displaymath}
Further, $ \alpha^{\prime} > \alpha^{\prime \prime}$ and $ s + \frac{1}{\beta} > \alpha^{\prime}$ imply that
\begin{displaymath}
    \frac{2 \alpha^{\prime \prime}}{\alpha^{\prime \prime} - s} > \frac{2 \alpha^{\prime }}{\alpha^{\prime} - s} > \frac{2(s \beta + 1)}{2 + (s-\alpha^{\prime}) \beta}.
\end{displaymath}
So \eqref{require of q} holds for all $s + \frac{1}{\beta} > \alpha_{0}$ and we finish the proof of this case.

\vspace{15pt}
\paragraph{The $ \bm{s + \frac{1}{\beta} \le \alpha_{0}} $ case:}
Denote $ \xi_{i} = \xi(x_{i},y_{i}) = T_{\nu}^{- \frac{1}{2}} (K_{x_{i}} y_{i} - T_{x_i}f_{\nu} )$, for any fixed $\alpha_{0} < \alpha \le 1$, (\ref{goal of theorem 4.9-2}) is equivalent to
   \begin{equation}\label{equivalent goal-2}
       \left\| \frac{1}{n} \sum\limits_{i=1}^{n} \xi_{i} - \mathbb{E}\xi_{x} \right\|_{\mathcal{H}} \le \ln{\frac{2}{\delta}} C \frac{\nu^{\frac{\alpha-s}{2}}}{\sqrt{n}} \le \ln{\frac{2}{\delta}} C \left(\frac{n}{\ln^{r}(n)}\right)^{-\frac{1}{2} \frac{s}{\alpha}}, 
   \end{equation}
We also consider the subset $\Omega_{1} = \{x \in \Omega: |f_{\rho}^{*}(x)| \le t \}$ and $\Omega_{2} = \mathcal{X} \backslash \Omega_{1}$. Assume that for some $ q \ge 2$, 
\begin{displaymath}
    [\mathcal{H}]^{s} \hookrightarrow L^{q}(\mathcal{X},\mu).
\end{displaymath}
Similarly, decompose $\xi_{i}$ as $\xi_{i} I_{x_{i} \in \Omega_{1} } +  \xi_{i} I_{x_{i} \in \Omega_{2} }$ and we have
\begin{align}\label{decomposition-2}
    \left\|\frac{1}{n} \sum_{i=1}^n \xi_i-\mathbb{E} \xi_x\right\|_\mathcal{H} \le \left\|\frac{1}{n} \sum_{i=1}^n \xi_i I_{x_{i} \in \Omega_{1}}-\mathbb{E} \xi_x I_{x \in \Omega_{1}} \right\|_\mathcal{H} + \| \frac{1}{n} \sum_{i=1}^n \xi_i I_{x_{i} \in \Omega_{2}} \|_{_\mathcal{H}} + \| \mathbb{E} \xi_x I_{x \in \Omega_{2}} \|_{_\mathcal{H}}.
\end{align}
For the first term in (\ref{decomposition-2}), denoted as \text{\uppercase\expandafter{\romannumeral1}}, Theorem \ref{theorem 4.9 boundness} shows that there for this $\alpha > \alpha_{0} $, with probability at least $1-\delta$, we have
\begin{equation}\label{lead terms 2-2}
  \text{\uppercase\expandafter{\romannumeral1}} \leq \ln \frac{2}{\delta}\left(\frac{C_1 \nu^{\frac{\alpha}{2}}}{n} \cdot \tilde{M} +\frac{C_2 \nu^{\frac{1}{2 \beta}}}{\sqrt{n}}+\frac{C_3 \nu^{\frac{\alpha - s}{2}}}{\sqrt{n}}\right),
\end{equation}
where $\tilde{M} = M_{\alpha} (E+F_{\tau}) R \nu^{\frac{ \alpha - s}{2}} + t + L$. Simple calculation shows that by choosing $ \nu = (\frac{n}{\ln^{r}(n)})^{\frac{1}{\alpha}}$,
\begin{itemize}
    \item the second term in \eqref{lead terms 2-2}:
    \begin{align}\label{1.2-4-2}
        \frac{C_2 \mathcal{N}^{\frac{1}{2}}(\nu)}{\sqrt{n}} \asymp \frac{\nu^{\frac{1}{2 \beta}}}{\sqrt{n}} \lesssim \frac{\nu^{\frac{\alpha - s}{2}}}{\sqrt{n}};  
    \end{align}
    \item the third term in \eqref{lead terms 2-2}:
    \begin{equation}\label{1.2-5-2}
        \frac{C_3 \nu^{\frac{\alpha - s}{2}}}{\sqrt{n}} \asymp  n^{-\frac{1}{2}} n^{\frac{1}{2}-\frac{s}{2\alpha}} \left(\frac{1}{\ln^{r}(n)}\right)^{\frac{1}{2} - \frac{s}{2 \alpha}} \lesssim \left(\frac{n}{\ln^{r}(n)}\right)^{-\frac{1}{2} \frac{s}{\alpha}}.
    \end{equation}
    \item the first term in \eqref{lead terms 2-2}:
    \begin{align}\label{1.2-6-2}
        \frac{C_1 \nu^{\frac{\alpha}{2}}}{n}\cdot \tilde{M} &\asymp \frac{\nu^{\frac{\alpha}{2}}}{n} \nu^{\frac{\alpha - s}{2}} + \frac{\nu^{\frac{\alpha}{2}}}{n} \cdot t + \frac{\nu^{\frac{\alpha}{2}}}{n} \cdot L.
    \end{align}
    
\end{itemize}
Further calculations show that
\begin{displaymath}
    \frac{\nu^{\frac{\alpha}{2}}}{n} \nu^{\frac{\alpha - s}{2}} = \frac{\nu^{\frac{\alpha - s}{2}}}{\sqrt{n}} \cdot \frac{\nu^{\frac{\alpha}{2}}}{\sqrt{n}} = \frac{\nu^{\frac{\alpha - s}{2}}}{\sqrt{n}} \cdot \left(\frac{1}{\ln^{r}(n)}\right)^{\frac{\alpha}{2}} \lesssim \frac{\nu^{\frac{\alpha-s}{2}}}{\sqrt{n}}.
\end{displaymath}
and 
\begin{displaymath}
    \frac{\nu^{\frac{\alpha}{2}}}{n} \lesssim \frac{\nu^{\frac{\alpha-s}{2}}}{\sqrt{n}},
\end{displaymath}
To make $\eqref{1.2-6-2} \lesssim  \left(\frac{n}{\ln^{r}(n)}\right)^{-\frac{1}{2} \frac{s}{\alpha}}$ when $ \nu = (\frac{n}{\ln^{r}(n)})^{\frac{1}{\alpha}}$, letting $\frac{\nu^{\frac{\alpha}{2}}}{n} \cdot t \le  \frac{\nu^{\frac{\alpha-s}{2}}}{\sqrt{n}}$, we have the first restriction of $t$ (ignoring the $\log$ term):
\begin{equation}\label{restrict1-less}
    \textbf{(R1-2)}:\quad t \le n^{\frac{1}{2}\left(1 - \frac{s}{\alpha} \right)}.
\end{equation}
For the second and third terms in \eqref{decomposition-2}, we repeat the procedure as the case $ s +\frac{1}{\beta} > \alpha_{0}$, therefore the other restriction of $t$ remains unchanged, i.e.,  
\begin{equation}\label{restrict2-less}
   \textbf{(R2)}:\quad t \gg n^{\frac{1}{q}}, ~\text{or}~ n^{\frac{1}{q}} = o(t).
\end{equation}
These restrictions (\ref{restrict1-less}) and (\ref{restrict2-less}) shows that such $t$ exists if and only if $ [\mathcal{H}]^{s} \hookrightarrow L^{q}(\mathcal{X}, \mu)$ for some $q$ satisfying
\begin{equation}\label{1.9-1}
    \frac{1}{q} < \frac{1}{2}\left(1 - \frac{s}{\alpha} \right) \Longleftrightarrow q > \frac{2\alpha}{\alpha-s}.
\end{equation}
Recalling that $ \alpha > \alpha_{0}$ and $s +\frac{1}{\beta} \le \alpha_{0} $ implies $ s \le \alpha_{0}$, Theorem \ref{integrability of Hs} shows that there exists $\alpha_{0} < \alpha^{\prime} < \alpha$ such that
\begin{displaymath}
    [\mathcal{H}]^{s} \hookrightarrow L^{q_{s}}(\mathcal{X}, \mu),\quad q_{s} = \frac{2 \alpha^{\prime}}{\alpha^{\prime} - s},
\end{displaymath}
and
\begin{displaymath}
    \frac{2 \alpha^{\prime}}{\alpha^{\prime} - s} > \frac{2 \alpha}{\alpha - s}.
\end{displaymath}
So \eqref{1.9-1} holds for all $ s + \frac{1}{\beta} \le \alpha_{0}$ and we finish the proof of this case.
\end{proof}

\begin{theorem}[bound of estimation error]\label{estimation error thm}
  Suppose that Assumption \ref{ass EDR},\ref{assumption embedding}, \ref{ass source condition} and \ref{ass mom of error} hold for $ 0 < s \le 2 \tau$ and $\frac{1}{\beta} \le \alpha_{0} < 1$. Let $\hat{f}_{\nu}$ be the estimator defined by \eqref{SA estimator}. Then for $0 \le \gamma \le 1$ with $\gamma \le s$:
  \begin{itemize}[leftmargin = 18pt]
      \item In the case of $s + \frac{1}{\beta} > \alpha_{0} $, by choosing $\nu \asymp n^{\frac{\beta }{s \beta + 1}}$, for any fixed $\delta \in (0,1)$, when $n$ is sufficiently large, with probability at least $1 - \delta$, we have
      \begin{equation}\label{3.10-2}
          \left\|\hat{f}_\nu-f_{\nu}\right\|_{[\mathcal{H}]^{\gamma}}^2 \leq\left(\ln \frac{6}{\delta}\right)^2 C n^{-\frac{(s-\gamma) \beta}{s \beta+1}},
      \end{equation}
      where $C$ is a constant independent of $n$ and $\delta$.
      
      \item In the case of $s + \frac{1}{\beta} \le \alpha_{0} $, for any $ \alpha_{0} < \alpha \le 1$, by choosing $\nu \asymp (\frac{n}{\ln^{r}(n)})^{\frac{1}{\alpha}}$, for some $r > 1$, for any fixed $\delta \in (0,1)$, when $n$ is sufficiently large, with probability at least $1 - \delta$, we have
      \begin{equation}\label{3.10-3}
          \left\|\hat{f}_\nu-f_{\nu}\right\|_{[\mathcal{H}]^{\gamma}}^2 \leq \left(\ln \frac{6}{\delta}\right)^2 C \left(\frac{n}{\ln ^r(n)}\right)^{-\frac{s-\gamma}{\alpha}},
      \end{equation}
      where $C$ is a constant independent of $n$ and $\delta$.
  \end{itemize}
\end{theorem}
\begin{proof}
Using Lemma \ref{lemma 4.7}, Theorem \ref{theorem 4.9} and Lemma \ref{concentra of operator} for $ \frac{\delta}{3} \in (0,\frac{1}{3})$, with probability at least $1-\delta $, we have the following results hold simultaneously
\begin{equation}\label{3bounds-1}
   \left\|T_\nu^{-\frac{1}{2}} T_{X \nu}^{\frac{1}{2}}\right\|^2 \leq 2, \quad\left\|T_\nu^{\frac{1}{2}} T_{X \nu}^{-\frac{1}{2}}\right\|^2 \leq 3;
\end{equation}
\begin{displaymath}
    \eqref{goal of theorem 4.9} ~\text{and}~ \eqref{goal of theorem 4.9-2};
\end{displaymath}
\begin{equation}\label{3bounds-3}
    \left\|T_X-T\right\| \le \frac{8 \sqrt{2} \kappa^2}{\sqrt{n}} \ln \frac{6}{\delta}.
\end{equation}
Note that when choosing $\nu$ as in \eqref{3.10-2} or \eqref{3.10-3}, the condition \eqref{require 3.10} required in Lemma \ref{lemma 4.7} is always satisfied when $ n $ is sufficiently large. 

$ $\\
\textbf{Step 1}:
First, we rewrite the estimation error as follows,
\begin{align}\label{proof of 3.1-1}
    \left\|\hat{f}_\nu-f_\nu\right\|_{[\mathcal{H}]^{\gamma}} & =\left\|L_{k}^{-\frac{\gamma}{2}} S_{k} \left(\hat{f}_\nu-f_\nu\right)\right\|_{L^{2}} \notag \\
    & =\left\|L_{k}^{-\frac{\gamma}{2}} S_{k} T_\nu^{-\frac{1}{2}} \cdot T_\nu^{\frac{1}{2}} T_{X \nu}^{-\frac{1}{2}} \cdot T_{X \nu}^{\frac{1}{2}}\left(\hat{f}_\nu-f_\nu\right)\right\|_{L^{2}} \notag \\
    &\leq \left\|L_{k}^{-\frac{\gamma}{2}} S_{k} T_\nu^{-\frac{1}{2}}\right\|_{\mathscr{B}\left(\mathcal{H},L^{2}\right)} \cdot \left\|T_\nu^{\frac{1}{2}} T_{X \nu}^{-\frac{1}{2}}\right\|_{\mathscr{B}(\mathcal{H})} \cdot \left\|T_{X \nu}^{\frac{1}{2}}\left(\hat{f}_\nu-f_\nu\right)\right\|_{\mathcal{H}} .
\end{align}
For any $f \in \mathcal{H}$ and $\|f\|_{\mathcal{H}}=1$, suppose that $f = \sum\limits_{i \in N} a_{i} \lambda_{i}^{1/2} e_{i}$ satisfying that $ \sum\limits_{i \in N} a_{i}^{2} = 1$. So for the first term in \eqref{proof of 3.1-1}, we have
\begin{align}
    \left\|L_{k}^{-\frac{\gamma}{2}} S_{k} T_\nu^{-\frac{1}{2}}\right\|_{\mathscr{B}\left(\mathcal{H},L^{2}\right)} &= \sup_{\|f\|_{\mathcal{H}}=1} \left\|L_{k}^{-\frac{\gamma}{2}} S_{k} T_\nu^{-\frac{1}{2}} f\right\|_{L^{2}} \notag \\
    &\le \sup_{\|f\|_{\mathcal{H}}=1} \left\|\sum\limits_{i \in N} \frac{\lambda_{i}^{\frac{1-\gamma}{2}}}{(\lambda_{i}+\nu^{-1})^{\frac{1}{2}}} a_{i} e_{i} \right\|_{L^{2}} \notag \\
    &\le  \sup_{i \in N} \frac{\lambda_{i}^{\frac{1-\gamma}{2}}}{(\lambda_{i}+\nu^{-1})^{\frac{1}{2}}} \cdot \sup_{\|f\|_{\mathcal{H}}=1} \left\|\sum\limits_{i \in N} a_{i} e_{i} \right\|_{L^{2}} \notag \\
    &\le \nu^{\frac{\gamma}{2}}, \notag 
\end{align}
where we use Lemma \ref{basic ineq} for the last inequality. For the second term in (\ref{proof of 3.1-1}), \eqref{3bounds-1} shows that
\begin{displaymath}
  \left\|T_\nu^{\frac{1}{2}} T_{X \nu}^{-\frac{1}{2}}\right\|_{\mathscr{B}(\mathcal{H})} \le 3.
\end{displaymath}
For the third term in (\ref{proof of 3.1-1}), noticing that $z \varphi_\nu + \psi_\nu=1$, we have
  \begin{align}
    \hat{f}_\nu-f_\nu & =\varphi_\nu\left(T_X\right) g_Z-\left(T_X \varphi_\nu\left(T_X\right)+\psi_\nu\left(T_X\right)\right) f_\nu \notag \\
    & =\varphi_\nu\left(T_X\right)\left(g_Z-T_X f_\nu\right)-\psi_\nu\left(T_X\right) f_\nu \notag 
  \end{align}
So for the third term in (\ref{proof of 3.1-1}),
  \begin{equation}\label{third decompose}
      \left\|T_{X \nu}^{\frac{1}{2}}\left(\hat{f}_\nu-f_\nu\right)\right\|_{\mathcal{H}} \le \left\|T_{X \nu}^{\frac{1}{2}} \varphi_\nu\left(T_X\right)\left(g_Z-T_X f_\nu\right)\right\|_{\mathcal{H}} + \left\|T_{X \nu}^{\frac{1}{2}} \psi_\nu\left(T_X\right) f_\nu\right\|_{\mathcal{H}}.
  \end{equation}
$ $\\
\textbf{Step 2}: Now we begin to bound the first term in (\ref{third decompose}), i.e., 
\begin{align}\label{third decompose rewrite}
    \left\|T_{X \nu}^{\frac{1}{2}} \varphi_\nu\left(T_X\right)\left(g_Z-T_X f_\nu\right)\right\|_{\mathcal{H}}= & \left\|T_{X \nu}^{\frac{1}{2}} \varphi_\nu\left(T_X\right) T_{X \nu}^{\frac{1}{2}} \cdot T_{X \nu}^{-\frac{1}{2}} T_\nu^{\frac{1}{2}} \cdot T_\nu^{-\frac{1}{2}}\left(g_Z-T_X f_\nu\right)\right\|_{\mathcal{H}} \notag \\
    \leq & \left\|T_{X \nu}^{\frac{1}{2}} \varphi_\nu\left(T_X\right) T_{X \nu}^{\frac{1}{2}}\right\|_{\mathscr{B}(\mathcal{H})} \cdot \left\|T_{X \nu}^{-\frac{1}{2}} T_\nu^{\frac{1}{2}}\right\|_{\mathscr{B}(\mathcal{H})} \cdot \left\|T_\nu^{-\frac{1}{2}}\left(g_Z-T_X f_\nu\right)\right\|_{\mathcal{H}}.
  \end{align}
   The property of filter function (\ref{prop1}) shows that $z \varphi_\nu(z) \leq E$ and $\varphi_\nu(z) \leq E \nu$. So we have 
   \begin{equation}\label{proof 3.1 mid-1}
       \left\|T_{X \nu}^{\frac{1}{2}} \varphi_\nu\left(T_X\right) T_{X \nu}^{\frac{1}{2}}\right\|_{\mathscr{B}(\mathcal{H})} = \left\|\left(T_X+\nu^{-1}\right) \varphi_\nu\left(T_X\right)\right\|_{\mathscr{B}(\mathcal{H})} \leq 2 E;
   \end{equation}
   \eqref{3bounds-1} shows that 
   \begin{equation}\label{proof 3.1 mid-2}
       \left\|T_{X \nu}^{-\frac{1}{2}} T_\nu^{\frac{1}{2}}\right\|_{\mathscr{B}(\mathcal{H})} \le 2;
   \end{equation}
   In addition, recalling that at the beginning we have assumed that \eqref{goal of theorem 4.9} and \eqref{goal of theorem 4.9-2} hold, therefore we have
   \begin{itemize}
      \item In the case of $s + \frac{1}{\beta} > \alpha_{0}$, by choosing $ \nu \asymp n^{\frac{ \beta}{s \beta + 1}}$, we have
      \begin{align}\label{proof 3.1 mid-3}
       \left\|T_\nu^{-\frac{1}{2}}\left(g_Z-T_X f_\nu\right)\right\|_{\mathcal{H}} &\le \left\|T_\nu^{-\frac{1}{2}}\left[\left(g_Z-T_X f_\nu\right)-\left(g-T f_\nu\right)\right]\right\|_{\mathcal{H}} + \left\|T_\nu^{-\frac{1}{2}}\left(g -T f_\nu\right)\right\|_{\mathcal{H}} \notag \\
       &\le \ln(\frac{6}{\delta}) C n^{-\frac{1}{2} \frac{s \beta}{s \beta +1}} + \left\|T_\nu^{-\frac{1}{2}} \left( S_{k}^{*} f_{\rho}^{*} - S_{k}^{*} S_{k} f_\nu\right)\right\|_{\mathcal{H}} \notag \\
       &\le \ln(\frac{6}{\delta}) C n^{-\frac{1}{2} \frac{s \beta}{s \beta +1}} + \left\|T_\nu^{-\frac{1}{2}} S_{k}^{*}\right\|_{\mathscr{B}(L^{2},\mathcal{H})} \|  f_{\rho}^{*} - f_{\nu}  \|_{L^{2}} \notag \\
       &\le \ln(\frac{6}{\delta}) C n^{-\frac{1}{2} \frac{s \beta}{s \beta +1}} + \|  f_{\rho}^{*} - f_{\nu}  \|_{L^{2}} \notag \\ 
       &\le \ln(\frac{6}{\delta}) C n^{-\frac{1}{2} \frac{s \beta}{s \beta +1}} + F_{\tau} R n^{-\frac{1}{2} \frac{s \beta}{s \beta +1}}. \notag \\
       &\le \ln(\frac{6}{\delta}) C n^{-\frac{1}{2} \frac{s \beta}{s \beta +1}},
   \end{align}
   where we use the fact that $\left\|T_\nu^{-\frac{1}{2}} S_{k}^{*}\right\|_{\mathscr{B}(L^{2},\mathcal{H})} \le 1$ and use Theorem \ref{theorem of approximation error} with $\gamma = 0$ to bound $\|  f_{\rho}^{*} - f_{\nu}  \|_{L^{2}}$.
      \item In the case of $s + \frac{1}{\beta} \le \alpha_{0} $, for any $ \alpha_{0} < \alpha \le 1$, by choosing $\nu \asymp (\frac{n}{\ln^{r}(n)})^{\frac{1}{\alpha}}$, for some $r > 1$, we have
      \begin{align}\label{proof 3.1 mid-4}
           \left\|T_\nu^{-\frac{1}{2}}\left(g_Z-T_X f_\nu\right)\right\|_{\mathcal{H}} &\le \left\|T_\nu^{-\frac{1}{2}}\left[\left(g_Z-T_X f_\nu\right)-\left(g-T f_\nu\right)\right]\right\|_{\mathcal{H}} + \left\|T_\nu^{-\frac{1}{2}}\left(g -T f_\nu\right)\right\|_{\mathcal{H}} \notag \\
           &\le \ln{\frac{6}{\delta}} C \left(\frac{n}{\ln^{r}(n)}\right)^{-\frac{1}{2} \frac{s}{\alpha}} + F_{\tau} R \left(\frac{n}{\ln^{r}(n)}\right)^{-\frac{1}{2} \frac{s}{\alpha}} \notag \\
           &\le \ln{\frac{6}{\delta}} C \left(\frac{n}{\ln^{r}(n)}\right)^{-\frac{1}{2} \frac{s}{\alpha}}.
      \end{align}
  \end{itemize}
Therefore, plugging \eqref{proof 3.1 mid-1} \eqref{proof 3.1 mid-2} \eqref{proof 3.1 mid-3} \eqref{proof 3.1 mid-4} into \eqref{third decompose rewrite}, we get the desired upper bounds of the first term in \eqref{third decompose}. Specifically, the bound in \eqref{proof 3.1 mid-3} determines the bound of \eqref{third decompose rewrite} in the the case of $s + \frac{1}{\beta} > \alpha_{0}$; and \eqref{proof 3.1 mid-4} determines the case of $s + \frac{1}{\beta} \le \alpha_{0}$.

$ $\\
\textbf{Step 3}: Now we begin to bound the second term in (\ref{third decompose}), i.e., 
\begin{equation}\label{step3 goal}
    \left\|T_{X \nu}^{\frac{1}{2}} \psi_\nu\left(T_X\right) f_\nu\right\|_{\mathcal{H}}.
\end{equation}
We discuss three conditions of $s$.
\begin{itemize}[leftmargin = 18pt]
    \item $0 < s < 1$: Since $ (a+b)^{p} \le a^{p} + b^{p}$ for $ p \in [0,1]$, we have
    \begin{displaymath}
        \left\|T_{X \nu}^{\frac{1}{2}} \psi_\nu\left(T_X\right)\right\|_{\mathscr{B}(\mathcal{H})} \le \sup_{z \in [0, \kappa^{2}]} (z + \nu^{-1})^{\frac{1}{2}} \psi_{\nu}(z) \le \sup_{z \in [0, \kappa^{2}]} (z^{\frac{1}{2}} + \nu^{-\frac{1}{2}}) \psi_{\nu}(z).
    \end{displaymath}
    Using the property of filter function (\ref{prop2}), we have 
    \begin{equation}\label{1.8-8}
        \sup_{z \in [0, \kappa^{2}]} (z^{\frac{1}{2}} + \nu^{-\frac{1}{2}}) \psi_{\nu}(z) \le F_{\tau}\nu^{-\frac{1}{2}} + \nu^{-\frac{1}{2}} F_{\tau} \le 2 F_{\tau}\nu^{-\frac{1}{2}}.
    \end{equation}
    Furthermore, using the property of filter function \eqref{prop1} and recalling that
    \begin{displaymath}
      f_{\nu} = \varphi_\nu(T) S_{k}^{*} L_{k}^{\frac{s}{2}} g_{0} = \varphi_\nu(T) T^{\frac{s}{2}} S_{k}^{*} g_{0},
  \end{displaymath}
  for some $g_{0} \in L^{2}$ with $\|g_{0}\|_{L^{2}} \le R$, we have 
  \begin{align}\label{58-2}
    \left\|T_{X \nu}^{\frac{1}{2}} \psi_\nu\left(T_X\right) f_\nu\right\|_{\mathcal{H}} & \leq\left\|T_{X \nu}^{\frac{1}{2}} \psi_\nu\left(T_X\right)\right\|_{\mathscr{B}(\mathcal{H})}  \left\|\varphi_\nu(T) T^{\frac{s}{2}} S_{k}^{*} g_{0}\right\|_{\mathcal{H}} \notag \\
    &\le 2 F_\tau \nu^{-\frac{1}{2}} \cdot \left\|\varphi_\nu(T) T^{\frac{s}{2}} S_{k}^{*} \right\| \left\|g_{0}  \right\|_{L^{2}} \notag \\
    &= 2 F_\tau \nu^{-\frac{1}{2}} \cdot \left\|\varphi_\nu(T) T^{\frac{s}{2}} T^{\frac{1}{2}} \right\| \left\|g_{0}  \right\|_{L^{2}} \notag \\
    &\le 2 F_\tau \nu^{-\frac{1}{2}} \cdot \left\|T^{\frac{1 + s}{2}} \varphi_{\nu}(T)\right\|_{\mathscr{B}(\mathcal{H})}  \left\|g_{0}  \right\|_{L^{2}}  \notag \\
    &\le 2 F_\tau \nu^{-\frac{1}{2}} E \nu^{\frac{1-s}{2}} R \notag \\
    &= 2 F_{\tau} E R \nu^{-\frac{s}{2}}.
  \end{align}
  
  \item $1 \le s \le 2$: We can rewrite \eqref{step3 goal} as follows,
    \begin{align}\label{1.8-1}
        \left\|T_{X \nu}^{\frac{1}{2}} \psi_\nu\left(T_X\right) f_\nu\right\|_{\mathcal{H}} &= \left\|T_{X \nu}^{\frac{1}{2}} \psi_\nu\left(T_X\right) \varphi_\nu(T) T^{\frac{s}{2}} S_{k}^{*} g_{0} \right\|_{\mathcal{H}} \notag \\
        &= \left\|T_{X \nu}^{\frac{1}{2}} \psi_\nu\left(T_X\right) \varphi_\nu(T) T^{\frac{s}{2}} S_{k}^{*}\right\| \left\| g_{0} \right\|_{L^{2}} \notag \\
        &= \left\|T_{X \nu}^{\frac{1}{2}} \psi_\nu\left(T_X\right) \varphi_\nu(T) T^{\frac{s}{2}} T^{\frac{1}{2}}\right\| \left\| g_{0} \right\|_{L^{2}} \notag \\
        & \leq\left\|T_{X \nu}^{\frac{1}{2}} \psi_\nu\left(T_X\right)  \varphi_\nu(T) T^{\frac{s+1}{2}} \right\| R.
    \end{align}
    Next, we can further decompose \eqref{1.8-1} as follows
    \begin{align}\label{1.8-2}
        \left\|T_{X \nu}^{\frac{1}{2}} \psi_\nu\left(T_X\right) \varphi_\nu(T) T^{\frac{s+1}{2} }\right\| &=\left\|T_{X \nu}^{\frac{1}{2}} \psi_\nu\left(T_X\right) T_{X \nu}^{\frac{s-1}{2}} \cdot T_{X \nu}^{-\frac{s-1}{2}} T_\nu^{\frac{s-1}{2}} \cdot T_\nu^{-\frac{s-1}{2}} T^{\frac{s-1}{2}} \cdot T^{-\frac{s-1}{2}} \varphi_\nu(T) T^{\frac{s+1}{2}}\right\| \notag \\
        &= \left\|T_{X \nu}^{\frac{1}{2}} \psi_\nu\left(T_X\right) T_{X \nu}^{\frac{s-1}{2}}\right\|\left\|T_{X \nu}^{-\frac{s-1}{2}} T_\nu^{\frac{s-1}{2}}\right\|\left\|T_\nu^{-\frac{s-1}{2}} T^{\frac{s-1}{2}}\right\|\left\|T^{-\frac{s-1}{2}} \varphi_\nu(T) T^{\frac{s+1}{2}}\right\| \notag \\
        &= \left\|T_{X \nu}^{\frac{s}{2}} \psi_\nu\left(T_X\right) \right\| \left\|T_{X \nu}^{-\frac{s-1}{2}} T_\nu^{\frac{s-1}{2}}\right \|  \left\|T_\nu^{-\frac{s-1}{2}} T^{\frac{s-1}{2}}\right\|\left\|T \varphi_\nu(T) \right\|.
    \end{align}
    Next, we need to bound the four terms in \eqref{1.8-2}. For the first term in \eqref{1.8-2}, using the inequality $ (a+b)^{p} \le a^{p} + b^{p}$ for $ p \in [0,1]$ again, we have
    \begin{equation}\label{1.8-3}
        \left\|T_{X \nu}^{\frac{s}{2}} \psi_\nu\left(T_X\right)\right\|_{\mathscr{B}(\mathcal{H})} \le \sup_{z \in [0, \kappa^{2}]} (z + \nu^{-1})^{\frac{s}{2}} \psi_{\nu}(z) \le \sup_{z \in [0, \kappa^{2}]} (z^{\frac{s}{2}} + \nu^{-\frac{s}{2}}) \psi_{\nu}(z) \le 2 F_{\tau} \nu^{-\frac{s}{2}}.
    \end{equation}
    For the second term in \eqref{1.8-2}, using Lemma \ref{cordes} and \eqref{3bounds-1}, we have,
    \begin{equation}\label{1.8-4}
        \left\|T_{X \nu}^{-\frac{s-1}{2}} T_\nu^{\frac{s-1}{2}}\right\| \le \left\|T_{X \nu}^{-\frac{1}{2}} T_\nu^{\frac{1}{2}}\right\|^{s-1} \le 3^{s-1} \le 3.
    \end{equation}
    For the third term in \eqref{1.8-2}, 
    \begin{equation}\label{1.8-5}
        \left\|T_\nu^{-\frac{s-1}{2}} T^{\frac{s-1}{2}}\right\| = \sup_{i \in N} \left(\frac{\lambda_{i}}{\lambda_{i} + \nu^{-1}}\right)^{\frac{s-1}{2}} \le 1.
    \end{equation}
    For the fourth term in \eqref{1.8-2}, using the property of filter function \eqref{prop1}, we have
    \begin{equation}\label{1.8-6}
        \left\|T \varphi_\nu(T) \right\| \le E.
    \end{equation}
    Plugging \eqref{1.8-3} \eqref{1.8-4} \eqref{1.8-5} \eqref{1.8-6} into \eqref{1.8-2}, we obtain the bound
    \begin{align}\label{condition s 2}
        \left\|T_{X \nu}^{\frac{1}{2}} \psi_\nu\left(T_X\right) f_\nu\right\|_{\mathcal{H}} \le 6 E F_{\tau} R \nu^{-\frac{s}{2}}.
    \end{align}
    
    \item $s > 2$: Recalling \eqref{1.8-1}, we have
    \begin{align}\label{3.10-4}
        \left\|T_{X \nu}^{\frac{1}{2}} \psi_\nu\left(T_X\right) f_\nu\right\|_{\mathcal{H}} & \leq\left\|T_{X \nu}^{\frac{1}{2}} \psi_\nu\left(T_X\right)  \varphi_\nu(T) T^{\frac{s+1}{2}} \right\| R \notag \\
        &\le \left\|T_{X \nu}^{\frac{1}{2}} \psi_\nu\left(T_X\right) T^{\frac{s-1}{2}} \right\| \left\|T \varphi_\nu(T)  \right\| R\notag \\
        &\le \left\|T_{X \nu}^{\frac{1}{2}} \psi_\nu\left(T_X\right) T^{\frac{s-1}{2}} \right\| E R. 
    \end{align}
    Further, we can have the following decomposition
    \begin{displaymath}
        T_{X \nu}^{\frac{1}{2}} \psi_\nu\left(T_X\right) T^{\frac{s-1}{2}}=T_{X \nu}^{\frac{1}{2}} \psi_\nu\left(T_X\right)\left(T^{\frac{s-1}{2}}-T_X^{\frac{s-1}{2}}\right) + T_{X \nu}^{\frac{1}{2}} \psi_\nu\left(T_X\right) T_X^{\frac{s-1}{2}}.
    \end{displaymath}
    So we have 
    \begin{equation}\label{1.8-7}
        \left\|T_{X \nu}^{\frac{1}{2}} \psi_\nu\left(T_X\right) T^{\frac{s-1}{2}} \right\| \le \left\| T_{X \nu}^{\frac{1}{2}} \psi_\nu\left(T_X\right) \right\| \left\| T^{\frac{s-1}{2}}-T_X^{\frac{s-1}{2}} \right\| + \left\| T_{X \nu}^{\frac{1}{2}} \psi_\nu\left(T_X\right) T_X^{\frac{s-1}{2}} \right\|.
    \end{equation}
    For the first term in \eqref{1.8-7}, using Lemma \ref{lemma phi operator} and the fact that $ \left\|T_X\right\|,\|T\| \leq \kappa^2$, we have
    \begin{equation}\label{1.8-9}
      \left\|T^{\frac{s-1}{2}}-T_X^{\frac{s-1}{2}}\right\| \leq \begin{cases}\left\|T-T_X\right\|^{\frac{s-1}{2}} & s \in(2,3], \\ \frac{s-1}{2} \kappa^{s-3}\left\|T-T_X\right\| & s \geq 3 .\end{cases}
    \end{equation}
    In addition, \eqref{3bounds-3} shows that
    \begin{equation}\label{1.8-10}
      \left\|T_X-T\right\| \leq\left\|T_X-T\right\|_2 \leq \frac{8 \sqrt{2} \kappa^2}{\sqrt{n}} \ln \frac{6}{\delta}.
    \end{equation}
    Further, recalling \eqref{1.8-8}, we have 
    \begin{equation}\label{1.8-11}
        \left\| T_{X \nu}^{\frac{1}{2}} \psi_\nu\left(T_X\right) \right\| \le 2 F_{\tau} \nu^{-\frac{1}{2}}.
    \end{equation}
    In addition, similarly as \eqref{1.8-3}, we have
    \begin{align}\label{1.8-12}
       \left\| T_{X \nu}^{\frac{1}{2}} \psi_\nu\left(T_X\right) T_X^{\frac{s-1}{2}} \right\| &\le \left\| T_{X}^{\frac{1}{2}} \psi_\nu\left(T_X\right) T_X^{\frac{s-1}{2}} \right\| + \nu^{-\frac{1}{2}} \left\| \psi_\nu\left(T_X\right) T_X^{\frac{s-1}{2}} \right\| \notag \\
       &= \left\| T_{X}^{\frac{s}{2}} \psi_\nu\left(T_X\right) \right\| + \nu^{-\frac{1}{2}} \left\|T_X^{\frac{s-1}{2}}  \psi_\nu\left(T_X\right) \right\| \notag \\
       &\le F_{\tau} \nu^{-\frac{s}{2}} + \nu^{-\frac{1}{2}} F_{\tau} \nu^{\frac{1-s}{2}} \notag \\
       &= 2 F_{\tau} \nu^{-\frac{s}{2}}.
    \end{align}
    To sum up, denote 
    \begin{displaymath}
      \Delta_0 := 2 E F_\tau R \nu^{-\frac{1}{2}} \kappa^{s-1} \cdot \begin{cases}n^{-\frac{s-1}{4}}\left(8 \sqrt{2} \ln \frac{6}{\delta}\right)^{\frac{s-1}{2}} & s \in(2,3], \\ n^{-\frac{1}{2}} \cdot \frac{s-1}{2} \cdot 8 \sqrt{2} \ln \frac{6}{\delta}, & s \geq 3 .\end{cases}
    \end{displaymath}
    Then plugging \eqref{1.8-9} $\sim$ \eqref{1.8-12} into \eqref{1.8-7} and use \eqref{3.10-4}, we have 
    \begin{align}
        \left\|T_{X \nu}^{\frac{1}{2}} \psi_\nu\left(T_X\right) f_\nu\right\|_{\mathcal{H}} \le \Delta_0 + 2 E F_{\tau} R \nu^{-\frac{s}{2}}. \notag
    \end{align}
    Without loss of generality, we assume that $\ln \frac{6}{\delta} >1 $. Simple calculation shows that, 
    \begin{equation}\label{Delta1}
      \Delta_0 \leq 32 \max \left(\frac{s-1}{2}, 1\right) E F_\tau R \kappa^{s-1} \nu^{-\frac{1}{2}} n^{-\frac{\min (s, 3)-1}{4}} \ln \frac{6}{\delta} := \Delta_{1} .
    \end{equation}
    Then we have 
    \begin{align}\label{condition s 3}
        \left\|T_{X \nu}^{\frac{1}{2}} \psi_\nu\left(T_X\right) f_\nu\right\|_{\mathcal{H}} \le \Delta_1 + 2 E F_{\tau} R \nu^{-\frac{s}{2}}.
    \end{align}
\end{itemize}

Combining the bounds of three conditions of $s$, i.e., \eqref{58-2} \eqref{condition s 2} \eqref{condition s 3}, we finally bound the goal of Step 3, i.e., \eqref{step3 goal} by
\begin{displaymath}
    \left\|T_{X \nu}^{\frac{1}{2}} \psi_\nu\left(T_X\right) f_\nu\right\|_{\mathcal{H}} \le 6 F_{\tau} E R \nu^{-\frac{s}{2}} + \Delta_1 I_{s > 2}.
\end{displaymath}

$ $\\
\textbf{Step 4}: Now we are able to use the results of Step1 $\sim$ Step3 to finish the proof of the estimation error. Still, we consider two cases, $s +\frac{1}{\beta} > \alpha_{0} $ and $s +\frac{1}{\beta} \le \alpha_{0} $. 

\begin{itemize}
    \item $s +\frac{1}{\beta} > \alpha_{0}: $ Plugging the results of Step2 and Step3 into \eqref{third decompose} and using the decomposition \eqref{proof of 3.1-1}, by choosing $ \nu \asymp n^{\frac{ \beta}{s \beta + 1}}$, we have
    \begin{align}
        \left\|\hat{f}_\nu-f_\nu\right\|_{[\mathcal{H}]^{\gamma}} &\le 3 \nu^{\frac{\gamma}{2}} \left( \ln(\frac{6}{\delta}) C n^{-\frac{1}{2} \frac{s \beta}{s \beta +1}} + 6 F_{\tau} E R \nu^{-\frac{s}{2}}  + \Delta_{1} I_{s > 2} \right) \notag \\
        &= 3 n^{\frac{1}{2} \frac{\gamma \beta}{s \beta +1}} \left( \ln(\frac{6}{\delta}) C n^{-\frac{1}{2} \frac{s \beta}{s \beta +1}} + 6 F_{\tau} E R n^{-\frac{1}{2} \frac{s \beta}{s \beta +1}} + \Delta_{1} I_{s > 2} \right). \notag
    \end{align}
    Recalling the expression of $\Delta_{1}$ in \eqref{Delta1}, when $ 2 < s \le 3$,
    \begin{displaymath}
        \Delta_{1} \asymp n^{-\frac{r_{0}}{2}},
    \end{displaymath}
    where 
    \begin{displaymath}
        r_{0} = \frac{\beta}{s \beta + 1} + \frac{s-1}{2}.
    \end{displaymath}
    Since $ s > 2$ implies $ s + \frac{1}{\beta} > 2$, so we have
    \begin{displaymath}
        r_{0} - \frac{s \beta }{s \beta +1} = \frac{s-1}{2} - \frac{s-1}{s +\frac{1}{\beta}} > 0.
    \end{displaymath}
    So we have $\Delta_{1} \lesssim n^{-\frac{1}{2} \frac{s \beta}{s \beta + 1}}$.
    
    When $s > 3$, we also have $r_{0} = \frac{\beta}{s \beta + 1} + 1 > \frac{s \beta }{s \beta +1}$. Therefore, we know that 
    \begin{displaymath}
        \Delta_{1} I_{s > 2} \le C \ln\frac{6}{\delta} n^{-\frac{1}{2} \frac{s \beta}{s \beta + 1}}.
    \end{displaymath}
    To sum up, we prove that when $s +\frac{1}{\beta} > \alpha_{0}$, the estimation error satisfies that
    \begin{equation}\label{final esti error-1}
        \left\|\hat{f}_\nu-f_{\nu}\right\|_{[\mathcal{H}]^{\gamma}} \leq \ln \frac{6}{\delta} C n^{-\frac{1}{2} \frac{(s-\gamma) \beta}{s \beta+1}}.
    \end{equation}
    
    \item $s + \frac{1}{\beta} \le \alpha_{0}:$ In this case, $s \le 1$. Similarly, for some fixed $\alpha_{0} < \alpha \le 1$, by choosing $\nu \asymp (\frac{n}{\ln^{r}(n)})^{\frac{1}{\alpha}}$, we have 
    \begin{align}\label{final esti error-2}
        \left\|\hat{f}_\nu-f_\nu\right\|_{[\mathcal{H}]^{\gamma}} &\le 3 \nu^{\frac{\gamma}{2}} \left( \ln{\frac{6}{\delta}} C \left(\frac{n}{ \ln^{r}(n)}\right)^{-\frac{1}{2} \frac{s}{\alpha}}  + 6 F_{\tau} E R \nu^{-\frac{s}{2}} \right) \notag \\
        &= 3 \left(\frac{n}{ \ln^{r}(n)}\right)^{\frac{\gamma}{2 \alpha} } \left( \ln{\frac{6}{\delta}} C \left(\frac{n}{ \ln^{r}(n)}\right)^{-\frac{1}{2} \frac{s}{\alpha}} + 6 F_{\tau} E R \left(\frac{n}{ \ln^{r}(n)}\right)^{-\frac{1}{2} \frac{s}{\alpha}} \right) \notag \\
        &\le \ln \frac{6}{\delta} C \left(\frac{n}{\ln ^r(n)}\right)^{-\frac{s-\gamma}{2 \alpha}}.
    \end{align}
\end{itemize}
Then, the proof of Theorem \ref{estimation error thm} follows from \eqref{final esti error-1} and \eqref{final esti error-2}.
\end{proof}

\paragraph{Proof of Theorem \ref{main theorem}}
We first decompose the $[\mathcal{H}]^{\gamma}$-norm generalization error into two terms, which are often referred to as the approximation error and the estimation error:
\begin{equation}\label{error decompose}
    \left\|\hat{f}_\nu-f_{\rho}^{*}\right\|_{[\mathcal{H}]^{\gamma}} = \left\|f_\nu-f_{\rho}^{*}\right\|_{[\mathcal{H}]^{\gamma}} + \left\|\hat{f}_\nu-f_{\nu}\right\|_{[\mathcal{H}]^{\gamma}}.
\end{equation}
For the approximation error, Theorem \ref{theorem of approximation error} proves that
\begin{itemize}
    \item By choosing $ \nu \asymp n^{\frac{ \beta}{s \beta + 1}}$, 
    \begin{equation}\label{plug appr 1}
        \left\|f_\nu-f_{\rho}^{*}\right\|_{[\mathcal{H}]^{\gamma}} \le F_{\tau} R n^{- \frac{1}{2}\frac{(s-\gamma) \beta}{s \beta+1}};
    \end{equation}
    \item by choosing $\nu \asymp (\frac{n}{\ln^{r}(n)})^{\frac{1}{\alpha}}$, for some $r > 1$,
    \begin{equation}\label{plug appr 2}
        \left\|f_\nu-f_{\rho}^{*}\right\|_{[\mathcal{H}]^{\gamma}} \le F_{\tau} R \left(\frac{n}{\ln ^r(n)}\right)^{-\frac{s-\gamma}{2 \alpha}}.
    \end{equation}
\end{itemize}
Then the proof follows from plugging \eqref{plug appr 1}, \eqref{plug appr 2} and the bounds of estimation error in Theorem \ref{estimation error thm} into \eqref{error decompose}.

\subsection{Lower bound}
The following lemma is a standard approach to derive the minimax lower bound, which can be found in \citet[Theorem 2.5]{tsybakov2009_IntroductionNonparametric}. 
\begin{lemma}\label{lower prop from tsy}
Suppose that there is a non-parametric class of functions $ \Theta$ and a (semi-)distance $d(\cdot,\cdot)$ on $ \Theta$. $\left\{ P_{\theta}, \theta \in \Theta \right\}$ is a family of probability distributions indexed by $\Theta$. Assume that $M \ge 2$ and suppose that $ \Theta$ contains elements $ \theta_0, \theta_1, \cdots, \theta_M$ such that, 
\begin{itemize}
    \item[(1)] $ d\left(\theta_j, \theta_k\right) \geq 2 s>0, \quad \forall 0 \leq j<k \leq M$;
    \item[(2)] $P_j \ll P_0, \quad \forall j=1, \cdots, M$, and 
    \begin{displaymath}
        \frac{1}{M } \sum_{j=1}^M K\left(P_j, P_0\right) \leq a \log M,
    \end{displaymath}
\end{itemize}
    with $ 0<a<1 / 8$ and $ P_j=P_{\theta_j}, j=0,1, \cdots, M$. Then
    \begin{displaymath}
    \inf _{\hat{\theta}} \sup _{\theta \in \Theta} P_\theta(d(\hat{\theta}, \theta) \geq s) \geq \frac{\sqrt{M}}{1+\sqrt{M}}\left(1-2 a-\sqrt{\frac{2 a}{\log M}}\right).
    \end{displaymath}
\end{lemma}

\begin{lemma}\label{lemma of KL}
   Suppose that $\mu$ is a distribution on $\mathcal{X}$ and $f_{i} \in L^{2}(\mathcal{X},\mu)$. Suppose that
   \begin{displaymath}
       y=f_i(x)+\epsilon, \quad i=1,2,
   \end{displaymath}
   where $\epsilon \sim \mathcal{N}(0,\sigma^{2})$ are independent Gaussian random error. Denote the two corresponding distributions on $ \mathcal{X} \times \mathcal{Y}$ as $ \rho_{i}, i=1,2$. The KL divergence of two probability distributions on $\Omega$ is 
   \begin{displaymath}
       K\left(P_1, P_2\right) \coloneqq \int_{\Omega} \log \left(\frac{\mathrm{d} P_1}{\mathrm{~d} P_2}\right) \mathrm{d} P_1,
   \end{displaymath}
   if $P_1 \ll P_2$ and otherwise $K\left(P_1, P_2\right) \coloneqq \infty $.
   Then we have 
   \begin{displaymath}
       \mathrm{KL}\left(\rho_1^n, \rho_2^n\right)=n \mathrm{KL}\left(\rho_1, \rho_2\right)=\frac{n}{2 \sigma^2}\left\|f_1-f_2\right\|_{L^2(\mathcal{X}, d \mu)}^2,
   \end{displaymath}
   where $ \rho_{i}^{n} $ denotes the independent product of $n$ distributions $\rho_{i}, i=1,2$.
\end{lemma}
\begin{proof}
The lemma directly follows from the definition of KL divergence and the fact that 
\begin{displaymath}
    \mathrm{KL}\left(N\left(f_1(x), \sigma^2\right), N\left(f_2(x), \sigma^2\right)\right) = \frac{1}{2 \sigma^2}\left|f_1(x)-f_2(x)\right|^2.
\end{displaymath}

\end{proof}

The following lemma is a result from \citet[Lemma 2.9]{tsybakov2009_IntroductionNonparametric}
\begin{lemma}\label{lemma of ham}
   Denote $\Omega=\left\{\omega=\left(\omega_1, \cdots, \omega_m\right), \omega_i \in\{0,1\}\right\}=\{0,1\}^m$. Let $m\ge 8$, there exists a subset $\left\{\omega^{(0)}, \cdots, \omega^{(M)}\right\} $ of ~$ \Omega$ such that $\omega^{(0)}=(0, \cdots, 0)$,
   \begin{displaymath}
       d_{\text {Ham }}\left(\omega^{(i)}, \omega^{(j)}\right) \coloneqq \sum_{k=1}^m\left|\omega_k^{(i)}-\omega_k^{(j)}\right| \geq \frac{m}{8}, \quad \forall 0 \leq i<j \leq M,
   \end{displaymath}
   and $M \geq 2^{m / 8}$.
\end{lemma}

Now we are ready to prove the minimax lower bound given by Theorem \ref{prop information lower bound}.
\paragraph{Proof of Theorem \ref{prop information lower bound}}
We will construct a family of probability distributions on $ \mathcal{X} \times \mathcal{Y}$ and apply Lemma \ref{lower prop from tsy}. Recall that $\mu$ is a probability distribution on $\mathcal{X}$ such that Assumption \ref{ass EDR} is satisfied. Denote the class of functions 
\begin{displaymath}
    B^{s}(R)=\left\{f \in[\mathcal{H}]^s: \|f\|_{[\mathcal{H}]^{s}} \leq R\right\},
\end{displaymath}
and for every $f \in B^{s}(R)$, define the probability distribution $\rho_{f}$ on $\mathcal{X} \times \mathcal{Y}$ such that
\begin{displaymath}
    y = f(x) + \epsilon, ~~ x \sim \mu, 
\end{displaymath}
where $\epsilon \sim \mathcal{N}(0,\bar{\sigma}^{2})$ and $\bar{\sigma} = \min(\sigma, L) $. It is easy to show that such $\rho_{f}$ falls into the family $\mathcal{P}$ in Lemma \ref{prop information lower bound}. (Assumption \ref{ass EDR} and \ref{ass source condition} are satisfied obviously. Assumption \ref{ass mom of error} follows from results of moments of Gaussian random variables, see, e.g., \citet[Lemma 21]{fischer2020_SobolevNorm}).

Using Lemma \ref{lemma of ham}, for $m = n^{\frac{1}{s\beta + 1}}$, there exists $ \omega^{(0)}, \cdots, \omega^{(M)} \in \{0,1\}^{m}$ for some $M \ge 2^{m/8} $ such that
\begin{equation}\label{proof lower-2}
    \sum_{k=1}^m\left|\omega_k^{(i)}-\omega_k^{(j)}\right| \geq \frac{m}{8}, \quad \forall 0 \leq i<j \leq M.
\end{equation}
For $\epsilon = C_{0} m^{- (s-\gamma)\beta - 1}$, define the functions $ f_i, i=1,2,\cdots, M $ as 
\begin{displaymath}
    f_i:=\epsilon^{1 / 2} \sum_{k=1}^m \omega_k^{(i)} \lambda_{m+k}^{\frac{\gamma}{2}} e_{m+k}.
\end{displaymath}
Since
\begin{align}\label{proof lower-1}
    \left\|f_i\right\|_{[\mathcal{H}]^{s}} &= \epsilon \sum_{k=1}^m \lambda_{m+k}^{\gamma-s}\left(\omega_k^{(i)}\right)^2  \notag
    \le \epsilon \sum_{k=1}^m \lambda_{2m}^{\gamma-s} \\ &\leq 2^{(s-\gamma)\beta} c \epsilon \sum_{k=1}^m m^{(s-\gamma) \beta} \le 2^{(s-\gamma)\beta} c \epsilon m^{(s-\gamma) \beta+1} = 2^{(s-\gamma)\beta} c C_{0},
\end{align}
Where $c$ in \eqref{proof lower-1} only depends on the constants in Assumption \ref{ass EDR}. So if $C_{0}$ is small such that 
\begin{equation}\label{C0-1}
    2^{(s-\gamma)\beta} c C_{0} \le R, 
\end{equation}
then we have  $f_{i} \in B^{s}(R), i=1,2,\cdots,M.$

Using Lemma \ref{lemma of KL}, we have 
\begin{align}
    \mathrm{KL}\left(\rho_{f_{i}}^n, \rho_{f_{0}}^n\right) &=\frac{n}{2 \bar{\sigma}^{2}}\left\|f_i\right\|_{L^2(\mathcal{X}, \mu)}^2 \notag \\
    &=\frac{n \epsilon}{2 \bar{\sigma}^{2}} \sum_{k=1}^m \lambda_{m+k}^{\gamma}\left(\omega_k^{(i)}\right)^2 \notag \\
    &\leq \frac{n \epsilon C m^{-\gamma \beta + 1} }{2 \bar{\sigma}^{2}}  = \frac{ n }{2 \bar{\sigma}^{2}} C C_{0} m^{-s\beta}, \notag
\end{align}
Where $C$ only depends on the constants in Assumption \ref{ass EDR}.
Recall that $ M \ge 2^{m/8}$ implies $ \ln M \geq \frac{\ln 2}{8} m$. For a fixed $a\in(0,\frac{1}{8})$, since $m = n^{\frac{1}{s \beta +1}}$, letting
\begin{equation}\label{proof 2.8-1}
  \mathrm{KL}\left(\rho_{f_{i}}^n, \rho_{f_{0}}^n\right) \le \frac{ n }{2 \bar{\sigma}^{2}}C C_{0} m^{-s\beta} \leq a \frac{\ln 2}{8} m \le a \ln M,
\end{equation}
we have 
\begin{equation}\label{C0-2}
    C_{0} \le \frac{\bar{\sigma}^{2} \ln 2 }{4C} a.
\end{equation}
So we can choose $C_{0} = c^{\prime} a$ such that \eqref{C0-1} and \eqref{C0-2} are satisfied, where $c^{\prime}$ only depends on the constants in Assumption \ref{ass EDR}.

Denote $ \left\{ \rho_{f_{i}}^n , f_{i} \in B^{s}(R)\right\}$ as a family of probability distribution index by $ f_{i} $, then \eqref{proof 2.8-1} implies the second condition in Lemma \ref{lower prop from tsy} holds. Further, using \eqref{proof lower-2}, we have 
\begin{equation}\label{proof 2.8-2}
    d \left(f_i, f_j\right)^2=\left\|f_i-f_j \right\|_{[\mathcal{H}]^{\gamma}}^2=\epsilon \sum_{k=1}^m\left(\omega_{k}^{(i)} - \omega_{k}^{(j)} \right)^2 \geq \frac{\epsilon m}{8}=\frac{c^{\prime}a}{8} m^{- (s-\gamma) \beta} \geq c^{\prime} a n^{-\frac{(s-\gamma) \beta}{s \beta + 1}},
\end{equation}
where $c^{\prime}$ only depends on the constants in Assumption \ref{ass EDR}.

Applying Lemma \ref{lower prop from tsy} to \eqref{proof 2.8-1} and \eqref{proof 2.8-2}, we have
\begin{equation}\label{proof lower-3}
\inf _{\hat{f}_n} \sup _{f \in B^s(R)} \mathbb{P}_{\rho_f}\left\{\left\|\hat{f}_n-f\right\|_{[\mathcal{H}]^{\gamma}}^2 \geq c^{\prime} a n^{-\frac{(s-\gamma) \beta}{s \beta + 1}}\right\} \geq \frac{\sqrt{M}}{1+\sqrt{M}}\left(1-2 a-\sqrt{\frac{2 a}{\ln M}}\right).
\end{equation}
When $n$ is sufficiently large so that $M$ is sufficiently large, the probability in the R.H.S. of \eqref{proof lower-3} is larger than $ 1- 3a $. For $\delta \in (0,1)$, choose $ a = \frac{\delta}{3}$, without loss of generality we assume $ a \in (0,\frac{1}{8})$. Then \eqref{proof lower-3} shows that there exists a constant $C$ only depends on the constants in Assumption \ref{ass EDR}, for all estimator $\hat{f}, $ we can find a function $f \in B^{s}(R)$ and the corresponding distribution $\rho_{f} \in \mathcal{P}$ such that, with probability at least $1-\delta$,
\begin{displaymath}
    \left\| \hat{f} - f \right\|_{[\mathcal{H}]^{\gamma}}^{2} \ge  C \delta n^{-\frac{(s-\gamma) \beta}{s \beta +1}}.
\end{displaymath}
So we finish the proof. (In fact, it can be argued that the constant $C$ only depends on the constants in \ref{ass EDR}, in dependent of $s$).

\subsection{Shift-invariant kernels}\label{section proof invariant}
Let $\mu$ be the uniform measure on $[-\pi,\pi)^d$.
It is well known that the Fourier basis
\begin{align*}
    \phi_{\bm{m}}(x) \coloneqq \exp(i \langle \bm{m},x \rangle) 
\end{align*}
are orthonormal in $L^2([-\pi,\pi)^d,\mu)$:
\begin{align*}
    \int_{[-\pi,\pi)^d} \phi_{\bm{m}}(x) \phi_{\bm{m}'}(x) \mathrm{d}\mu(x) = 
    \frac{1}{(2\pi)^d} \int_{[-\pi,\pi)^d} \phi_{\bm{m}}(x) \phi_{\bm{m}'}(x) \mathrm{d}x = \bm{1}_{\{\bm{m} = \bm{m}'\}}.
\end{align*}
Now suppose $k$ is a kernel on $[-\pi,\pi)^d$ satisfying 
\begin{align*}
    k(x,y) = g\big( (x - y) \bmod [-\pi,\pi)^d\big).
\end{align*}
Then, noticing that $\phi_{\bm{m}}(x)$ is periodic, we have
\begin{align*}
    \int_{[-\pi,\pi)^d} k(x,y) \phi_{\bm{m}}(x) \mathrm{d}\mu(x) &= 
    \int_{[-\pi,\pi)^d} g\big( (x - y) \bmod [-\pi,\pi)^d\big)\exp(i \langle \bm{m},x \rangle)  \mathrm{d}\mu(x) \\
    &= \int_{[-\pi,\pi)^d} g\big( z \big)\exp(i \langle \bm{m}, y+z \rangle)  \mathrm{d}\mu(z) \\
    &= \exp(i \langle \bm{m}, y \rangle) \int_{[-\pi,\pi)^d} g\big( z \big)\exp(i \langle \bm{m}, z \rangle)  \mathrm{d}\mu(z).
\end{align*}
It shows that $\phi_{\bm{m}}(x)$ is an eigenfunction of the integral operator $T$ associated with $k$.
Since $|{\phi_{\bm{m}}(x)}| \leq 1,$ that is, the eigenfunctions are uniformly bounded, we conclude that the embedding index $\alpha_0 = \frac{1}{\beta}$.

\subsection{Spherical harmonics and dot-product kernels}\label{section proof sphere}

Let us consider the unit $d$-sphere $\mathbb{S}^d = \{ x \in \mathbb{R}^{d+1} ~|~ \|x\| = 1 \}$ and denote by $\sigma$ the uniform measure on $\mathbb{S}^d$.
The eigen-system of spherical Laplacian $\Delta_{\mathbb{S}^d}$ yields an orthogonal decomposition
\begin{align*}
    L^2(\mathbb{S}^d, \sigma) = \bigoplus_{n = 0}^\infty \mathcal{H}_n(\mathbb{S}^d),
\end{align*}
where $\mathcal{H}_n(\mathbb{S}^d)$ is the subspace of homogenenous harmonic polynomials of degree $n$ and each $Y_n \in \mathcal{H}_n(\mathbb{S}^d)$ is an eigenfunction of $\Delta_{\mathbb{S}^d}$ corresponding to eigenvalue $-n(n+d-1)$.
In particular, we can take an orthonormal basis 
\begin{align*}
    \{Y_{n,l}, l = 1,\dots,a_n,~n=0,1,\dots\},
\end{align*}
where $Y_{n,l} \in \mathcal{H}_n(\mathbb{S}^d)$ and 
\begin{align*}
  a_n \coloneqq \dim \mathcal{H}_n(\mathbb{S}^d) = \binom{n+d}{n} - \binom{n-2+d}{n-2}.
\end{align*}
Such an orthonormal basis is often referred to as the \textit{spherical harmonics}.
Although the specific choice of $Y_{n,l}$ can vary, the sum
\begin{displaymath}
    Z_n(x,y) = \sum_{l=1}^{a_n} Y_{n,l}(x)Y_{n,l}(y)
\end{displaymath}
is invariant. 
Moreover, $Z_n(x,y)$ depends only on $\langle x,y \rangle$ and satisfies~\citep[Corollary 1.2.7]{dai2013_ApproximationTheory}
\begin{displaymath}
    |Z_n(x,y)| \leq Z_n(x,x) = a_n,\quad \forall x,y \in \mathbb{S}^d.
\end{displaymath}

The following Funk-Hecke formula is important, see also \citet[Theorem 1.2.9]{dai2013_ApproximationTheory}.

\begin{theorem}[Funk-Hecke formula]
  \label{thm:FunkFormula}
  Let $d \geq 3$ and $f$ be an integrable function such that $\int_{-1}^1 |f(t)| (1-t^2)^{d/2-1} \mathrm{d} t$ is finite.
  Then for every $Y_n \in \mathcal{H}_n(\mathbb{S}^d)$,
  \begin{align}
    \label{eq:C_FunkHecke}
    \frac{1}{\omega_d}\int_{\mathbb{S}^d} f(\langle{x,y}\rangle) Y_n(y) \mathrm{d} \sigma(y) = \mu_k(f) Y_n(x),\quad \forall x \in \mathbb{S}^{d}, 
  \end{align}
  where $\mu_n(f)$ is a constant defined by
  \begin{align}
    \mu_n(f) = \omega_d \int_{-1}^1 f(t) \frac{C_n^\lambda(t)}{C_n^\lambda(1)} (1-t^2)^{\frac{d-2}{2}} \mathrm{d} t, \notag
  \end{align}
  and $\omega_d$ is the surface area of $\mathbb{S}^d$.
\end{theorem}

Suppose $k$ is a dot-product kernel. 
Recalling the definition of the integral operator $T$ associated with $k$, 
\eqref{eq:C_FunkHecke} shows that elements in $\mathcal{H}_n(\mathbb{S}^d)$, in particular $Y_{n,l}$, are eigenfunctions of $T$.
Therefore, we obtain the following Mercer's decomposition:
\begin{align}
\label{eq:MercerSphere}
  k(x,y) = \sum_{n=0}^{\infty} \mu_n \sum_{l=1}^{a_n} Y_{n,l}(x)Y_{n,l}(y).
\end{align}

\begin{proposition}
    \label{prop:EMBIdx_Sphere}
    Let $k$ be an dot-product kernel satisfying $\mu_n \asymp n^{-d\beta}$ for some $\beta>1$, where $\mu_n$ is defined in \eqref{eq:MercerSphere}.
    Then, the EDR of the corresponding RKHS is $\beta$ and the embedding index $\alpha_0 = \beta^{-1}$.
\end{proposition}
\begin{proof}
    Notice that $\mu_n$ is an eigenvalue of multiplicity $a_n$.
    Then, the eigenvalue decay rate is easily obtained by the estimation $a_n \asymp n^{d-1}$ and $\sum_{r=0}^n a_r \asymp n^d$.
    Considering the equivalent definition of the embedding property~\eqref{eq:EMB_Eigenvalues}, we have
    \begin{align*}
        \sum_{n=0}^{\infty} \mu_n^\alpha \sum_{l=1}^{a_n} Y_{n,l}(x)^2 & = \sum_{n=0}^{\infty} \mu_n^\alpha Z_n(x,x) 
        \leq \sum_{n=0}^{\infty} a_n \mu_n^\alpha \\
        & \leq \sum_{n=0}^{\infty} C n^{d-1} n^{- \alpha d \beta} =  C \sum_{n=0}^{\infty} n^{-1 - d (\alpha \beta - 1)} \\
        & < \infty \quad \text{if} \quad \alpha > \frac{1}{\beta}.
    \end{align*}
\end{proof}




\acks{Lin’s research was supported in part by the National Natural Science Foundation of China (Grant 92370122, Grant 11971257). The authors are very grateful to the anonymous reviewers for the suggestions on improving the
	presentation of this work.}

\newpage

\appendix
\section*{Appendix A.}\label{appendix A}
In this appendix, we introduce some useful results of real interpolation and Lorentz spaces \cite[Chapter 22-26]{tartar2007introduction}.  

\subsection*{\textbf{A.1} Real interpolation and the Reiteration theorem}
We first introduce the definition of real interpolation through the K-method. For two normed spaces $ E_{i}, i=0,1$, denote their norms as $\| \cdot \|_{i}, i=0,1$.

\begin{definition}[K-functional]
   Let $E_{0}$ and $E_{1}$ be two normed spaces, continuously embedded into a topological vector space $\mathcal{E}$ (($E_{0}$, $E_{1}$) is a compatible couple). For $a \in E_{0} + E_{1}$ and $ t >0$, define the K-functional by
   \begin{displaymath}
       K(t ; a)=\inf _{a=a_0+a_1}\left(\left\|a_0\right\|_0+t\left\|a_1\right\|_1\right).
   \end{displaymath}
\end{definition}

\begin{definition}[Real interpolation]\label{def real interpolation}
   Let $E_{0}$ and $E_{1}$ be two normed spaces, continuously embedded into a topological vector space $\mathcal{E}$ (($E_{0}$, $E_{1}$) is a compatible couple). For $0 < \theta < 1$ and $1 \le  p \le \infty$ (or for $\theta =0, 1$ with $p = \infty$), the real interpolation space is defined by
   \begin{displaymath}
      \left(E_0, E_1\right)_{\theta, p}=\left\{a \in E_0+E_1 \mid t^{-\theta} K(t ; a) \in L^p\left(\mathbb{R}^{+} ; \frac{\mathrm{d} t}{t}\right)\right\},
   \end{displaymath}
   with the norm
   \begin{displaymath}
     \|a\|_{\left(E_0, E_1\right)_{\theta, p}}=\left\|t^{-\theta} K(t ; a)\right\|_{L^p(\mathbb{R}^{+} ; \mathrm{d} t / t)}.
   \end{displaymath}
   
\end{definition}

\begin{lemma}\label{mono of RI}
  If $ 0 < \theta < 1$ and $ 1\le p \le q \le \infty $, we have 
  \begin{displaymath}
      \left(E_0, E_1\right)_{\theta, p} \subset \left(E_0, E_1\right)_{\theta, q},~~ \text{with continuous embedding}.
  \end{displaymath}
\end{lemma}
  
The following lemma gives the result of exchanging the two spaces $E_{0},E_{1}$.
\begin{lemma}\label{exchanging e0e1}
  One has $(E_{1},E_{0})_{\theta,p} = (E_{0},E_{1})_{1 - \theta, p}$ for $0 < \theta <1$ and $1 \le p \le \infty$; the same result holds for $\theta =0 ~\text{or}~1$, and $ p =1$ or $ p = \infty$.
\end{lemma}

The following Lions–Peetre Reiteration Theorem is an important property of real interpolation spaces. 
\begin{theorem}[Reiteration theorem]\label{reiteration theorem}
  If $ 0 \le \theta_{0} \neq \theta_{1} \le 1$, and the two normed spaces $F_{0}, F_{1}$ satisfy that
  \begin{align}
      \left(E_0, E_1\right)_{\theta_0, 1} &\subset F_0 \subset\left(E_0, E_1\right)_{\theta_0, \infty}; \notag \\
      \left(E_0, E_1\right)_{\theta_1, 1} &\subset F_1 \subset\left(E_0, E_1\right)_{\theta_1, \infty}. \notag
  \end{align}
  Then for $ 0<\theta<1$ and $1 \le p \le \infty$, denote $ \eta = (1-\theta) \theta_{0} + \theta \theta_{1}$, we have 
  \begin{displaymath}
    \left(F_0, F_1\right)_{\theta, p}=\left(E_0, E_1\right)_{\eta, p}, ~~\text{with equivalent norms}.
  \end{displaymath}
\end{theorem}

\begin{remark}
  This theorem implies that, if we replace $F_{0}$ with any space $\widetilde{F_{0}}$ satisfying $ \left(E_0, E_1\right)_{\theta_0, 1} \subset \widetilde{F_0} \subset\left(E_0, E_1\right)_{\theta_0, \infty} $, the real interpolation space remains `unchanged', i.e., $\left(F_0, F_1\right)_{\theta, p} \cong \left(\widetilde{F_0}, F_1\right)_{\theta, p}$. 
\end{remark}

\subsection*{\textbf{A.2} Lorentz space}
\begin{definition}[Lorentz space]
   For $ 1 < p < \infty$ and $1 \le q \le \infty$, the Lorentz space $ L^{p,q}(\mathcal{X},\mu)$ is defined as 
   \begin{displaymath}
       L^{p,q}(\mathcal{X},\mu) = \left( L^{1}(\mathcal{X},\mu), L^{\infty}(\mathcal{X},\mu)\right)_{\frac{1}{p^{\prime}}, q}, 
   \end{displaymath}
   where $\frac{1}{p^{\prime}} + \frac{1}{p} = 1$.
\end{definition}

Using Lemma \ref{mono of RI}, it is easy to show that $L^{p,q_{1}}(\mathcal{X},\mu) \subseteq L^{p,q_{2}}(\mathcal{X},\mu) $ for $ 1 \le q_{1} \le q_{2} \le \infty$. In addition, the following lemma gives the relation between Lorentz space and $L^{p}$ space.
\begin{lemma}\label{cong of lorentz}
  For $ 1 < p < \infty$, we have
  \begin{displaymath}
      L^{p,p}(\mathcal{X},\mu) \cong L^{p}(\mathcal{X},\mu); \quad L^{p,\infty}(\mathcal{X},\mu) \cong L^{p,w}(\mathcal{X},\mu),
  \end{displaymath}
  where $L^{p,w}(\mathcal{X},\mu)$ denotes the weak $L^{p}$ space.
\end{lemma}

\section*{Appendix B. Auxiliary results} \label{appendix B}
\begin{lemma}\label{basic ineq}
  For any $\lambda > 0$ and $ s \in [0,1]$, we have
  \begin{displaymath}
      \sup _{t \geq 0} \frac{t^s}{t+\lambda} \leq \lambda^{s-1}.
  \end{displaymath}
\end{lemma}

\begin{proof}
  Since $ a^{s} \le a + 1$ for any $a \ge 0$ and $s \in [0,1] $, the lemma follows from
  \begin{displaymath}
      \left(\frac{t}{\lambda}\right)^{s} \le \frac{t}{\lambda} + 1 = \frac{t + \lambda}{\lambda}.
  \end{displaymath}
\end{proof}
\begin{lemma}\label{lemma of effect}
    If $\lambda_i \asymp i^{-\beta}$, we have
    \begin{align}
        \mathcal{N}(\nu) \asymp \nu^{\frac{1}{\beta}}. \notag
    \end{align}

\end{lemma}
\begin{proof}
    Since $c ~i^{-\beta} \leq \lambda_i \leq C i^{-\beta}$, we have
    \begin{align}
        \mathcal{N}(\nu) &= \sum_{i = 1}^{\infty}  \frac{\lambda_i}{\lambda_i + \nu^{-1}}  
        \leq \sum_{i = 1}^{\infty} \frac{C i^{-\beta}}{C i^{-\beta} + \nu^{-1}}  = \sum_{i = 1}^{\infty}  \frac{C }{C+ \nu^{-1} i^{\beta}} \notag \\
        &\leq \int_{0}^{\infty}  \frac{C }{\nu^{-1} x^{\beta} + C}  \mathrm{d} x
        = \nu^{\frac{1}{\beta}} \int_{0}^{\infty}  \frac{C }{y^{\beta} + C} \mathrm{d} y \leq C_{1} \nu^{\frac{1}{\beta}}. \notag
    \end{align}
    for some constant $C_{1}$. Similarly, we can prove 
    \begin{displaymath}
    \mathcal{N}(\nu) \geq C_{2} \nu^{\frac{1}{\beta}},
    \end{displaymath}
    for some constant $C_{2}$.
\end{proof}

The following concentration inequality about self-adjoint Hilbert-Schmidt operator
valued random variables is frequently used in related literature, e.g., \citet[Theorem 27]{fischer2020_SobolevNorm} and \citet[Lemma 26]{lin2020_OptimalConvergence}.
\begin{lemma}\label{lemma concentration of operator}
   Let $(\Omega, \mathcal{B}, P)$ be a probability space, $\mathcal{H}$ be a separable Hilbert space. Suppose that $ A_{1}, \cdots, A_{n}$ are i.i.d. random variables with values in the set of self-adjoint Hilbert-Schmidt operators. If  $\mathbb{E} A_{i} = 0$, and the operator norm $ \| A_{i} \| \le L,P \text{-a.e.}$, and there exists a self-adjoint positive semi-definite trace class operator $V$ with $\mathbb{E} A_{i}^{2} \preceq V $. Then for $\delta \in (0,1)$, with probability at least $1 - \delta$, we have 
   \begin{align}
        \left\| \frac{1}{n}\sum_{i=1}^n A_i \right\|
        \leq \frac{2L\beta}{3n} + \sqrt {\frac{2 \| V \| \beta}{n}},\quad
        \beta = \ln \frac{4 \rm{tr} V}{\delta \| V \|}. \notag
   \end{align}
\end{lemma}

The following Bernstein inequality about vector-valued random variables is frequently used, e.g., \citet[Proposition 2]{Caponnetto2007OptimalRF} and \citet[Theorem 26]{fischer2020_SobolevNorm}.
\begin{lemma}[Bernstein inequality]\label{bernstein}
   Let $(\Omega,\mathcal{B},P)$ be a probability space, $H$ be a separable Hilbert space, and $\xi: \Omega \to H$ be a random variable with 
   \begin{displaymath}
     \mathbb{E}\|\xi\|_H^m \leq \frac{1}{2} m ! \sigma^2 L^{m-2},
   \end{displaymath}
   for all $m>2$. Then for $\delta \in (0,1)$, $\xi_{i}$ are i.i.d. random variables, with probability at least $1 - \delta$, we have
   \begin{displaymath}
       \left\|\frac{1}{n} \sum_{i=1}^n \xi_{i} - \mathbb{E} \xi\right\|_H \le 4\sqrt{2} \ln{\frac{2}{\delta}} \left(\frac{L}{n} + \frac{\sigma}{\sqrt{n}}\right).
   \end{displaymath}
\end{lemma}

\begin{lemma}[Cordes inequality]\label{cordes}
Let $A$ and $B$ be two positive bounded linear operators on a separable Hilbert space. Then we have
\begin{displaymath}
\left\|A^s B^s\right\| \leq\|A B\|^s, \quad \text { when } 0 \leq s \leq 1.
\end{displaymath}
\end{lemma}

The following lemma is a corollary of \citet[Lemma 5.8]{lin2018_OptimalRates}.
\begin{lemma}\label{lemma phi operator}
  Suppose that $A$ and $B$ are two positive self-adjoint operators on some Hilbert space, then 
  \begin{itemize}
      \item for $ r \in (0,1]$, we have
      \begin{displaymath}
        \left\|A^r-B^r\right\| \leq\|A-B\|^r.
      \end{displaymath}
      \item for $ r \ge 1 $, denote $c=\max (\|A\|,\|B\|)$, we have
      \begin{displaymath}
        \left\|A^r-B^r\right\| \leq r c^{r-1}\|A-B\|.
      \end{displaymath}
  \end{itemize}
\end{lemma}

\section*{Appendix C. Details of experiments}\label{appendix detail experiments}
First, we prove that the series in \eqref{series of sobolev} converges and $f^{*}(x)$ is continuous on $(0,1)$ for $0 < s < \frac{1}{\beta} = 0.5$. We begin with the computation of the sum of first $N$ terms of $\{ \sin 2 k \pi x + \cos 2 k \pi x \}$, note that 
\begin{align}
    &-2 \sin(\pi x) \left( \sin\left(2 \pi x \right) + \sin\left(4 \pi x\right) + \cdots + \sin\left(2 N \pi x\right) \right) \notag \\
    &= \left[ \cos\left(2 \pi + \pi \right)x - \cos\left(2 \pi - \pi \right)x \right] + \left[ \cos\left( 4 \pi + \pi \right)x - \cos\left(4 \pi - \pi\right)x \right] \notag \\
    &\quad \quad + \cdots + \left[ \cos\left( 2 N \pi + \pi \right)x - \cos\left(2 N \pi - \pi\right)x \right] \notag \\
    &= \cos\left( 2 N \pi + \pi \right)x - \cos \pi x. \notag
\end{align}
So we have 
\begin{align}\label{sin-1}
    \left| \left( \sin\left(2 \pi x \right) + \sin\left(4 \pi x\right) + \cdots + \sin\left(2 N \pi x\right) \right) \right| = \frac{\left| \cos\left( 2 N \pi + \pi \right)x - \cos \pi x \right|}{\left| 2 \sin(\pi x) \right|};
\end{align}
Similarly, we have
\begin{align}\label{cos-2}
    \left| \left( \cos\left(2 \pi x \right) + \cos\left(4 \pi x\right) + \cdots + \cos\left(2 N \pi x\right) \right) \right| = \frac{\left| \sin\left( 2 N \pi + \pi \right)x - \sin \pi x \right|}{\left| 2 \sin(\pi x) \right|}.
\end{align}
Note that \eqref{sin-1} and \eqref{cos-2} are uniformly bounded in $[\delta_{0}, 1-\delta_{0}]$ for any $\delta_{0} > 0$ and $N$. In addition, $\{ k ^{-(s + 0.5)} \}$ is monotone and decreases to zero. Use the Dirichlet criterion and we know that the series in \eqref{series of sobolev} is uniformly convergence in $[\delta_{0}, 1-\delta_{0}]$. Due to the arbitrariness of $\delta_{0}$, we know that the series converges and $f^{*}(x)$ is continuous on $(0,1)$.

Similarly, we can prove that the series in \eqref{series of min} converges and $f^{*}(x)$ is continuous on $(0,1)$ for $0 < s < \frac{1}{\beta} = 0.5$. 



In Figure \ref{figure appendix allc}, we present the results of different choices of $c$ for $\nu = c n^{\frac{\beta}{s \beta + 1}}$ in the experiment of Section \ref{section experiments}.

\begin{figure}[ht]
\vskip 0.05in
\centering
\subfigure[]{\includegraphics[width=0.3\columnwidth]{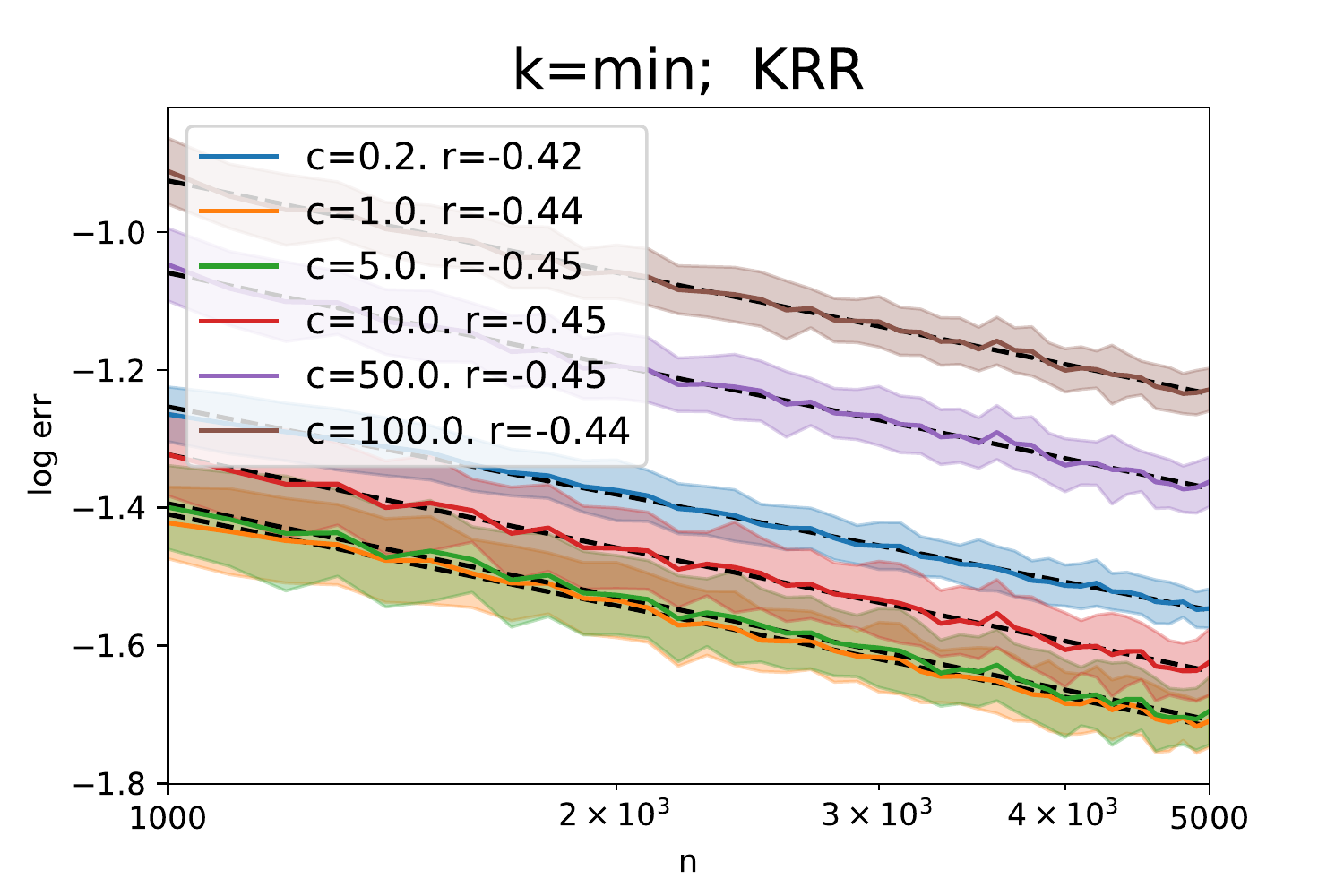}}
\subfigure[]{\includegraphics[width=0.3\columnwidth]{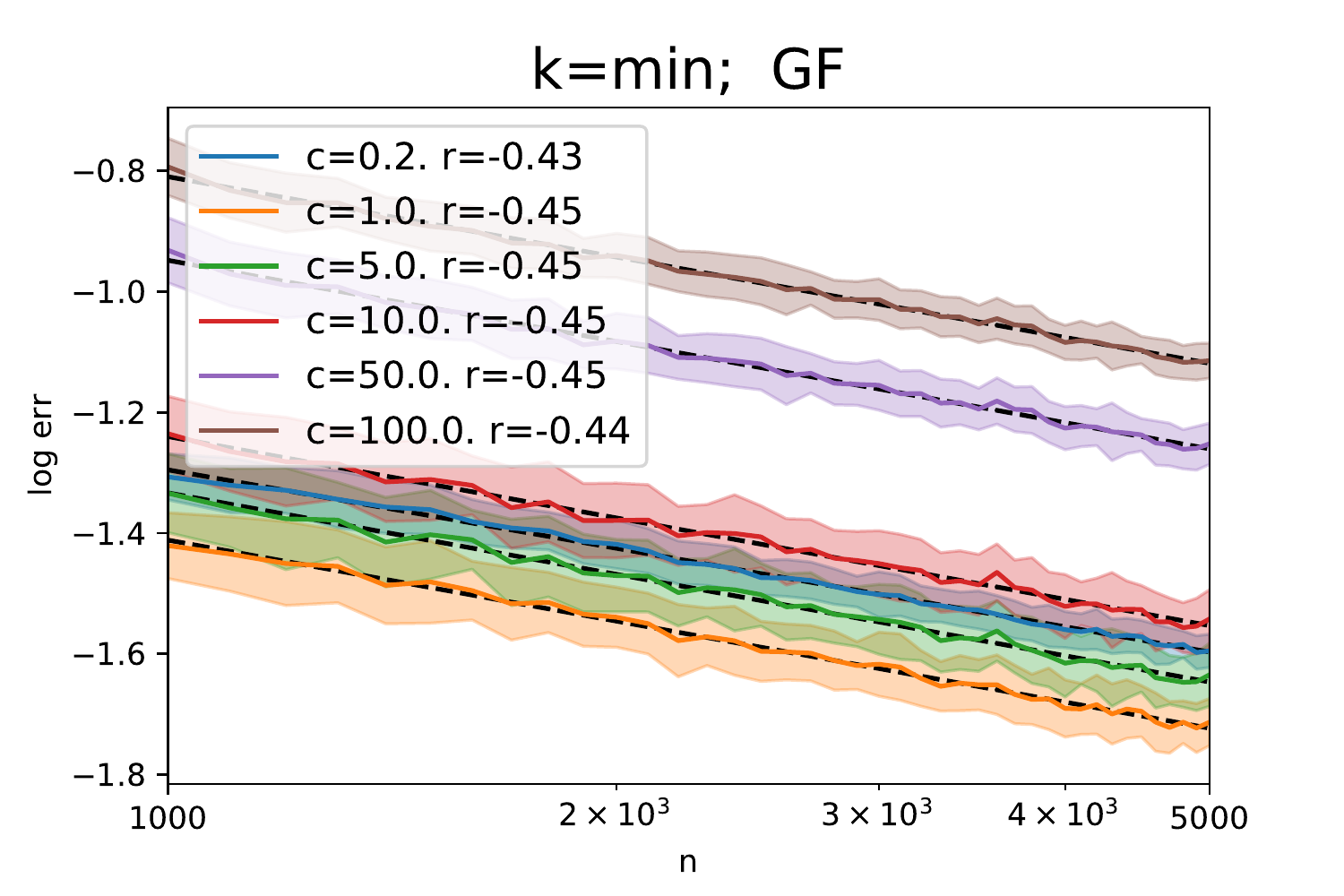}}
\subfigure[]{\includegraphics[width=0.3\columnwidth]{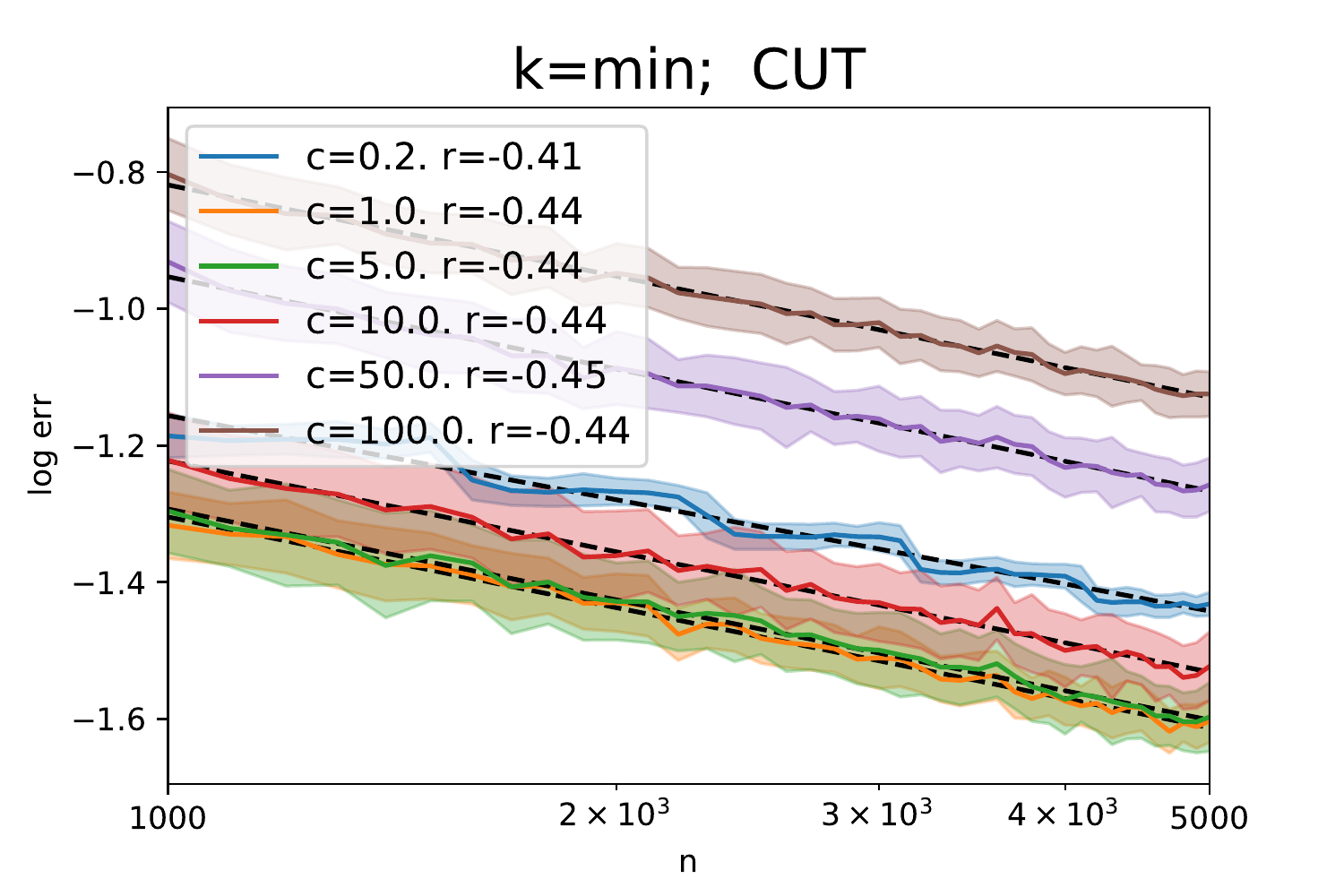}}

\subfigure[]{\includegraphics[width=0.3\columnwidth]{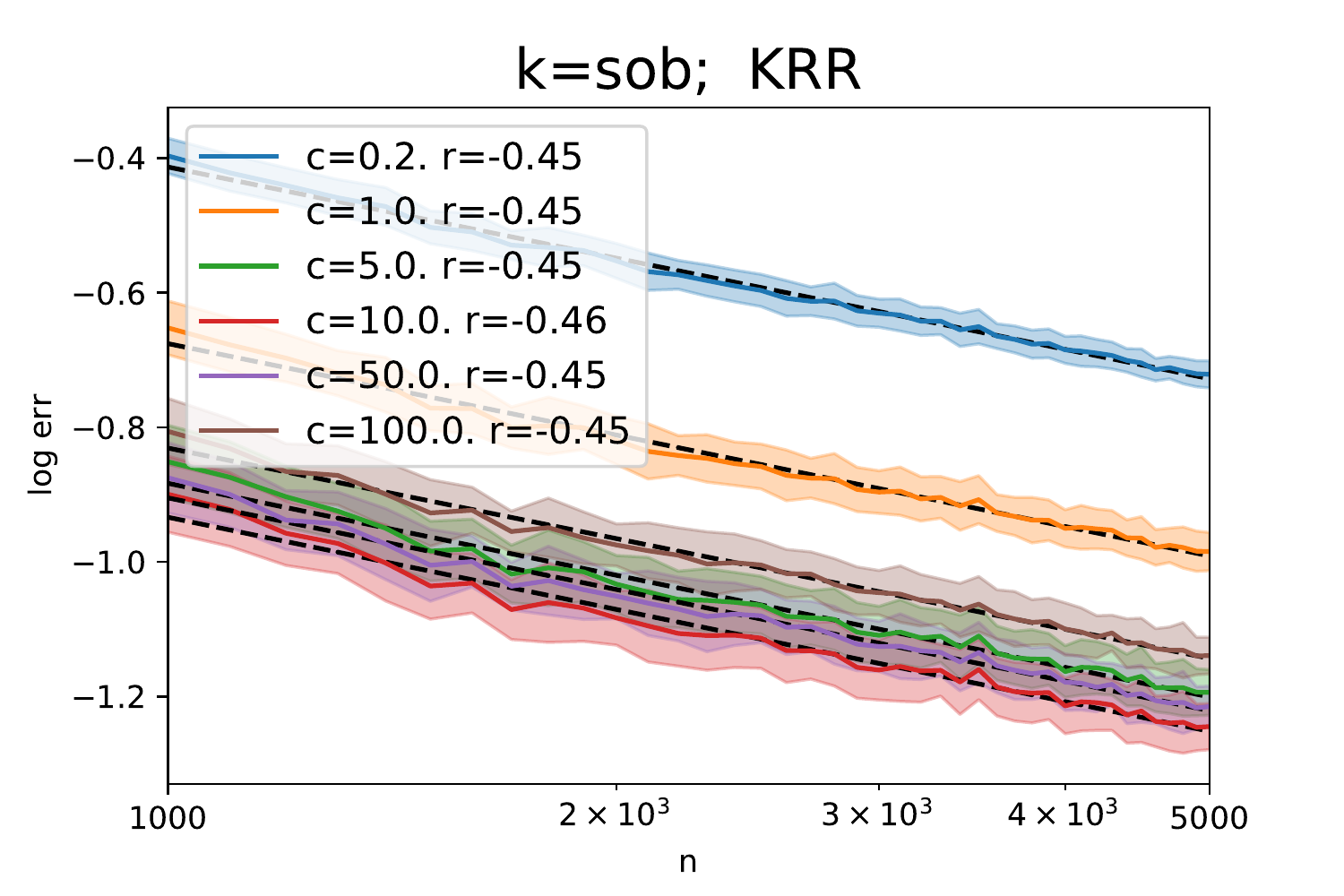}}
\subfigure[]{\includegraphics[width=0.3\columnwidth]{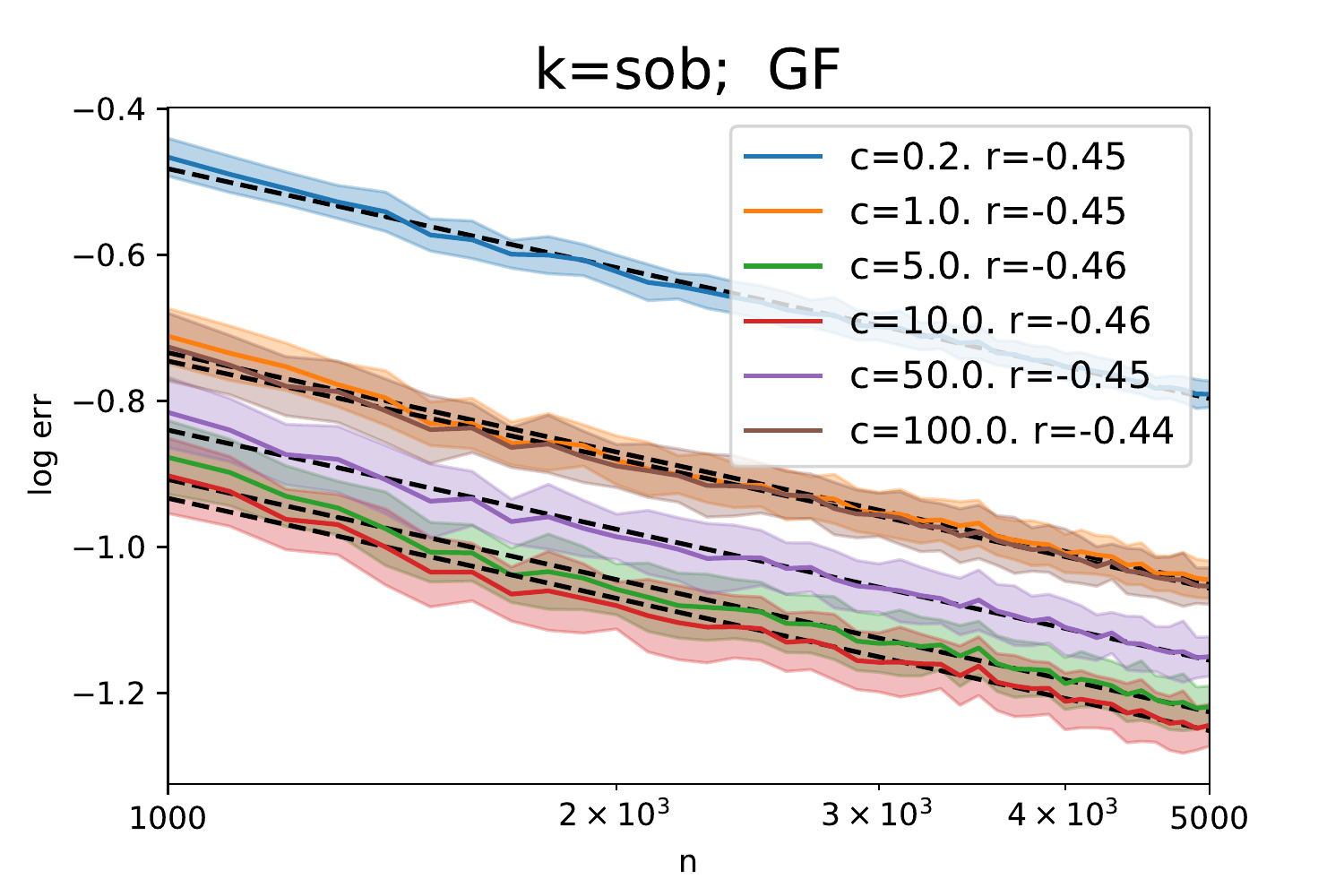}}
\subfigure[]{\includegraphics[width=0.3\columnwidth]{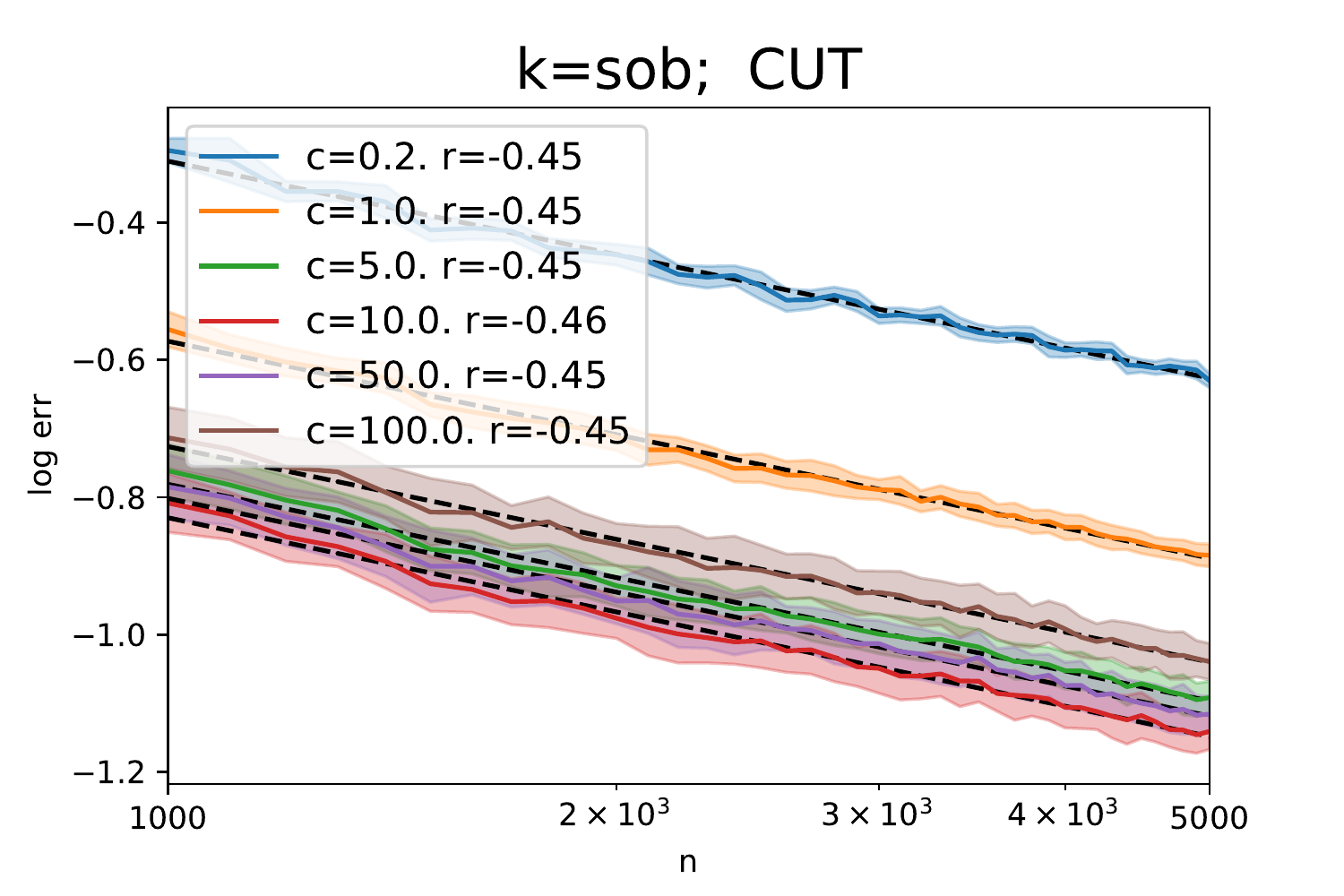}}
\caption{Error decay curves of two kinds of RKHSs and three kinds of spectral algorithms with different choices of $c$. Both axes are logarithmic.}
\label{figure appendix allc}
\vskip 0.05in
\end{figure}

\section*{Appendix D. A table of notations}\label{appendix notation table}



In this appendix, we provide a table of important notations frequently used throughout the main text.

\begin{table}[htbp]
\renewcommand{\arraystretch}{0.3}
\centering
\begin{tabular}{cc} 
\toprule
$\mathcal{X}, \mathcal{Y}$~     &  input space, output space  \\
\midrule
 $\{(x_{i},y_{i})\}_{i=1}^{n}$~    &  samples   \\ 

\midrule
 $\rho, \mu$~    &  distribution on $ \mathcal{X} \times \mathcal{Y}$, marginal distribution on $\mathcal{X}$   \\ 

\midrule
 $f_{\rho}^{*}$~    &  true function   \\ 
 
\midrule
 $\mathcal{H}$~    &  reproducing kernel Hilbert space (RKHS)   \\ 
 
\midrule
 $k(x,x^{\prime})$~    &  continuous kernel function   \\

\midrule
 $\kappa$~    &  upper bound of $k(x,x)$~ (Section \ref{section basic concept})   \\

\midrule
 $\tau$~    &  qualification of a filter function (spectral algorithm)   \\

\midrule
 $\beta$~    &  eigenvalue decay rate (EDR)   \\

\midrule
 $s$~    &  source condition   \\

\midrule
 $\alpha_{0}$~    &  embedding index   \\

\midrule
 $[\mathcal{H}]^{s}$~    &  interpolation space (power space)   \\

\midrule
 $L^{p}(\mathcal{X},\mu),~ L^{p}$~    &  $L^{p}$-space   \\
\midrule

 $S_{k}$~    &  natural embedding inclusion operator from $\mathcal{H}$ to $L^{2}$   \\
 
\midrule
 $S_{k}^{*}$~    &  adjoint operator of  $S_{k}$  \\

\midrule
 $L_{k}, T$~    &  integral operator given by $ S_{k} S_{k}^{*}, S_{k}^{*} S_{k}$  \\
  

\midrule
 $\text{tr}$, $\|\cdot\|_{1}$~    &  trace norm  \\
   
\midrule
 $\|\cdot\|_{2}$~    &  Hilbert-Schmidt norm  \\
    
\midrule
 $\|\cdot\|, ~\| \cdot \|_{\mathscr{B}(B_1,B_2)}$~    & operator norm  \\

\midrule
 $\{\lambda_{i}\}_{i=1}^{\infty},~\{e_{i}\}_{i=1}^{\infty}$~    & eigenvalues and eigenfunctions of $T$ (or $L_{k}, \mathcal{H}, k$) \\
     
      
\midrule
 $Z=\{(x_{i},y_{i})\}_{i=1}^{n}$~    & another expression of the samples \\
       
\midrule
 $K_{x}, K_{x}^{*}$~    & sampling operator and its adjoint operator (Section \ref{section spectral algorithms}) \\
       
\midrule
 $T_{x}$~    & $K_{x} K_{x}^{*} $ (Section \ref{section spectral algorithms}) \\
    
\midrule
 $T_{X}$~    & sample covariance operator (Section \ref{section spectral algorithms}) \\
     
\midrule
 $g_{Z}$~    & sample basis function (Section \ref{section spectral algorithms}) \\
 
\midrule
 $\nu$~    & regularization parameter \\
     
\midrule
 $\varphi_{\nu}(z), ~\psi_{\nu}(z)$~    & filter function, ~$ 1 - z \varphi(z)$ \\
     
      
\midrule
 $\hat{f}_{\nu}$~    & spectral algorithm estimator \\
      
\midrule
 $E, F_{\tau}$~    & constants in Definition \ref{def filter} \\
       
\midrule
 $M_{\alpha}$~    & the smallest constant $A>0$ such that \eqref{eq:EMB_Eigenvalues} is satisfied \\
       
\midrule
 $R$~    & constants in Assumption \ref{ass source condition} \\
       
\midrule
 $\sigma, L$~    & constants in Assumption \ref{ass mom of error} \\
        
\midrule
 $L^{p,q}(\mathcal{X},\mu),~L^{p,q}$~    & Lorentz space \\
         
\midrule
$T_{\nu},~T_{X\nu} $~    & $T+\nu^{-1},~T_{X}+\nu^{-1}$ ~(Section \ref{section some bounds}) \\
         
         
\midrule
$\mathcal{N}(\nu)$~    & $\text{tr}\big(T (T + \nu^{-1})^{-1}\big)$  \\
         
\midrule
$g$~    & $\mathbb{E} g_{Z} = S_{k}^{*} f_{\rho}^{*}$  \\
         
\midrule
$f_{\nu}$~    & $\varphi_{\nu}(T) g$  \\

\bottomrule
\end{tabular}
\caption{A table of important notations.}
\label{table2}
\end{table}

\vskip 0.2in
\bibliography{sample}

\begin{thebibliography}{50}
\providecommand{\natexlab}[1]{#1}
\providecommand{\url}[1]{\texttt{#1}}
\expandafter\ifx\csname urlstyle\endcsname\relax
  \providecommand{\doi}[1]{doi: #1}\else
  \providecommand{\doi}{doi: \begingroup \urlstyle{rm}\Url}\fi

\bibitem[Adams(1975)]{adams1975sobolev}
R.A. Adams.
\newblock \emph{Sobolev Spaces. Adams}.
\newblock Pure and applied mathematics. Academic Press, 1975.
\newblock URL \url{https://books.google.co.uk/books?id=JxzpSAAACAAJ}.

\bibitem[Adams and Fournier(2003)]{adams2003_SobolevSpaces}
Robert~A Adams and John~JF Fournier.
\newblock \emph{Sobolev Spaces}.
\newblock {Elsevier}, 2003.

\bibitem[Bach(2017)]{bach2017_BreakingCurse}
Francis Bach.
\newblock Breaking the curse of dimensionality with convex neural networks.
\newblock \emph{The Journal of Machine Learning Research}, 18\penalty0 (1):\penalty0 629--681, 2017.

\bibitem[Bauer et~al.(2007)Bauer, Pereverzyev, and Rosasco]{bauer2007_RegularizationAlgorithms}
F.~Bauer, S.~Pereverzyev, and L.~Rosasco.
\newblock On regularization algorithms in learning theory.
\newblock \emph{Journal of complexity}, 23\penalty0 (1):\penalty0 52--72, 2007.

\bibitem[Beaglehole et~al.(2022)Beaglehole, Belkin, and Pandit]{beaglehole2022_KernelRidgeless}
Daniel Beaglehole, Mikhail Belkin, and Parthe Pandit.
\newblock Kernel ridgeless regression is inconsistent in low dimensions, June 2022.

\bibitem[Blanchard and M{\"u}cke(2018)]{blanchard2018_OptimalRates}
G.~Blanchard and Nicole M{\"u}cke.
\newblock Optimal rates for regularization of statistical inverse learning problems.
\newblock \emph{Foundations of Computational Mathematics}, 18:\penalty0 971--1013, 2018.

\bibitem[Bordelon et~al.(2020)Bordelon, Canatar, and Pehlevan]{Bordelon2020SpectrumDL}
Blake Bordelon, Abdulkadir Canatar, and Cengiz Pehlevan.
\newblock Spectrum dependent learning curves in kernel regression and wide neural networks.
\newblock In \emph{ICML}, 2020.

\bibitem[Caponnetto(2006)]{caponnetto2006optimal}
Andrea Caponnetto.
\newblock Optimal rates for regularization operators in learning theory.
\newblock Technical report, MASSACHUSETTS INST OF TECH CAMBRIDGE COMPUTER SCIENCE AND ARTIFICIAL~…, 2006.

\bibitem[Caponnetto and de~Vito(2007)]{Caponnetto2007OptimalRF}
Andrea Caponnetto and Ernesto de~Vito.
\newblock Optimal rates for the regularized least-squares algorithm.
\newblock \emph{Foundations of Computational Mathematics}, 7:\penalty0 331--368, 2007.

\bibitem[Celisse and Wahl(2020)]{Celisse2020AnalyzingTD}
Alain Celisse and Martin Wahl.
\newblock Analyzing the discrepancy principle for kernelized spectral filter learning algorithms.
\newblock \emph{J. Mach. Learn. Res.}, 22:\penalty0 76:1--76:59, 2020.

\bibitem[Cho and Saul(2009)]{cho2009_KernelMethods}
Youngmin Cho and Lawrence Saul.
\newblock Kernel methods for deep learning.
\newblock In Y.~Bengio, D.~Schuurmans, J.~Lafferty, C.~Williams, and A.~Culotta, editors, \emph{Advances in Neural Information Processing Systems}, volume~22. {Curran Associates, Inc.}, 2009.

\bibitem[Cucker and Smale(2001)]{Cucker2001OnTM}
Felipe Cucker and Stephen Smale.
\newblock On the mathematical foundations of learning.
\newblock \emph{Bulletin of the American Mathematical Society}, 39:\penalty0 1--49, 2001.

\bibitem[Cui et~al.(2021)Cui, Loureiro, Krzakala, and Zdeborov'a]{Cui2021GeneralizationER}
Hugo Cui, Bruno Loureiro, Florent Krzakala, and Lenka Zdeborov'a.
\newblock Generalization error rates in kernel regression: The crossover from the noiseless to noisy regime.
\newblock In \emph{NeurIPS}, 2021.

\bibitem[Dai and Xu(2013)]{dai2013_ApproximationTheory}
Feng Dai and Yuan Xu.
\newblock \emph{Approximation Theory and Harmonic Analysis on Spheres and Balls}.
\newblock Springer {{Monographs}} in {{Mathematics}}. {Springer New York}, {New York, NY}, 2013.
\newblock ISBN 978-1-4614-6659-8 978-1-4614-6660-4.
\newblock \doi{10.1007/978-1-4614-6660-4}.

\bibitem[Dicker et~al.(2017)Dicker, Foster, and Hsu]{dicker2017_KernelRidge}
L.~Dicker, Dean~Phillips Foster, and Daniel~J. Hsu.
\newblock Kernel ridge vs. principal component regression: {{Minimax}} bounds and the qualification of regularization operators.
\newblock \emph{Electronic Journal of Statistics}, 11:\penalty0 1022--1047, 2017.

\bibitem[Dieuleveut and Bach(2016)]{dieuleveut2016nonparametric}
Aymeric Dieuleveut and Francis Bach.
\newblock Nonparametric stochastic approximation with large step-sizes1.
\newblock \emph{THE ANNALS}, 44\penalty0 (4):\penalty0 1363--1399, 2016.

\bibitem[Edmunds and Triebel(1996)]{edmunds_triebel_1996}
D.~E. Edmunds and H.~Triebel.
\newblock \emph{Function Spaces, Entropy Numbers, Differential Operators}.
\newblock Cambridge Tracts in Mathematics. Cambridge University Press, 1996.
\newblock \doi{10.1017/CBO9780511662201}.

\bibitem[Fischer and Steinwart(2020)]{fischer2020_SobolevNorm}
Simon-Raphael Fischer and Ingo Steinwart.
\newblock Sobolev norm learning rates for regularized least-squares algorithms.
\newblock \emph{Journal of Machine Learning Research}, 21:\penalty0 205:1--205:38, 2020.

\bibitem[Gerfo et~al.(2008)Gerfo, Rosasco, Odone, Vito, and Verri]{gerfo2008_SpectralAlgorithms}
L.~Lo Gerfo, Lorenzo Rosasco, Francesca Odone, E.~De Vito, and Alessandro Verri.
\newblock Spectral algorithms for supervised learning.
\newblock \emph{Neural Computation}, 20\penalty0 (7):\penalty0 1873--1897, 2008.

\bibitem[Guo et~al.(2017)Guo, Lin, and Zhou]{Guo2017LearningTO}
Zheng-Chu Guo, Shaobo Lin, and Ding-Xuan Zhou.
\newblock Learning theory of distributed spectral algorithms.
\newblock \emph{Inverse Problems}, 33, 2017.

\bibitem[Jacot et~al.(2018)Jacot, Gabriel, and Hongler]{jacot2018_NeuralTangent}
Arthur Jacot, Franck Gabriel, and Clement Hongler.
\newblock Neural tangent kernel: {{Convergence}} and generalization in neural networks.
\newblock In S.~Bengio, H.~Wallach, H.~Larochelle, K.~Grauman, N.~{Cesa-Bianchi}, and R.~Garnett, editors, \emph{Advances in Neural Information Processing Systems}, volume~31. {Curran Associates, Inc.}, 2018.

\bibitem[Jun et~al.(2019)Jun, Cutkosky, and Orabona]{jun2019kernel}
Kwang-Sung Jun, Ashok Cutkosky, and Francesco Orabona.
\newblock Kernel truncated randomized ridge regression: Optimal rates and low noise acceleration.
\newblock \emph{Advances in neural information processing systems}, 32, 2019.

\bibitem[Kohler and Krzyżak(2001)]{Kohler2001NonparametricRE}
Michael Kohler and Adam Krzyżak.
\newblock Nonparametric regression estimation using penalized least squares.
\newblock \emph{IEEE Trans. Inf. Theory}, 47:\penalty0 3054--3059, 2001.

\bibitem[Li et~al.(2022)Li, Meunier, Mollenhauer, and Gretton]{Li2022OptimalRF}
Zhu Li, Dimitri Meunier, Mattes Mollenhauer, and Arthur Gretton.
\newblock Optimal rates for regularized conditional mean embedding learning.
\newblock \emph{Advances in Neural Information Processing Systems}, 35:\penalty0 4433--4445, 2022.

\bibitem[Lin and Cevher(2020{\natexlab{a}})]{lin2020_OptimalConvergence}
Junhong Lin and Volkan Cevher.
\newblock Optimal convergence for distributed learning with stochastic gradient methods and spectral algorithms.
\newblock \emph{Journal of Machine Learning Research}, 21:\penalty0 147--1, 2020{\natexlab{a}}.

\bibitem[Lin and Cevher(2020{\natexlab{b}})]{lin2020convergences}
Junhong Lin and Volkan Cevher.
\newblock Convergences of regularized algorithms and stochastic gradient methods with random projections.
\newblock \emph{Journal of Machine Learning Research}, 21\penalty0 (20):\penalty0 1--44, 2020{\natexlab{b}}.

\bibitem[Lin and Rosasco(2017)]{lin2017optimal}
Junhong Lin and Lorenzo Rosasco.
\newblock Optimal rates for multi-pass stochastic gradient methods.
\newblock \emph{The Journal of Machine Learning Research}, 18\penalty0 (1):\penalty0 3375--3421, 2017.

\bibitem[Lin et~al.(2018)Lin, Rudi, Rosasco, and Cevher]{lin2018_OptimalRates}
Junhong Lin, Alessandro Rudi, L.~Rosasco, and V.~Cevher.
\newblock Optimal rates for spectral algorithms with least-squares regression over {{Hilbert}} spaces.
\newblock \emph{Applied and Computational Harmonic Analysis}, 48:\penalty0 868--890, 2018.

\bibitem[Lin et~al.(2017)Lin, Guo, and Zhou]{Lin2016DistributedLW}
Shao-Bo Lin, Xin Guo, and Ding-Xuan Zhou.
\newblock Distributed learning with regularized least squares.
\newblock \emph{Journal of Machine Learning Research}, 18\penalty0 (92):\penalty0 1--31, 2017.

\bibitem[Liu and Shi(2022)]{Liu2022StatisticalOO}
Jiading Liu and Lei Shi.
\newblock Statistical optimality of divide and conquer kernel-based functional linear regression.
\newblock \emph{ArXiv}, abs/2211.10968, 2022.

\bibitem[Mendelson and Neeman(2010)]{mendelson2010_RegularizationKernel}
Shahar Mendelson and Joseph Neeman.
\newblock Regularization in kernel learning.
\newblock \emph{The Annals of Statistics}, 38\penalty0 (1):\penalty0 526--565, February 2010.

\bibitem[M{\"u}cke and Blanchard(2018)]{Mcke2018ParallelizingSR}
Nicole M{\"u}cke and Gilles Blanchard.
\newblock Parallelizing spectrally regularized kernel algorithms.
\newblock \emph{J. Mach. Learn. Res.}, 19:\penalty0 30:1--30:29, 2018.

\bibitem[Pillaud-Vivien et~al.(2018)Pillaud-Vivien, Rudi, and Bach]{PillaudVivien2018StatisticalOO}
Loucas Pillaud-Vivien, Alessandro Rudi, and Francis Bach.
\newblock Statistical optimality of stochastic gradient descent on hard learning problems through multiple passes.
\newblock \emph{Advances in Neural Information Processing Systems}, 31, 2018.

\bibitem[Rastogi and Sampath(2017)]{rastogi2017_OptimalRates}
Abhishake Rastogi and Sivananthan Sampath.
\newblock Optimal rates for the regularized learning algorithms under general source condition.
\newblock \emph{Frontiers in Applied Mathematics and Statistics}, 3, 2017.

\bibitem[Rosasco et~al.(2005)Rosasco, De~Vito, and Verri]{rosasco2005_SpectralMethods}
Lorenzo Rosasco, Ernesto De~Vito, and Alessandro Verri.
\newblock Spectral methods for regularization in learning theory.
\newblock \emph{DISI, Universita degli Studi di Genova, Italy, Technical Report DISI-TR-05-18}, 2005.

\bibitem[Sawano(2018)]{sawano2018theory}
Yoshihiro Sawano.
\newblock \emph{Theory of Besov spaces}, volume~56.
\newblock Springer, 2018.

\bibitem[Smale and Zhou(2007)]{Smale2007LearningTE}
Stephen Smale and Ding-Xuan Zhou.
\newblock Learning theory estimates via integral operators and their approximations.
\newblock \emph{Constructive Approximation}, 26:\penalty0 153--172, 2007.

\bibitem[Smola et~al.(2000)Smola, Ov{\'a}ri, and Williamson]{smola2000_RegularizationDotproduct}
Alex Smola, Zolt{\'a}n Ov{\'a}ri, and Robert~C. Williamson.
\newblock Regularization with dot-product kernels.
\newblock \emph{Advances in neural information processing systems}, 13, 2000.

\bibitem[Spigler et~al.(2020)Spigler, Geiger, and Wyart]{Spigler2020AsymptoticLC}
Stefano Spigler, Mario Geiger, and Matthieu Wyart.
\newblock Asymptotic learning curves of kernel methods: empirical data versus teacher–student paradigm.
\newblock \emph{Journal of Statistical Mechanics: Theory and Experiment}, 2020, 2020.

\bibitem[Steinwart and Christmann(2008)]{Steinwart2008SupportVM}
Ingo Steinwart and Andreas Christmann.
\newblock Support vector machines.
\newblock In \emph{Information Science and Statistics}, 2008.

\bibitem[Steinwart and Scovel(2012)]{steinwart2012_MercerTheorem}
Ingo Steinwart and C.~Scovel.
\newblock Mercer's theorem on general domains: {{On}} the interaction between measures, kernels, and {{RKHSs}}.
\newblock \emph{Constructive Approximation}, 35\penalty0 (3):\penalty0 363--417, 2012.

\bibitem[Steinwart et~al.(2009)Steinwart, Hush, and Scovel]{steinwart2009_OptimalRates}
Ingo Steinwart, D.~Hush, and C.~Scovel.
\newblock Optimal rates for regularized least squares regression.
\newblock In \emph{{{COLT}}}, pages 79--93, 2009.

\bibitem[Talwai and Simchi-Levi(2022)]{Talwai2022OptimalLR}
Prem~M. Talwai and David Simchi-Levi.
\newblock Optimal learning rates for regularized least-squares with a fourier capacity condition.
\newblock 2022.

\bibitem[Tartar(2007)]{tartar2007introduction}
Luc Tartar.
\newblock \emph{An introduction to Sobolev spaces and interpolation spaces}, volume~3.
\newblock Springer Science \& Business Media, 2007.

\bibitem[Thomas-Agnan(1996)]{ThomasAgnan1996ComputingAF}
Christine Thomas-Agnan.
\newblock Computing a family of reproducing kernels for statistical applications.
\newblock \emph{Numerical Algorithms}, 13:\penalty0 21--32, 1996.

\bibitem[Tsybakov(2009)]{tsybakov2009_IntroductionNonparametric}
Alexandre~B. Tsybakov.
\newblock \emph{Introduction to Nonparametric Estimation}.
\newblock Springer Series in Statistics. {Springer}, {New York ; London}, 1st edition, 2009.

\bibitem[Wainwright(2019)]{wainwright2019_HighdimensionalStatistics}
Martin~J. Wainwright.
\newblock \emph{High-Dimensional Statistics: {{A}} Non-Asymptotic Viewpoint}.
\newblock Cambridge {{Series}} in {{Statistical}} and {{Probabilistic Mathematics}}. {Cambridge University Press}, 2019.

\bibitem[Wang and Jing(2022)]{JMLR:v23:21-0570}
Wenjia Wang and Bing-Yi Jing.
\newblock Gaussian process regression: Optimality, robustness, and relationship with kernel ridge regression.
\newblock \emph{Journal of Machine Learning Research}, 23\penalty0 (193):\penalty0 1--67, 2022.
\newblock URL \url{http://jmlr.org/papers/v23/21-0570.html}.

\bibitem[Yao et~al.(2007)Yao, Rosasco, and Caponnetto]{Yao2007OnES}
Y.~Yao, Lorenzo Rosasco, and Andrea Caponnetto.
\newblock On early stopping in gradient descent learning.
\newblock \emph{Constructive Approximation}, 26:\penalty0 289--315, 2007.

\bibitem[Zhang et~al.(2013)Zhang, Duchi, and Wainwright]{Zhang2013DivideAC}
Yuchen Zhang, John~C. Duchi, and Martin~J. Wainwright.
\newblock Divide and conquer kernel ridge regression: a distributed algorithm with minimax optimal rates.
\newblock \emph{J. Mach. Learn. Res.}, 16:\penalty0 3299--3340, 2013.

\end{thebibliography}

\end{document}